\newcommand{\pointsize}{11pt}

\documentclass[oneside, \pointsize]{amsbook}

\usepackage[
   includehead,
   includefoot,
     left = 1.5in, 
      top = 0.5in, 
    right = 1in,
   bottom = 1in
]{geometry}
\usepackage{fancyhdr}
\usepackage{setspace}
\usepackage{calc}

\fancyheadoffset[R]{0.5in} 

\fancypagestyle{prelim}{%
   \renewcommand{\headrulewidth}{0pt} 
   \fancyhf{}           
   \pagenumbering{roman}    
   \cfoot{-\thepage-}       
}

\fancypagestyle{maintext}{%
   \renewcommand{\headrulewidth}{0.4pt}
   \pagenumbering{arabic}
   \fancyhf{}
   \fancyhead[L]{\rightmark}
   \rhead{\thepage}
}

\numberwithin{figure}{chapter} 
\numberwithin{table}{chapter}
\numberwithin{equation}{chapter}
\numberwithin{section}{chapter}

\usepackage{graphicx}

\usepackage{enumerate}
\usepackage{verbatim}

\usepackage{amscd}     
\usepackage{amsxtra}
\usepackage{mathabx}
\usepackage{color}

 \usepackage{amsmath,amsthm, amsfonts,amssymb}
\usepackage[all]{xy}
\usepackage[mathscr]{euscript}

\usepackage[linktocpage,bookmarksopen,bookmarksnumbered,
                pdftitle={UC Davis Ph.D. Doctoral Thesis},
                pdfauthor={Department of Mathematics},%
                pdfsubject={UC Davis Ph.D. Doctoral Thesis},%
                pdfkeywords={UC Davis Ph.D. Doctoral Thesis}]{hyperref}

\theoremstyle{plain}
	\newtheorem{theorem}{Theorem}[section]
	\newtheorem{lemma}[theorem]{Lemma}
	\newtheorem{corollary}[theorem]{Corollary}
	\newtheorem{proposition}[theorem]{Proposition}

\theoremstyle{definition}
	\newtheorem{definition}[theorem]{Definition}
	\newtheorem{example}[theorem]{Example}

\theoremstyle{remark}
	\newtheorem{remark}[theorem]{Remark}

\newcommand\strongof[1]{{#1}^+}
\newcommand \partitionof[1]{\widetilde{#1}}
\newcommand\reverse[1]{{#1}^*}
\newcommand{\lex}{\geq_{\text{lex}}}

\newcommand{\Sym}{\textsl{Sym}}
\newcommand{\Qsym}{\textsl{QSym}}
\newcommand{\xx}{\mathbf{x}}

\newcommand{\rsqschur}{\mathcal{RS}}
\newcommand{\csqschur}{\mathcal{CS}}
\newcommand{\dzatom}{\mathcal{A}}

\newcommand{\kone}{\mathcal{K}_1}
\newcommand{\ktwo}{\mathcal{K}_2}
\newcommand{\colform}{\text{colform}}

\DeclareMathOperator{\cont}{cont}

\DeclareMathOperator{\arm}{arm}
\DeclareMathOperator{\coinv}{coinv}
\DeclareMathOperator{\dg}{dg}

\DeclareMathOperator{\leg}{leg}
\DeclareMathOperator{\maj}{maj}
\DeclareMathOperator{\Des}{Des}
\DeclareMathOperator{\Inv}{Inv}
\DeclareMathOperator{\inv}{inv}
\DeclareMathOperator{\cotrip}{Coinvtrip}
\DeclareMathOperator{\invtrip}{Invtrip}
\DeclareMathOperator{\Stab}{Stab}

\newcommand{\D}{\widehat{H_n}}

\newcommand{\augfill}{\widehat{\sigma}}

\newcommand{\Q}{\mathbb{Q}}
\newcommand{\C}{\mathbb{C}}
\newcommand{\Z}{\mathbb{Z}}
\newcommand{\N}{\mathbb{N}}
\newcommand{\R}{\mathbb{R}}

\definecolor{gr90}{gray}{0.90}

\definecolor{gr75}{gray}{0.75}

\definecolor{gyblue}{cmyk}{0,0.5,0,0}

\newsavebox{\myonesquare}
\savebox{\myonesquare}{\textcolor{gr75}{\rule{18pt}{18pt}}}
\newsavebox{\mytwosquare}
\savebox{\mytwosquare}{\textcolor{gr75}{\rule{36pt}{18pt}}}
\newsavebox{\mythreesquare}
\savebox{\mythreesquare}{\textcolor{gr75}{\rule{54pt}{18pt}}}
\newsavebox{\myfoursquare}
\savebox{\myfoursquare}{\textcolor{gr75}{\rule{72pt}{18pt}}}

\newsavebox{\othertwosquare}
\savebox{\othertwosquare}{\textcolor{gr75}{\rule{36pt}{18pt}}}

\newcommand{\sqone}{\usebox{\myonesquare}}
\newcommand{\sqtwo}{\usebox{\mytwosquare}}
\newcommand{\sqthree}{\usebox{\mythreesquare}}

\newsavebox{\Jmyonesquare}
\savebox{\Jmyonesquare}{\textcolor{gr75}{\rule{14pt}{14pt}}}
\newsavebox{\Jmytwosquare}
\savebox{\Jmytwosquare}{\textcolor{gr75}{\rule{28pt}{14pt}}}
\newsavebox{\Jmythreesquare}
\savebox{\Jmythreesquare}{\textcolor{gr75}{\rule{42pt}{14pt}}}
\newsavebox{\Jmyfoursquare}
\savebox{\Jmyfoursquare}{\textcolor{gr75}{\rule{56pt}{14pt}}}

\newsavebox{\Jothertwosquare}
\savebox{\Jothertwosquare}{\textcolor{gr75}{\rule{28pt}{14pt}}}

\newcommand{\Sqone}{\usebox{\Jmyonesquare}}
\newcommand{\Sqtwo}{\usebox{\Jmytwosquare}}
\newcommand{\Sqthree}{\usebox{\Jmythreesquare}}

\newlength{\cellsize} 
\newsavebox{\cell}

\sbox{\cell}{\begin{picture}(18,18) \put(0,0){\line(1,0){18}}
\put(0,0){\line(0,1){18}} \put(18,0){\line(0,1){18}}
\put(0,18){\line(1,0){18}}
\end{picture}}

\newcommand\cellify[1]{\def\thearg{#1}\def\nothing{}%
\ifx\thearg\nothing \vrule width0pt height\cellsize depth0pt\else
\hbox to 0pt{\usebox{\cell} \hss}\fi%
\vbox to \cellsize{ \vss \hbox to \cellsize{\hss$#1$\hss} \vss}}

\newcommand\tableau[1]{\vtop{\let\\\cr
\baselineskip -16000pt \lineskiplimit 16000pt \lineskip 0pt
\ialign{&\cellify{##}\cr#1\crcr}}}

\savebox3{%
\begin{picture}(25,25)
\put(0,0){\line(1,0){25}}
\put(0,0){\line(0,1){25}}
\put(25,0){\line(0,1){25}}
\put(0,25){\line(1,0){25}}
\end{picture}}
\newcommand\bigcellify[1]{\def\thearg{#1}\def\nothing{}%
\ifx\thearg\nothing
\vrule width0pt height\cellsize depth0pt\else
\hbox to 0pt{\usebox3\hss}\fi%
\vbox to 25\unitlength{
\vss
\hbox to 25\unitlength{\hss$#1$\hss}
\vss}}
\newcommand\Tableau[1]{\vtop{\let\\=\cr
\setlength\baselineskip{-16000pt}
\setlength\lineskiplimit{16000pt}
\setlength\lineskip{0pt}
\halign{&\bigcellify{##}\cr#1\crcr}}}

\savebox4{%
\begin{picture}(14,14)
\put(0,0){\line(1,0){14}}
\put(0,0){\line(0,1){14}}
\put(14,0){\line(0,1){14}}
\put(0,14){\line(1,0){14}}
\end{picture}}
\newcommand\smallcellify[1]{\def\thearg{#1}\def\nothing{}%
\ifx\thearg\nothing
\vrule width0pt height\cellsize depth0pt\else
\hbox to 0pt{\usebox4\hss}\fi%
\vbox to 14\unitlength{
\vss
\hbox to 14\unitlength{\hss$#1$\hss}
\vss}}
\newcommand\smalltableau[1]{\vtop{\let\\=\cr
\setlength\baselineskip{-16000pt}
\setlength\lineskiplimit{16000pt}
\setlength\lineskip{0pt}
\halign{&\smallcellify{##}\cr#1\crcr}}}

\newcommand\bas[1]{\omit \vbox to \cellsize{ \vss \hbox to \cellsize{\hss$#1$\hss} \vss}}

\newcommand\smallbas[1]{\omit \vbox to 14\unitlength{ \vss \hbox to 14\unitlength{\hss$#1$\hss} \vss}}

\makeatletter
\def\revddots{\mathinner{\mkern1mu\raise\p@
\vbox{\kern7\p@\hbox{.}}\mkern2mu
\raise4\p@\hbox{.}\mkern2mu\raise7\p@\hbox{.}\mkern1mu}}
\makeatother

\begin{document}
   \frontmatter

   \pagestyle{prelim}
   
   %
   \fancypagestyle{plain}{%
      \fancyhf{}
      \cfoot{-\thepage-}
   }%
\begin{center}
   \null\vfill
   \textbf{
   \begin{Large}
      Row-strict Quasisymmetric Schur Functions, Characterizations of Demazure Atoms, and Permuted Basement Nonsymmetric Macdonald Polynomials
      \end{Large}
   }
   \\
   \bigskip
   By \\
   \bigskip
   JEFFREY PAUL FERREIRA \\
   \bigskip
   B.S. (California State University, East Bay) 2006 \\
   M.A. (University of California, Davis) 2008 \\
   \bigskip
   DISSERTATION \\
   \bigskip
   Submitted in partial satisfaction of the requirements for the degree of \\
   \bigskip
   DOCTOR OF PHILOSOPHY \\
   \bigskip
   in \\
   \bigskip
   Mathematics \\
   \bigskip
   in the \\
   \bigskip
   OFFICE OF GRADUATE STUDIES \\
   \bigskip        
   of the \\
   \bigskip
   UNIVERSITY OF CALIFORNIA \\
   \bigskip
   DAVIS \\
   \bigskip
   Approved: \\
   \bigskip
   \bigskip
   \makebox[3in]{\hrulefill} \\
   Monica J. Vazirani \\
   \bigskip
   \bigskip
   \makebox[3in]{\hrulefill} \\
   Anne Schilling \\
   \bigskip
   \bigskip
   \makebox[3in]{\hrulefill} \\
   Jes\'{u}s De Loera\\
   \bigskip
   Committee in Charge \\
   \bigskip
   2011 \\
   \vfill
\end{center}

   \newpage
   
   %
   \doublespacing
   \tableofcontents
   \newpage
   
{\singlespacing
   \begin{flushright}
      Jeffrey Paul Ferreira \\
      December 2011 \\
      Mathematics \\
   \end{flushright}
}

\bigskip

\begin{center}
   Row-strict Quasisymmetric Schur Functions, Characterizations of Demazure Atoms, and Permuted Basement Nonsymmetric Macdonald Polynomials \\
\end{center}

\section*{Abstract}

We give a Littlewood-Richardson type rule for expanding the product of a row-strict quasisymmetric Schur function and a symmetric Schur function in terms of row-strict quasisymmetric Schur functions. This expansion follows from several new properties of an insertion algorithm defined by Mason and Remmel (2011) which inserts a positive integer into a row-strict composition tableau. We then apply this Littlewood-Richardson type rule to give a basis for the quotient of quasisymmetric functions by the ideal generated by symmetric functions with zero constant term.

We then discuss a family of polynomials called Demazure atoms.  We review the known characterizations of these polynomials and then present two new characterizations.  The first new characterization is a bijection between semi-standard augmented fillings and triangular arrays of nonnegative integers, which we call composition array patterns.  We also provide a bijection between composition array patterns with first row $\gamma$ and Gelfand-Tsetlin patterns whose first row is the partition $\lambda$ whose parts are the parts of $\gamma$ in weakly decreasing order.  The second new characterization shows that Demazure atoms are the polynomials obtained by summing the weights of all Lakshmibai-Seshadri paths which begin in a given direction.

Finally, we consider a family of polynomials called permuted basement nonsymmetric Macdonald polynomials which are obtained by permuting the basement of the combinatorial formula of Haglund, Haiman, and Loehr for nonsymmetric Macdonald polynomials.  We show that these permuted basement nonsymmetric Macdonald polynomials are the simultaneous eigenfunctions of a family of commuting operators in the double affine Hecke algebra.

   \newpage

   \section*{Acknowledgments}
   I first want to thank Monica Vazirani for the guidance she has given me over my graduate career.  I am incredibly grateful for her patient instruction which she always crafted to compliment my mathematical interests. Through years of meetings, emails, notes, and conversations she has put an incredibly positive effect on my research and helped me develop as a mathematician.  Also not to be forgotten are her words of encouragement, caring acknowledgment, and understanding of the times life's events made proving theorems challenging.

Thanks must also go to the other advisers who have made a lasting impact on my research.  Most notably I must thank Sarah Mason, who took considerable care instructing me on the results of her own research and allowed me to investigate particular avenues for myself.  Without Sarah's involvement in my graduate career and research I surely would be lacking much of the interesting combinatorics I now know and which are contained in this dissertation.  Thank you also to Jes\'{u}s De Loera for his tireless efforts at maintaining the UC Davis VIGRE grant, which I and the rest of the UC Davis math students benefited from.  Jes\'{u}s De Loera and Anne Schilling also provided excellent instruction in formal classes and informal seminars.  Over my years at UC Davis Fu Liu, Alex Woo, Brant Jones, Jason Bandlow, and Andrew Berget have all provided invaluable opportunities to learn mathematics.

Thank you to my peers at UC Davis, especially Tom Denton, Steve Pon, Chris Berg, Sonya Berg, Qiang Wang, and Robert Gysel who all entertained stimulating conversations with me over the years. Thanks also to my officemates Pat Dragon, Jason Hole, Sean O'Rourke, Katie O'Reilly, and Robert Gysel for being sounding boards for ideas ranging from analysis homework to research projects.  

The mathematics department staff at UC Davis are, by far, the most qualified individuals in their profession.  In particular I want to thank Celia Davis, Tina Denena, and Perry Gee for their efforts in helping me over the years.

Before entering graduate school, many people have had a significant impact on my mathematical schooling.  Among them are Carla Schick, Gerald Brody, Linda Henley, Wendy Struhl, Joseph Borzellino, Linda Patton, Todd Grundmeier, Dennis Eichhorn, and Russell and Karen Merris.

Thank you to my family, especially my mom and dad, for always encouraging me in my endeavors. Thank you to my brothers Matt and Dan and my sister Jenny for providing musical, culinary, zymological, literary, and political balance to my education. Finally, thank you to my wife Mariko for her support over my graduate career.

   \mainmatter
   
   \pagestyle{maintext}
   
   %
   \fancypagestyle{plain}{%
      \renewcommand{\headrulewidth}{0pt}
      \fancyhf{}
      \rhead{\thepage}
   }%
   
   \chapter{Introduction}
   \label{ch:IntroductionLabel}
  One of the most important bases of symmetric functions is the Schur basis $s_\lambda$, where $\lambda$ is a partition, which form an orthonormal $\mathbb{Z}$-basis for the algebra of symmetric functions. The Schur functions can be defined combinatorially as the generating functions of semi-standard Young tableaux. There are many classical results that demonstrate the importance of Schur functions.  For example they are characters of irreducible polynomial representations of $GL_n$ \cite{Fulton1997Young-tableaux} and the image of the irreducible characters of the symmetric group under the Frobenius characteristic map \cite{Stanley1999EC2}.

Two other bases for symmetric functions are the power sum symmetric functions $p_\lambda$ and the elementary symmetric functions $e_\lambda$.  Their relationships to Schur functions was studied in two classical results \cite{Stanley1999EC2}.  The Murnaghan-Nakayama rule expresses the product $s_\lambda p_r$ as a sum of Schur functions, and it is important to note that the coefficients appearing in the expansion are integers.  The Pieri rule expresses the product $s_\lambda e_n$ as a sum of Schur functions and even more noteworthy, the coefficients appearing in this expansion are non-negative integers.  These two results proved instrumental in determining the structure coefficients $c_{\lambda \mu}^\nu$, called Littlewood-Richardson coefficients, for the Schur basis, which can be defined by the equation
\begin{equation}\label{Schur}
s_\lambda s_\mu = \sum_{\nu} c_{\lambda \mu}^\nu s_\nu.
\end{equation}

The well-known Littlewood-Richardson rule gives a combinatorial description of the coefficients $c_{\lambda \mu}^\nu$ as the number of Littlewood-Richardson skew tableaux \cite{Fulton1997Young-tableaux}.  In addition to their definition (\ref{Schur}), the coefficients $c_{\lambda \mu}^\nu$ appear in representation theory as multiplicities of irreducible representations in certain induced representations of the symmetric group \cite{Sagan2001The-symmetric} and also as multiplicities of irreducible representations in the tensor product of two irreducible $GL_n$-modules \cite{Fulton1997Young-tableaux}. In geometry they appear as intersection numbers in the Schubert calculus on a Grassmannian \cite{Fulton1997Young-tableaux}. Note that each of these interpretations of Littlewood-Richardson coefficients give proofs that the coefficients $c_{\lambda \mu}^\nu$ are non-negative integers.

A more geometric combinatorial device used to study Schur functions are Gelfand-Tsetlin patterns, or GT patterns. These patterns are triangular arrays of integers with defining inequalities imposed upon the entries of the array \cite{Stanley1999EC2}. A fundamental fact concerning GT patterns is that these arrays are in bijection with semi-standard Young tableaux.  Originally defined in \cite{Gelfand-Finite} in connection with the study of irreducible representations of Lie algebras, GT patterns are now studied in a variety of contexts. GT patterns of shape $\lambda$ and weight $\mu$ form a polytope, called the Gelfand-Tsetlin polytope, with the property that the number of integral lattice points is the dimension of the $\mu$-weight space in the irreducible highest weight representation $V_\lambda$ of $\mathfrak{gl}_n(\C)$.  A study of the geometric properties of the GT polytope has been undertaken in a number of papers, including \cite{De-Loera2004Vertices-of}, \cite{Kiritchenko2010Gelfand}, and \cite{Kogan2005Toric}.  Because of the deep connection between representation theory and mathematical physics, it is not surprising that GT patterns also appear in the study of certain physical systems, for example see \cite{Strand1974Gelfand},  \cite{Temme1991The-structure}, and \cite{Warren2009Some-examples}.

GT patterns and the associated polytope also appear in the more general setting of the Berenstein-Zelevinsky polytope \cite{Berenstein1992Triple}. The BZ polytope associated to the triple $(\lambda, \mu, \nu)$ of partitions has the property that the number of integral lattice points is the Littlewood-Richardson coefficient $c_{\lambda, \mu}^\nu$.  In the new description of the BZ polytope given in \cite{Knutson1999The-honeycomb}, which the authors call the honeycomb model, the BZ polytope is presented in a way that is analogous to the GT polytope.  Specifically, the authors define so called rhombus inequalities on a triangular array of integers; see \cite{Buch2000The-saturation} for an exposition of these inequalities.  The honeycomb model given in \cite{Knutson1999The-honeycomb} was the main tool in the authors' proof of the Saturation Conjecture, which in our present context can be stated as $c_{\lambda, \mu}^\nu\neq 0$ if and only if $c_{n\lambda, n\mu}^{n\nu} \neq 0$ for some $n \in \Z_{>0}$.

The classical Littlewood-Richardson rule and the bijection of semi-standard Young tableaux with GT patterns are examples of two results on Schur functions that can be sought for generalizations of Schur functions. It is the purpose of this dissertation to present analogues of these results for various families of polynomials.   

\section{Quasisymmetric Functions}

Quasisymmetric functions are a natural generalization of symmetric functions and were defined by Gessel in \cite{Gessel_Multipartite}. Although Gessel discovered many properties of quasisymmetric functions, they had already appeared in earlier work of Stanley \cite{Stanley1972Ordered}.  Since their introduction, quasisymmetric functions have become of increasing importance.  They have appeared in such areas of mathematics as representation theory \cite{Hivert2000Hecke-algebras}, symmetric function theory \cite{Bergeron2000Noncommutative}, and combinatorial Hopf algebras \cite{Aguiar2006Combinatorial}.  Quasisymmetric functions also provide ample opportunity for combinatorial explorations of their properties, as seen in \cite{Bessenrodt2011Skew-quasischur}, \cite{Haglund2008Quasisymmetric}, and \cite{Haglund2009Refinements}.

In \cite{Haglund2008Quasisymmetric}, the authors define a new basis of the algebra $\Qsym$ of quasisymmetric functions called column-strict quasisymmetric Schur functions, denoted $\csqschur_\alpha$, where $\alpha$ is a sequence of positive integers called a strong composition. The functions $\csqschur_\alpha$ can be defined as generating functions for composition shaped tableaux, which are certain fillings with positive integers of strong composition shape $\alpha$ subject to three relations on the entries of the filling.  This definition parallels the combinatorial definition of Schur Functions, and in fact the Schur function $s_\lambda$ can be obtained by taking a certain sum of the functions $\csqschur_\alpha$. Continuing the parallel, the authors in \cite{Bessenrodt2011Skew-quasischur} give a more general definition of skew column-strict quasisymmetric Schur functions.

In \cite{Haglund2009Refinements} the authors give a Littlewood-Richardson type rule for expanding the product $\csqschur_\alpha s_\lambda$, where $s_\lambda$ is the symmetric Schur function, as a nonnegative integral sum of the functions $\csqschur_\beta$. The proof of the Littlewood-Richardson type rule in \cite{Haglund2009Refinements} utilizes an analogue of Schensted insertion on semi-standard Young tableaux, which is an algorithm in classical symmetric function theory which inserts a positive integer $b$ into a Young tableau $T$. This Littlewood-Richardson type rule was used in \cite{Lauve2010QSym} to show a certain subset of the functions $\csqschur_\alpha$ over a finite number of variables gives a basis of the coinvariant space for quasisymmetric polynomials, thus proving a conjecture of Bergeron and Reutenauer in \cite{Bergeron_Reutenauer}. 

In \cite{Mason2010A-Dual-Basis}, the authors provide a row-strict analogue of column-strict composition tableaux; specifically they interchange the roles of weak and strict in each of the three relations mentioned above.  See Definition \ref{def_RCT} below. One of these relations requires the fillings to decrease strictly across each row, thus the name row-strict composition tableaux. This definition produces the generating functions $\rsqschur_\alpha$, called row-strict quasisymmetric Schur functions, which the authors show are again a basis of $\Qsym$. Also contained in \cite{Mason2010A-Dual-Basis} is an insertion algorithm which inserts a positive integer $b$ into a row-strict composition tableau, producing a new row-strict composition tableau. This insertion procedure is presented in Definition \ref{RCT Insertion}.

In Section \ref{sec:algo} we establish several new properties of the insertion algorithm given in \cite{Mason2010A-Dual-Basis}. These properties lead directly to Theorem \ref{LRrule}, which is a Littlewood-Richardson type rule for expanding the product $\rsqschur_\alpha s_\lambda$ as a nonnegative integral sum of the function $\rsqschur_\beta$. Theorem \ref{LRrule} was inspired by \cite{Haglund2009Refinements} and the combinatorics of this rule share many similarities with the classical Littlewood-Richardson rule for multiplying two Schur functions, see \cite{Fulton1997Young-tableaux} for an example.  We then follow the work in \cite{Lauve2010QSym} and show in Corollary \ref{cor:coinvariant} that a certain subset of $\rsqschur_\alpha$ is a new basis for the coinvariant space for quasisymmetric functions.

\section{Demazure Atoms}

As mentioned above, the basis $\csqschur_\alpha$ of column-strict quasisymmetric Schur functions defined in \cite{Haglund2008Quasisymmetric} can be defined as generating functions for composition shaped tableaux.  Originally, the functions $\csqschur_\alpha$ were first defined as certain positive integral sums of functions called Demazure atoms. Demazure atoms first appeared in \cite{Lascoux1990Keys} under the name ``standard bases," and later were characterized as specializations of nonsymmetric Macdonald polynomials when $q=t=0$. The latter characterization can be deduced from the results in \cite{Sanderson2000On-the} for affine Lie type $A$ and \cite{Ion2003Nonsymmetric} for general affine Lie type; latter in \cite{Mason2009An-explicit} an explicit formulation was given for finite type $A$. Because the functions $\csqschur_\alpha$ are a positive integral sum of Demazure atoms, the Schur function $s_\lambda$ decompose as a positive integral sum of Demazure atoms. Demazure atoms are also related to Demazure characters and Schubert polynomials through the use of divided difference operators \cite{Reiner1995Key-polynomials}; see Section \ref{sec:divdiff}. The combinatorial construction of standard bases in \cite{Lascoux1990Keys} involved the computation of certain Young tableaux called left and right keys, which we give in Definition \ref{def:leftandrightkey}. The computation of keys and the polynomials related to them have seen uses in combinatorics and representation theory, for example  \cite{Aval2007Keys-and-alter}, \cite{Lenart2004A-unified}, \cite{Lenart2007On-the}, and \cite{Reiner1995Key-polynomials}.

Even with their characterizations given in \cite{Lascoux1990Keys} and \cite{Mason2009An-explicit}, little is known about the multiplicative structure of Demazure atoms. In an effort to advance the study of these objects, we present here two characterizations of Demazure atoms not previously used in the literature.  

The characterization given in \cite{Mason2009An-explicit} defines Demazure atoms as the generating functions for certain semi-standard augmented fillings.  Our first characterization of Demazure atoms, Theorem \ref{thm-CT-array}, is analogous to describing Schur functions as Gelfand-Tsetlin patterns.  Specifically, we present a bijection between certain triangular arrays, which we call composition array patterns, and semi-standard augmented fillings.  Then in Theorem \ref{thm:GTandCTbiject} we provide an explicit bijection between composition array patterns and GT patterns which completes a commutative diagram involving semi-standard augmented fillings, composition array patterns, Young tableaux, and GT patterns.

Our second characterization, Proposition \ref{prop:atomaspath},  presents Demazure atoms as certain Lakshmibai-Seshadri paths.  Lakshmibai-Seshadri paths, or LS-paths, form the foundation of Littelmann's path model, which is a combinatorial tool for computing multiplicities of a given weight in a highest weight representation of a symmetrizable Kac-Moody algebra \cite{Littelmann1995Paths-and-root} \cite{Littelmann1994A-Littlewood}.  Our characterization shows that Demazure atoms are the sum of the weights of all LS-paths beginning in a given direction.

\section{Nonsymmetric Macdonald Polynomials}   

Another family of functions which generalizes Schur functions are the nonsymmetric Macdonald polynomials $E_\gamma$, which are defined and studied in \cite{Opdam1995Harmonic}, \cite{Macdonald1996Affine-Hecke}, and \cite{Cherednik1995Nonsymmetric} among others. The $E_\gamma$ can be defined as eigenfunctions of certain operators $Y^\beta$ in the double affine Hecke algebra defined by Cherednik \cite{Cherednik1995Double-affine}, \cite{Cherednik1994Induced}, \cite{Cherednik2005Double-affine}.  

The $E_\gamma$ are a generalization of Schur functions in that through symmetrization of the $E_\gamma$ one recovers the symmetric Macdonald polynomials, which are a fundamental basis of the algebra of symmetric functions \cite{Macdonald1995Symmetric}. Upon different specializations of the parameters $q$ and $t$ appearing in symmetric Macdonald polynomials one can obtain Schur functions, Hall-Littlewood symmetric functions, Jack polynomials, monomial symmetric functions, and elementary symmetric functions. Many of these symmetric functions play key roles in representation theory, and in physics Jack polynomials are the eigenfunctions of the Schr\"{o}dinger operator for a quantum-mechanical system \cite{Kirillov1997Lectures-on}.

In \cite{Haglund2008A-combinatorial} the authors establish a combinatorial expression of nonsymmetric Macdonald polynomials $E_\gamma$ as a sum over combinatorial diagrams, where each term in the sum depends on certain statistics computed from the corresponding diagram. The proof of this combinatorial expansion relies on a recurrence developed by Knop \cite{Knop1997Integrality} and Sahi \cite{Sahi1996Interpolation} which is actually a special case of an intertwining formula developed by Cherednik \cite{Cherednik1997Intertwining}. More recently, the authors in \cite{Ram2011A-combinatorial} give combinatorial formulas for nonsymmetric Macdonald polynomials of arbitrary Lie type.

One avenue to pursue in seeking classical type results for nonsymmetric Macdonald polynomials is to seek Littlewood-Richardson type rules for multiplication.  In \cite{Baratta2009Pieri-type} the author develops Pieri type formulas for the expansion of the product of a nonsymmetric Macdonald polynomials with certain elementary symmetric functions in terms of nonsymmetric Macdonald polynomials.  In \cite{Yip2010A-Littlewood} the author gives a Littlewood-Richardson type rule for the expansion of the product of a nonsymmetric Macdonald Polynomials with a symmetric Macdonald polynomials in terms of symmetric, and also nonsymmetric, Macdonald polynomials.

Another avenue of research to pursue is to study the effects of specialization the parameters $q$ and $t$ appearing in formulas for $E_\gamma$.  This as been done by several authors.  In \cite{Sanderson2000On-the} the author establishes a connection between nonsymmetric Macdonald polynomials and the Demazure characters of $\widehat{\mathfrak{sl}}_n$. This result is then generalized in \cite{Ion2003Nonsymmetric} where the author establishes a similar connection, but for arbitrary type. In finite type $A$, as we previously mentioned, the author of \cite{Mason2009An-explicit} shows that specializing both parameters $q$ and $t$ of the polynomial $E_\gamma$ produces functions now called Demazure atoms.

Since the polynomials $E_\gamma$ are eigenfunctions for operators $Y^\beta$ in the double affine Hecke algebra, and can be constructed using intertwining formulas in the affine Hecke algebra, it is natural to ask how the functions $E_\gamma$ transform under the action of the Hecke algebra generator $T_i$. This question was answered in \cite{Haglund2010The-action}, and we present it here as Proposition \ref{prop:theaction}. For reasons that will become clear in Chapter \ref{ch:4ndChapterLabel}, the functions $T_\tau E_\gamma$, for arbitrary permutations $\tau$, are called permuted basement nonsymmetric functions. We apply Proposition \ref{prop:theaction} to explicitly describe the operators $Y_i^\tau$ which have the functions $T_\tau E_\gamma$ as their simultaneous eigenfunctions. This is done in Proposition \ref{prop:permbaseeigen}. We should note that one can also specialize the parameters $q$ and $t$ in $T_\tau E_\gamma$, and the resulting polynomials are the subject of interest in \cite{Haglund2011Properties}.

   \chapter[%
      Row-Strict Quasisymmetric Schur Functions
   ]{%
      Row-Strict Quasisymmetric Schur Functions
   }%
   \label{ch:2ndChapterLabel}
   In this chapter we discuss certain formal power series called row-strict quasisymmetric Schur functions.  Our first goal is to show that the product of a row-strict quasisymmetric Schur function and a symmetric Schur function decomposes into a positive sum of row-strict quasisymmetric Schur functions. This result is obtain in Theorem \ref{LRrule}. Our second goal is to show that the Littlewood-Richardson type rule of Theorem \ref{LRrule} gives us a way to construct a basis for the coinvariant space for quasisymmetric functions.

This chapter is organized as follows.  Section \ref{sec-definitions} reviews the definitions of symmetric and quasisymmetric functions, and also provides the definitions needed to define row-strict quasisymmetric Schur functions and to describe our Littlewood-Richardson type rule.  Section \ref{sec:algo} describes the insertion algorithm originally defined in \cite{Mason2010A-Dual-Basis}, and in this section several new properties of the algorithm are established.  Section \ref{sec:LRrule} states and proves the Littlewood-Richardson type rule, and in Section \ref{sec:coinv-space} we apply Theorem \ref{LRrule} to produce a basis for the coinvariant space for quasisymmetric functions.

\section{Definitions}{\label{sec-definitions}}

\subsection{Compositions and reverse lattice words}

A \emph{strong composition} with $k$ parts, denoted $\alpha=(\alpha_1, \ldots, \alpha_k)$, is a sequence of positive integers. A \emph{weak composition} $\gamma=(\gamma_1, \ldots, \gamma_k)$ is a sequence of nonnegative integers, and a \emph{partition} $\lambda=(\lambda_1, \ldots, \lambda_k)$ is a weakly decreasing sequence of nonnegative integers. Let  $\reverse{\lambda}:=(\lambda_k, \lambda_{k-1}, \ldots, \lambda_1)$ be the \emph{reverse of} $\lambda$, and let $\lambda^t$ denote the \emph{transpose of} $\lambda$. Denote by $\partitionof{\alpha}$ the unique partition obtained by placing the parts of $\alpha$ in weakly decreasing order.  Denote by $\strongof{\gamma}$ the unique strong composition obtained by removing the zero parts of $\gamma$. For any sequence $\beta=(\beta_1, \ldots, \beta_s)$ let $\ell(\beta):=s$ be the \emph{length of} $\beta$. For $\gamma$ and $\beta$ arbitrary (possibly weak) compositions of the same length $s$ we say $\gamma$ is \emph{contained in} $\beta$, denoted $\gamma \subseteq \beta$, if $\gamma_i \leq \beta_i$ for all $1 \leq i \leq s$.

For example $\alpha=(1,2,3,2)$ is a strong composition, $\gamma=(1,0,2,3,0,2)$ is a weak composition, and $\lambda=(3,2,2,1)$ is a partition.  Here, $\strongof{\gamma}=\alpha$ and $\partitionof{\alpha}=\lambda$. Also, $\reverse{\lambda}=(1,2,2,3)$ and $\lambda^t=(4,3,1)$.

A finite sequence $w=w_1w_2\cdots w_n$ of positive integers with largest part $m$ is called a \emph{reverse lattice word} if in every prefix of $w$ there are at least as many $i$'s as $(i-1)$'s for each $1< i \leq m$.  The \emph{content} of a word $w$ is the sequence $\cont(w)=(\cont(w)_1, \ldots, \cont(w)_m)$ where $\cont(w)_i$ equals the number of times $i$ appears in $w$.  A reverse lattice word is called \emph{regular} if $\cont(w)_1 \neq 0$.  Note that if $w$ is a regular reverse lattice word, then $\cont(w)=\reverse{\lambda}$ for some partition $\lambda$. For example $w=4433421$ is a regular reverse lattice word with largest part $4$, and $\cont(w)=(1,1,2,4)$.

\subsection{Symmetric and Quasisymmetric Functions}

Throughout this chapter, we will let $\xx :=(x_1, x_2, x_3, \ldots)$ be a countable set of indeterminates. A symmetric function is a formal power series $f(\xx)$ of bounded degree which is invariant under the action of the symmetric group on indices.  In this dissertation, we will work over the field of rational numbers $\Q$, but much of the theory discussed works equally as well over $\Z$. See  \cite{Stanley1999EC2} for many of the properties of symmetric functions.

\begin{definition}
A symmetric function $f(x_1, \ldots)$ is a formal power series of bounded degree with rational coefficients such that for each strong composition $\alpha=(\alpha_1, \ldots, \alpha_k)$, the coefficient of $x_1^{\alpha_1}\cdots x_k^{\alpha_k}$ is equal to the coefficient of $x_{i_1}^{\alpha_1}\cdots x_{i_k}^{\alpha_k}$ for all ordered sequences $(i_1,\ldots, i_k)$ of distinct positive integers. In other words, $f(x_1, x_2, \ldots)$ is symmetric if $f(x_{\tau(1)}, x_{\tau(2)}, \ldots) = f(x_1, x_2, \ldots)$ for any permutation $\tau$ of the positive integers.  
\end{definition}

\begin{example}
The function $f(\xx)=\sum_{i<j} x_ix_j$ is symmetric.
\end{example}

We will denote the $\Q$-algebra of symmetric functions by $\Sym$, however in the literature this algebra is usually denoted $\Lambda$.  As mentioned in the Chapter \ref{ch:IntroductionLabel}, symmetric functions are ubiquitous in mathematics.  More recently, the larger space of quasisymmetric functions have begun to play a similar role. See  \cite{Stanley1999EC2} for many of the properties of quasisymmetric functions.

\begin{definition}
A quasi-symmetric function $f(x_1, \ldots)$ is a formal power series of bounded degree with rational coefficients such that for each strong composition $\alpha $ $= (\alpha_1, $ $\ldots, \alpha_k)$, the coefficient of $x_1^{\alpha_1}\cdots x_k^{\alpha_k}$ is equal to the coefficient of $x_{i_1}^{\alpha_1}\cdots x_{i_k}^{\alpha_k}$ for all $i_1 < i_2 < \cdots < i_k$. 
\end{definition}

\begin{example}
The function $g(\xx)=\sum_{i < j} x_i^2x_j$ is quasisymmetric.
\end{example}

We will denote the $\Q$-algebra of quasisymmetric functions by $\Qsym$.  Notice that $\Sym \subseteq \Qsym$ since every symmetric function is also quasisymmetric.  In Section \ref{sec:coinv-space} we will be interested in the ideal generated by symmetric functions with zero constant term inside $\Qsym$.

\subsection{Diagrams and fillings}

To any sequence $\alpha$ of nonnegative integers we may associate a \emph{diagram}, also denoted $\alpha$, of left justified boxes with $\alpha_i$ boxes in the $i$th row from the top.  In the case $\alpha=\lambda$ is a partition, the diagram of $\lambda$ is the usual Ferrers diagram in the English convention.  Given a diagram $\alpha$, let $(i,j)$ denote the box in the $i$th row and $j$th column.

Given two sequences $\gamma$ and $\alpha$ of the same length $s$ such that $\gamma \subseteq \alpha$, define the \emph{skew diagram} $\alpha/\gamma$ to be the array of boxes that are in $\alpha$ and not in $\gamma$.  The boxes in $\gamma$ are called the \emph{skewed boxes.}  For each skew diagram in this chapter an extra column, called the \emph{$0$th column}, with $s$ boxes will be added strictly to the left of the first column.

A \emph{filling} $U$ of a diagram $\alpha$ is an assignment of positive integers to the boxes of $\alpha$.  Given a filling $U$ of $\alpha$, let $U(i,j)$ be the entry in the box $(i,j)$.  A \emph{reverse row-strict Young tableau}, or RRST, $T$ is a filling of partition shape $\lambda$ such that each row strictly decreases when read left to right and each column weakly decreases when read top to bottom.  Since we will only be concerned with RRST tableaux in this chapter, we will refer to these objects as \emph{tableaux} when no confusion will result.  If $\lambda$ is a partition with $\lambda_1=m$, then let $T_\lambda$ be the tableau of shape $\lambda$ which has the entire $i$th column filled with the entry $(m+1-i)$ for all $1\leq i \leq m$.

A filling $U$ of a skew diagram $\alpha/\gamma$ is an assignment of positive integers to the boxes that are in $\alpha$ and not in $\gamma$. We follow the convention that each box in the $0$th column and each skewed box is assigned a virtual $\infty$ symbol.  With this convention, an entry $U(i,j)$ may equal $\infty$. Given two boxes filled with $\infty$, if they are in the same row we define these entries to strictly decrease left to right, while two such boxes in the same column are defined to be equal. 

The \emph{column reading order} of a (possibly skew) diagram is the total order $<_{col}$ on its boxes where $(i,j) <_{col} (i^\prime, j^\prime)$ if $j<j^\prime$ or ($j=j^\prime$ and $i>i^\prime$).  This is the total order obtained by reading the boxes from bottom to top in each column, starting with the left-most column and working rightwards. If $\alpha$ is a diagram with $k$ rows and longest row length $m$, it will occasionally be convenient to define this order on all cells $(i,j)$, where $0\leq i \leq k$ and $1 \leq j \leq m+2$, regardless of whether the cell $(i,j)$ is a box in $\alpha$.  The \emph{column reading word} of a (possibly skew) filling $U$ is the sequence of integers $w_{col}(U)$ obtained by reading the entries of $U$ in column reading order, where we ignore entries from skewed boxes and entries in the $0$th column.

The following definition first appeared in \cite{Mason2010A-Dual-Basis}.

\begin{definition}{\label{def_RCT}}
Let $\alpha$ be a strong composition with $k$ parts and largest part size $m$.  A {\it{row-strict composition tableau}} (RCT) $U$ is a filling of the diagram $\alpha$ such that
\begin{enumerate}
\item The first column is weakly increasing when read top to bottom.
\item Each row strictly decreases when read left to right.
\item \textbf{Triple Rule:} Supplement $U$ with zeros added to the end of each row so that the resulting filling $\hat{U}$ is of rectangular shape $k \times m$.  Then for $1 \leq i_1 < i_2 \leq k$ and $2 \leq j \leq m$,
\begin{equation*}
\left( \hat{U}(i_2,j) \neq 0 \text{ and } \hat{U}(i_2,j)>\hat{U}(i_1,j) \right)\Rightarrow \hat{U}(i_2,j) \geq \hat{U}(i_1, j-1).
\end{equation*}
\end{enumerate}
\end{definition}

If we let $\hat{U}(i_2,j)=b$, $\hat{U}(i_1,j)=a$, and $\hat{U}(i_1, j-1)=c$, then the Triple Rule ($b\neq 0$ and $b>a$ implies $b \geq c$) can be pictured as
\[
\begin{array}{ccc} \vspace{6pt}
\tableau{   c & a   \\  & \bas{{\vdots}} }  
 \\
 \tableau{ & b}
\end{array}.
\]

In addition to the triples that satisfy Definition \ref{def_RCT}, we also have a notion of inversion triples. Inversion triples were originally introduced by Haglund, Haiman, and Loehr in \cite{Haglund2005A-combinatorial} and \cite{Haglund2008A-combinatorial} to describe a combinatorial formula for symmetric, and later nonsymmetric, Macdonald polynomials.  In the present context inversion triples are defined as follows. Let $\gamma$ be a (possibly weak) composition and let $\beta$ be a strong composition with $\gamma \subseteq \beta$.  Let $U$ be some arbitrary filling of $\beta/\gamma$.  A \emph{Type A triple} is a triple of entries
\begin{displaymath}
U(i_1, j-1)=c, \; U(i_1, j)=a, \; U(i_2, j)=b
\end{displaymath}
in $U$ with $\beta_{i_1} \geq \beta_{i_2}$ for some rows $i_1 < i_2$ and some column $j> 0$.  A \emph{Type B triple} is a triple of entries
\begin{displaymath}
U(i_1, j)=b, \; U(i_2, j)=c, \; U(i_2, j+1)=a
\end{displaymath}
in $U$ with $\beta_{i_1}<\beta_{i_2}$ for some rows $i_1<i_2$ and some column $j\geq 0$. A triple of either type A or B is said to be an \emph{inversion triple} if either $b \leq a < c$ or $a < c \leq b$.  Note that triples of either type may involve boxes from $\gamma$ or boxes in the $0$th column. Type A and Type B triples can be visualized as
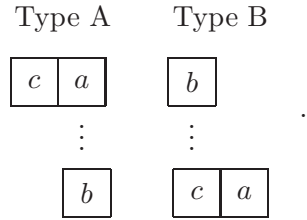
\[
\begin{array}{cc} \vspace{6pt}
\text{Type A} & \text{Type B} \\

\begin{array}{ccc} \vspace{6pt}
\tableau{   c & a   \\  & \bas{{\vdots}} }  
 \\
 \tableau{ & b}
\end{array}
&
\begin{array}{ccc} \vspace{6pt}
\tableau{   b &    \\  \bas{{\vdots}} & }  
 \\
 \tableau{c & a}
\end{array}

\end{array}.
\]

Central to the main theorem of this paper is the following definition.

\begin{definition}{\label{LRskewCT}}
Let $\beta$ and $\alpha$ be strong compositions. Let $\gamma$ be some (possibly weak) composition satisfying $\strongof{\gamma}=\alpha$ and $\gamma \subseteq \beta$. A {\it{Littlewood-Richardson skew row-strict composition tableau}} $S$, or {\it{LR skew RCT}}, of shape $\beta/\alpha$ is a filling of a diagram of skew shape $\beta/\gamma$ such that: 
\begin{enumerate}
\item Each row strictly decreases when read left to right.
\item Every Type A and Type B triple is an inversion triple.
\item The column reading word of $S$, $w_{col}(S)$, is a regular reverse lattice word.
\end{enumerate}
\end{definition}

Note that in Definition \ref{LRskewCT} the shape of an LR skew RCT is $\beta/\alpha$ although we refer to a filling of $\beta/\gamma$.

\begin{example}{\label{UandS}}
Below is a RCT, $U$, of shape $(1,3,2,2)$, and a LR skew RCT, $S$, of shape $(1,2,3,1,5,3)/(1,3,2,2)$ with $w_{col}(S)=4433421$.

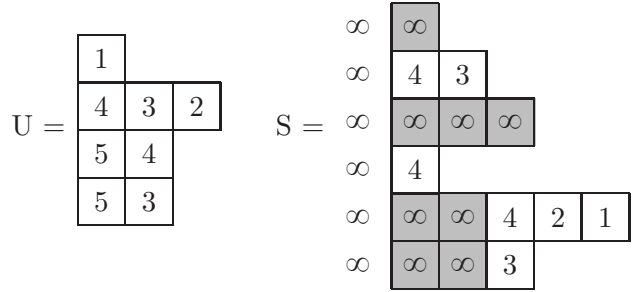
\begin{picture}(100,100)(0,0)
\put(100,40){U =}
\put(125,60){\tableau{ 1 \\
	      4 & 3 & 2 \\
	      5 & 4 \\
	      5 & 3
	      }}
\put(200, 40){S =}
\put(244,72){\sqone}
\put(244,36){\sqthree}
\put(244,0){\sqtwo}
\put(244,-18){\sqtwo}
\put(222,72){\tableau{ \bas{\infty} &\infty \\
    \bas{\infty} & 4& 3  \\
     \bas{\infty} & \infty & \infty & \infty  \\ 
    \bas{\infty} & 4  \\
    \bas{\infty} & \infty & \infty &4&2&1 \\ 
    \bas{\infty} & \infty & \infty  & 3
}}
\end{picture}
\end{example}

\subsection{Generating functons}

The content of any filling $U$ of partition or composition shape, denoted $\cont(U)$, is the content of its column reading word $w_{col}(U)$.  To any filling $U$ we may associate a monomial 
\begin{equation*}
\xx^U=\prod_{i \geq 1} x_i^{\cont(U)_i}.
\end{equation*}  

The algebra of symmetric functions $\Sym$ has the Schur functions $s_\lambda$ as a basis, where $\lambda$ ranges over all partitions.  The Schur function $s_\lambda$ can be defined in a number of ways.  In this dissertation it is advantageous to define $s_\lambda$ as the generating function of reverse row-strict tableaux of shape $\lambda^t$.  That is

\begin{displaymath}
s_\lambda = \sum \xx^T
\end{displaymath}
where the sum is over all reverse row-strict tableaux $T$ of shape $\lambda^t$. See \cite{Stanley1999EC2} for many of the properties of $s_\lambda$.

The generating function of row-strict composition tableaux of shape $\alpha$ are denoted $\rsqschur_\alpha$.  That is
\begin{displaymath}
\rsqschur_\alpha=\sum \xx^U
\end{displaymath}
where the sum is over all row-strict composition tableaux $U$ of shape $\alpha$.  The generating functions $\rsqschur_\alpha$ are called \emph{row-strict quasisymmetric Schur functions} and were originally defined in \cite{Mason2010A-Dual-Basis}.  In \cite{Mason2010A-Dual-Basis} the authors show $\rsqschur_\alpha$ are indeed quasisymmetric, and furthermore the collection of all $\rsqschur_\alpha$, as $\alpha$ ranges over all strong compositions, forms a basis of the algebra $\Qsym$ of quasisymmetric functions. The authors also show that the Schur function $s_\lambda$ decomposes into a positive sum of row-strict quasisymmetric Schur functions indexed by compositions that rearrange the transpose of $\lambda$.  Specifically,
\begin{equation*}
s_\lambda=\sum_{\partitionof{\alpha}=\lambda^t} \rsqschur_\alpha .
\end{equation*}

\section{Insertion algorithms}
\label{sec:algo}

Define a \emph{two-line array} $A$ by letting
\begin{equation*}
A=\left( \begin{matrix} i_1 & i_2 & \cdots & i_n \\
				   j_1 & j_2 & \cdots & j_n 	 \end{matrix} \right)
\end{equation*}
where $i_r, j_r$ are positive integers for $1 \leq r \leq n$, (a) $i_1 \geq i_2 \geq \cdots \geq i_n$, and (b) if $i_r=i_s$ and $r\leq s$ then $j_r \leq j_s$. Denote by $\widehat{A}$ the upper sequence $i_1, i_2, \ldots, i_n$ and denote by $\widecheck{A}$ the lower sequence $j_1, j_2, \ldots, j_n$.

The classical Robinson-Schensted-Knuth (RSK) correspondence gives a bijection between two-line arrays $A$ and pairs of (reverse row-strict) tableaux $(P,Q)$ of the same shape \cite{Fulton1997Young-tableaux}.  The basic operation of RSK is Schensted insertion on tableaux, which is an algorithm that inserts a positive integer into a tableau $T$ to produce a new tableau $T^\prime$.  In our setting, Schensted insertion is stated as

\begin{definition}
Given a tableau $T$ and $b$ a positive integer one can obtain $T^\prime :=b \rightarrow T$ by inserting $b$ as follows:
\begin{enumerate}
\item Let $\tilde{b}$ be the largest entry less than or equal to $b$ in the first row of $T$.  If no such $\tilde{b}$ exists, simply place $b$ at the end of the first row.
\item If $\tilde{b}$ does exist, replace (bump) $\tilde{b}$ with $b$ and proceed to insert $\tilde{b}$ into the second row using the method just described.
\end{enumerate}
\end{definition}

The RSK correspondence is the bijection obtained by inserting $\widecheck{A}$ in the empty tableau $\emptyset$ to obtain a tableau $P$ called the \emph{insertion tableau}, while simultaneously placing $\widehat{A}$ in the corresponding new boxes to obtain a tableau $Q$ called the \emph{recording tableau}.

\begin{example} Below is an example of the RSK correspondence on pairs of reverse row-strict Young tableaux.
\[
\begin{array}{c}
\begin{pmatrix}
\; \; \tableau{  4 & 3 & 2 & 1 \\ 4 & 3 \\ 2 }  & , &
\tableau{ 4 & 3 & 2 & 1 \\ 4 & 3 \\ 4 } \; \;
\end{pmatrix} 
\qquad \stackrel{RSK}{\Leftrightarrow} \qquad
\begin{pmatrix}
4 & 4 & 4 & 3 & 3 & 2 & 1 \\
2 & 4 & 4 & 3 & 3 & 2 & 1
\end{pmatrix} 
\end{array}
\]
\end{example}

The authors in \cite{Mason2010A-Dual-Basis} provide an analogous insertion algorithm on row-strict composition tableaux.

\begin{definition}{\label{RCT Insertion}}{({\it{RCT Insertion}})}
Let $U$ be a RCT with longest row of length $m$, and let $b$ be a positive integer.  One can obtain $U^\prime :=U \leftarrow b$ by inserting $b$ as follows. Scan the entries of $U$ in reverse column reading order, that is top to bottom in each column starting with the right-most column and working leftwards, starting with column $m+1$ subject to the conditions:
\begin{enumerate}
\item In column $m+1$, if the current position is at the end of a row of length $m$, and $b$ is strictly less than the last entry in that row, then place $b$ in this empty position and stop.  If no such position is found, continue scanning at the top of column $m$.  
\item 
\begin{enumerate}
\item Inductively, suppose some entry $b_j$ begins scanning at the top of column $j$.  In column $j$, if the current position is empty and is at the end of a row of length $j-1$, and $b_j$ is strictly less than the last entry in that row, then place $b_j$ in this empty position and stop.  
\item If a position in column $j$ is nonempty and contains $\tilde{b}_j \leq b_j$ such that $b_j$ is strictly less than the entry immediately to the left of $\tilde{b}_j$, then $b_j$ bumps $\tilde{b}_j$.  Continue scanning column $j$ with the entry $\tilde{b}_j$, bumping whenever possible using the criterion just described. After scanning the last entry in column $j$, begin scanning column $j-1$.
\end{enumerate}
\item If an entry $b_1$ is bumped into the first column, then place $b_1$ in a new row that appears after the lowest entry in the first column that is weakly less than $b_1$.
\end{enumerate}
\end{definition}

In \cite{Mason2010A-Dual-Basis} the authors show $U^\prime=U \leftarrow b$ is a row-strict composition tableau.  The algorithm of inserting $b$ into $U$ determines a set of boxes in $U^\prime$ called the \emph{insertion path of $b$} and denoted $I(b)$, which is the set of boxes in $U^\prime$ which contain an entry bumped during the algorithm.  Not that if some entry $b_j$ bumps an entry $\tilde{b}_j$ then $b_j \geq \tilde{b}_j$; thus the sequence of entries bumped during the algorithm is weakly decreasing. We call the row in $U^\prime$ in which the new box is ultimately added the \emph{row augmented by insertion.} If the new box has coordinates $(i,1)$, then for each $r>i$, row $r$ of $U^\prime$ is said to be the \emph{corresponding row} of row $(r-1)$ of $U$. 

\begin{example} The figure below gives an example of the RCT insertion algorithm, where row $4$ is the row augmented by insertion.  The italicized entries indicated the insertion path $I(4)$, which goes from top to bottom in each column and from right to left.

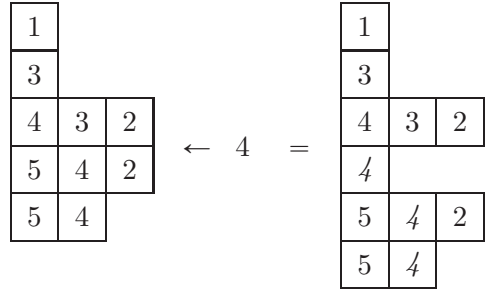
\begin{picture}(100,110)(0,-20)
\put(125,60){\tableau{ 1 \\
	       3 \\
	      4 & 3 & 2 \\
	      5 & 4 & 2\\
	      5 & 4
	      }}
\put(190, 20){$\leftarrow$}
\put(210, 20){$4$}
\put(230, 20){$=$}

\put(250,60){\tableau{
	1 \\
	3 \\
	4 & 3 & 2 \\
	{\textbf{\it{4}}} \\
	5 & {\textbf{\it{4}}} & 2 \\
	5 & {\textbf{\it{4}}} }}

\end{picture}

\end{example}

We establish several new lemmas concerning RCT insertion that are instrumental in proving the main theorem of this paper in Section \ref{sec:LRrule}.

\begin{lemma}{\label{rows}}
Let $U$ be a RCT and let $b$ be a positive integer.  Then each row of $U^\prime = U \leftarrow b$ contains at most one box from $I(b)$.
\end{lemma}

\begin{proof}
Suppose for a contradiction that some row $i$ in $U^\prime$ contains at least two boxes from the insertion path of $b$.  Consider two of these boxes, say in columns $j$ and $j^\prime$ such that (without loss of generality) $j^\prime<j$.  Let $U^\prime(i,j)=b_1$ and $U^\prime(i, j^\prime)=b_2$.  Since $b_1$ was bumped earlier in the algorithm than $b_2$, we must have $b_2 \leq b_1$. Since $b_1$ and $b_2$ are in the same row, and $b_2$ appears to the left of $b_1$,  this contradicts row-strictness of $U^\prime$.
\end{proof}

\begin{lemma}{\label{lastrow}} Let $U$ be a RCT and let $b$ be a positive integer.  Let $U^\prime = U \leftarrow b$ with row $i$ of $U^\prime$ the row augmented by insertion.  Then for all rows $r >i$ of $U^\prime$, the length of row $r$ is not equal to the length of row $i$.
\end{lemma}

\begin{proof}
Suppose for a contradiction that this is not the case.  Then there exists a row $r$ of $U^\prime$, $r>i$, whose length is equal to the length of row $i$.  Call this length $j$.  Since $i$ is the row in which the new cell was added then row $r$ in $U^\prime$ is the same as row $r$ in $U$, except in the case when the augmented row $i$ is of length $1$ in which case row $(r+1)$ of $U^\prime$ is the same as row $r$ of $U$.  Let $y$ be the entry that scans the top of the $j$th column.

We claim $y \geq U(r,j)$.  Suppose not.  Then $y<U(r,j)$.  When scanning column $j+1$ if the value $y$ was in hand at row $r$, we would have put $y$ in a new box with coordinates $(r, j+1)$.  Since this is not the case, $y$ was bumped from position $(s, j+1)$, $s>r$.  In this case $y=\hat{U}(s, j+1) >0=\hat{U}(r, j+1)$ with $y<\hat{U}(r,j)$.  This is a Triple Rule violation in $U$, thus $y\geq U(r,j)$.

If $j=1$ then since $y \geq U(r,j)$, $y$ would be inserted into a new row $i$ where $i>r$.  This is contrary to our assumption that the augmented row $i$ satisfies $r>i$. So we can assume $j>1$.

We must have $U(r,j)=U^\prime(r,j) \geq U(i, j-1)=U^\prime(i,j-1)$, or else $U$ would have a Triple Rule violation because $\hat{U}(i,j)=0$.  Since $U(i, j-1)=U^\prime(i, j-1)>U^\prime(i,j)$ we have $U^\prime(r,j) > U^\prime(i,j)$.

Consider now the portion of the insertion path in column $j$, say in rows $i_0 < i_1 < $ $\ldots <i_t = i$, where $y=U^\prime(i_0, j)$.  Since $y\geq U(r,j)=U^\prime(r,j) > U^\prime(i,j)$ and since the entries in the insertion path are weakly decreasing, there is some index $\ell$, $0\leq \ell < t$, such that 
\begin{equation}{\label{star1}}
U^\prime(i_\ell, j)\geq U^\prime(r,j) > U^\prime(i_{\ell+1}, j).
\end{equation}
Since rows strictly decrease, 
\begin{equation}{\label{star2}}
U^\prime(i_\ell, j-1) > U^\prime(i_\ell, j) \geq U^\prime(r,j).
\end{equation}
Further, note that
\begin{equation}{\label{star3}}
U(i_p, j)=U^\prime(i_{p+1}, j) \text{ for all } 0\leq p < t.
\end{equation}
Now combining (\ref{star1}),(\ref{star2}), and (\ref{star3}) we get in $U$ the inequalities 
\begin{equation}
{U(r,j)=U^\prime(r,j) > U^\prime(i_{\ell+1},j)= U(i_\ell, j)},
\end{equation}
but $U(r,j)=U^\prime(r,j)< U^\prime(i_\ell, j-1)=U(i_\ell, j-1)$, which is a Triple Rule violation in $U$.

Thus in all cases we obtain a contradiction.
\end{proof}

Consider the RCT obtained after $n$ successive insertions 
\begin{equation*}
U_n := (\cdots ((U \leftarrow b_1) \leftarrow b_2) \cdots) \leftarrow b_n
\end{equation*}
where the $b_i$ are arbitrary positive integers. Any row $i$ of $U_n$ will either consist entirely of boxes added during the successive insertions, or it will consist of some number of boxes from $U$ with some number of boxes added during the successive insertions.  In the former case row $i$ corresponds to some row $\hat{i}$ in each $U_j$ for $1 \leq k < j$, where $k$ is such that the insertion of $b_k$ adds a box in position $(\hat{i},1)$.  In the latter case row $i$ corresponds to some row $\hat{i}$ in each $U_j$ for all $0 \leq j \leq n$ where $U_0:=U$.  

As a direct consequence of Lemma \ref{lastrow} we have 

\begin{lemma}{\label{weakly longer}}
Consider $U_n$, the RCT obtained after $n$ successive insertions. Consider two rows $i$ and $i^\prime$ of $U_n$ such that $i < i^\prime$ and row $i$ is weakly longer than row $i^\prime$.  Suppose $b_{k_1}$ adds a box in position $(\hat{i},1)$ and $b_{k_2}$ adds a box in position $(\hat{i^\prime},1)$. Then $k_1<k_2$ and the corresponding row $\hat{i}$ is weakly longer than the corresponding row $\hat{i^\prime}$ in each $U_j$ for $j\geq k_2$.  
\end{lemma}

\begin{proof}
Suppose for a contradiction that at some intermediate step $U_j$ row $\hat{i}$ is strictly shorter than $\hat{i^\prime}$.  Since row $i$ is weakly longer than row $i^\prime$ in $U_n$, we must have that for some $\ell$, $j <\ell$, the new box produced in the insertion of $b_\ell$ into $U_{\ell-1}$ is at the end of the corresponding row $\hat{i}$ and rows $\hat{i}$ and $\hat{i^\prime}$ have the same length.  This contradicts Lemma \ref{lastrow}.

\end{proof}

Lemma \ref{lastrow} allows us to invert the insertion process for RCT's.  More specifically, given a RCT $U^\prime$ of shape $\alpha^\prime$ we can obtain a RCT $U$ of shape $\alpha$, where $\alpha^\prime=(\alpha_1, \ldots, \alpha_i+1, \ldots, \alpha_l)$ or $\alpha^\prime=(\alpha_1, \ldots, \alpha_{i-1}, 1, \alpha_i, \ldots, \alpha_l)$, in the following way.  We can un-insert the last entry, call it $y$, in row $i$ of $\alpha^\prime$, where row $i$ is the lowest row of length $j=\alpha_i+1$ or $j=1$. Do so by scanning up columns from bottom to top and un-bumping entries $\tilde{y}$ weakly greater than $y$ whenever $y$ is strictly greater than the entry to the right of $\tilde{y}$.  After scanning a column, we move one column to the right and continue scanning bottom to top.  In the end we will have un-inserted an entry $k$ and have produced a RCT $U$ such that $U^\prime=U \leftarrow k$.

\subsection{Main Bumping Property}

As above, let $U$ be a RCT with $k$ rows and longest row length $m$.  Consider $U \leftarrow b \leftarrow c$ with $b \leq c$.  Let $b_j^i$ be the entry ``in hand" which scans the entry in the $i$th row and $j$th column of $U$ during the insertion of $b$ into $U$, where $b_j^0$ is the element that begins scanning at the top of column $j$, so $b_{m+1}^0:=b$.  If the insertion of $b$ stops in position $(i_b,j_b)$ then $b_j^i:=0$ for all positions $(i,j) <_{col} (i_b,j_b)$ in $U$.  Similarly, let $c_j^i$ be the entry ``in hand" which we compare against the entry in the $i$th row and $j$th column of $U \leftarrow b$ during the insertion of $c$ into $U \leftarrow b$, where $c_j^0$ is the element that begins scanning the top of column $j$.  Note that $c$ will begin scanning in column $m+2$, since the insertion of $b$ may end in column $m+1$.  But when $b \leq c$ we have $c_{m+1}^0=c$ regardless of where the insertion of $b$ ends. If the insertion of $c$ into $U \leftarrow b$ stops in position $(i_c, j_c)$ we let $c_j^i:=0$ for all positions $(i,j)<_{col} (i_c, j_c)$ in $U \leftarrow b$. 

Now consider $U \leftarrow b \leftarrow a$ with $b>a$.  Define $b_j^i$ as above.  Similarly, we can define $a_j^i$ to be the entry which scans the entry in the $i$th row and $j$th column of $U \leftarrow b$ during the insertion of $a$ into $U \leftarrow b$.  Define $a_j^0$ to be the entry that begins scanning at the top of the $j$th column.  Define $a_{m+2}^0:=a$.  If the insertion of $a$ stops in position $(i_a, j_a)$ then let $a_j^i:=0$ for all positions $(i,j)<_{col} (i_a,j_a)$ in $U \leftarrow b$.  

\begin{lemma}{\label{scanning values}}
Let $U$ be a RCT with $k$ rows and longest row length $m$.  Let $a<b\leq c$ be positive integers. Suppose the insertion of $b$ into $U$ creates a new box in position $(i_b, j_b)$ in $U \leftarrow b$. The scanning values $b_j^i, c_j^i, a_j^i$ have the following relations.
\begin{enumerate}
\item Consider $U \leftarrow b \leftarrow c$.
	\begin{enumerate}
	\item If $U\leftarrow b$ has the same number of rows as $U$, then $b_j^i \leq c_j^i$ for all $(i,j)$ such that $0 \leq i \leq i_b	$ when $j=j_b$ and $0 \leq i \leq k$ when $j_b < j <m+1$.
	\item If $U\leftarrow b$ has one more row than $U$, that is $j_b=1$, then 
	\[
	\begin{array}{ll}
	b_j^i \leq c_j^i & \text{ for all } 0 \leq i \leq i_b \text{ and } 1 \leq j \leq m+1, \\
	b_j^{i} \leq c_j^{i+1} & \text{ for all } i_b \leq i \leq k+1 \text{ and } 2 \leq j \leq m+1.
	\end{array}
	\]
	\end{enumerate}
\item Consider $U \leftarrow b \leftarrow a$.
	\begin{enumerate}
	\item If $U\leftarrow b$ has the same number of rows as $U$, then $b_j^i > a_{j+1}^i$ for all $(i,j)$ such that $0 \leq i \leq i_b$ when $j=j_b$ and $0 \leq i \leq k$ when $j_b < j \leq m$.
	\item If $U\leftarrow b$ has one more row than $U$, that is $j_b=1$, then 
	\[
	\begin{array}{ll}
	b_j^i > a_{j+1}^i & \text{ for all } 0 \leq i \leq i_b \text{ and } 1 \leq j \leq m, \\
	b_j^i > a_{j+1}^{i+1} & \text{ for all } i_b \leq i \leq k+1 \text{ and } 2 \leq j \leq m+1.
	\end{array}
	\]
	\end{enumerate}
\end{enumerate}
\end{lemma}

\begin{remark}
Informally, Lemma \ref{scanning values} states that when doing consecutive insertions $U \leftarrow b \leftarrow c$ or $U \leftarrow b \leftarrow a$, the scanning values created by $b$ are weakly less than the scanning values created by $c$, and the scanning values of $b$ are strictly greater than the scanning values created by $a$.  Note that $b_j^i > a_{j+1}^i$ implies $b_j^i > a_j^i$ for $(i,j)$ satisfying the conditions of Lemma \ref{scanning values} part (2).
\end{remark}

\begin{proof}
\textbf{Proof of (1a):} Let $j=m+1$ and $i=0$.  Then clearly $b_{m+1}^0=b \leq c_{m+1}^0 =c$.  Now fix the column index $j>j_b$. Suppose by induction that $b_j^p \leq c_j^p$ for all $p\leq i$.  To show $b_j^{i+1} \leq c_j^{i+1}$ consider the following cases.

\emph{Case 1:} Suppose $b_j^i$ bumps the entry $b_j^{i+1}$ in position $(i, j)$ of $U$, and $c_j^i$ does not bump in position $(i,j)$ of $U^\prime$.  In this case, $b_j^{i+1} \leq b_j^i \leq c_j^i = c_j^{i+1}$.

\emph{Case 2:}  Suppose $b_j^i$ bumps the entry $b_j^{i+1}$ in position $(i, j)$ of $U$, and $c_j^i$ bumps the entry $c_j^{i+1}=b_j^i$ in position $(i,j)$ of $U^\prime$.  Then $b_j^{i+1} \leq b_j^i =c_j^{i+1}$.

\emph{Case 3:} Suppose neither $b_j^i$ nor $c_j^i$ bump in position $(i,j)$ of their respective RCT.  Then $b_j^{i+1}=b_j^i \leq c_j^i = c_j^{i+1}$.

\emph{Case 4:} Suppose $b_j^i$ does not bump in position $(i,j)$ of $U$, but $c_j^i$ bumps $c_j^{i+1}$ in position $(i,j)$ of $U^\prime$.  Consider the following diagram which depicts row $i$ and columns $j-1$ and $j$ in each of $U$, $U \leftarrow b$, and $U \leftarrow b \leftarrow c$.
\[
\begin{array}{ccc}
U & U \leftarrow b & U \leftarrow b \leftarrow c \\
\Tableau{ d & c_j^{i+1}} & \Tableau{ \tilde{d} & c_j^{i+1} } & \Tableau{ \tilde{d} & c_j^i }
\end{array}
\]

If $d$ is bumped by $\tilde{d}$ during the insertion of $b$, then $d \leq \tilde{d} \leq b_{j-1}^0 \leq b_j^i \leq c_j^i$ which contradicts row strictness of $U \leftarrow b \leftarrow c$.  So assume $d$ does not get bumped by $\tilde{d}$, that is $d=\tilde{d}$.  We get $d= \tilde{d} > c_j^i \geq b_j^i$ and since $b_j^i$ does not bump we must have $b_j^i < c_j^{i+1}$.  But then $b_j^i =b_j^{i+1} < c_j^{i+1}$.

The argument above shows that for fixed $j$, $b_j^i \leq c_j^i$ for all $0 \leq i \leq k$.  But this immediately implies $b_{j-1}^0 \leq c_{j-1}^0$ and thus we have $b_j^i \leq c_j^i$ for all $(i,j)$ indicated in the lemma.

\textbf{Proof of (1b):} Notice that row $i+1$ in $U^\prime$ will correspond to row $i$ in $U$ for all $i_b \leq i \leq k+1$.  Since row $i_b$ has only one box in it, then $c_j^{i_b}=c_j^{i_b+1}$ for $3\leq j \leq m+1$.  So assume $j$ is fixed such that $3 \leq j \leq m+1$.  The proof for part (1a) establishes $b_j^i \leq c_j^i$ for $0 \leq i \leq i_b$, which immediately implies $b_j^{i_b} \leq c_j^{i_b+1}$.  

Now suppose by induction that $b_j^p \leq c_j^{p+1}$ for all $p$ such that $i_b\leq p \leq i$ for some $i$.  We establish $b_j^{i+1} \leq c_j^{i+2}$ by considering the following cases.

\emph{Case 1:} Suppose $b_j^{i}$ bumps the entry $b_j^{i+1}$ in position $(i, j)$ of $U$, and $c_j^{i+1}$ does not bump in position $(i+1,j)$ of $U^\prime$.  Then $b_j^{i+1} \leq b_j^{i} \leq c_j^{i+1} =c_j^{i+2}$.

\emph{Case 2:} Suppose $b_j^{i}$ bumps the entry $b_j^{i+1}$ in position $(i, j)$ of $U$, and $c_j^{i+1}$ bumps the entry $c_j^{i+2}=b_j^{i}$ in position $(i+1,j)$ of $U^\prime$.  Then $b_j^{i+1} \leq b_j^i = c_j^{i+2}$.

\emph{Case 3:} Suppose neither $b_j^{i}$ does not bump in position $(i,j)$ of $U$ and $c_j^{i+1}$ does not bump in position $(i+1,j)$ of $U^\prime$.  Then $b_j^{i+1}=b_j^{i} \leq c_j^{i+1}=c_j^{i+2}$.

\emph{Case 4:} Suppose $b_j^{i}$ does not bump in position $(i,j)$ of $U$, but $c_j^{i+1}$ bumps $c_j^{i+2}$ in position $(i+1,j)$ of $U^\prime$. Consider the following diagram which depicts columns $j-1$ and $j$ and the labelled rows of $U$, $U \leftarrow b$, and $U\leftarrow b \leftarrow c$.
\[
\begin{array}{lccc}
 & U & U \leftarrow b & U \leftarrow b \leftarrow c \\
i& \Tableau{ d & c_j^{i+2}} & \vdots & \vdots  \\
i+1& \vdots & \Tableau{ \tilde{d} & c_j^{i+2} } & \Tableau{ \tilde{d} & c_j^{i+1} }
\end{array}
\]

If $d$ is bumped by $\tilde{d}$ during the insertion of $b$, then $d \leq \tilde{d} \leq b_{j-1}^0 \leq b_j^{i} \leq c_j^{i+1}$ which contradicts row strictness of $U \leftarrow b \leftarrow c$.  So assume $d$ does not get bumped by $\tilde{d}$, that is $d=\tilde{d}$.  We get $d= \tilde{d} > c_j^{i+1} \geq b_j^{i}$ and since $b_j^{i}$ does not bump we must have $b_j^{i} < c_j^{i+2}$.  But then $b_j^{i+1} =b_j^{i} < c_j^{i+2}$.

When $j=2$, the above argument shows $b_2^i \leq c_2^i$ for all $0 \leq i \leq i_b$, which implies the insertion of $c$ cannot add a new box with entry $c_2^{i_b}$ in position $(i_b, 2)$ of $U^\prime$, otherwise $U^\prime(i_b,1) \leq b_2^{i_b} \leq c_2^{i_b}$.  So $c_2^{i_b}=c_2^{i_b+1}$ and the above argument shows $b_2^i \leq c_2^{i+1}$ for all $i_b \leq i \leq k+1$.

The case of $j=2$ implies $b_1^0 \leq c_1^0$.  The definition of insertion implies $b_1^i=b_1^0$ for all $0 \leq i \leq i_b$, and $c_1^i=c_1^0$ for all $0 \leq i \leq i_b$.  Thus, the relations in part (1b) of the lemma follow. 

\textbf{Proof of (2a)} We have $b_{m+1}^0=b > a_{m+2}^0=a$ by assumption.  Now fix a column $j>j_b$ and assume by induction that $b_j^p > a_{j+1}^p$ for all $p\leq i$ for some $i$.  To show $b_j^{i+1} > a_{j+1}^{i+1}$ we consider the following cases.

\emph{Case 1:} Suppose $b_j^i$ bumps the entry $b_j^{i+1}$ in position $(i,j)$ of $U$, and $a_{j+1}^i$ bumps the entry $a_{j+1}^{i+1}$ in position $(i, j+1)$ of $U \leftarrow b$.  Then $U(i,j)=b_j^{i+1} > U(i,j+1)=a_{j+1}^{i+1}$ by row-strictness of $U$.

\emph{Case 2:} Suppose $b_j^i$ does not bump in position $(i,j)$ of $U$, and $a_{j+1}^i$ bumps the entry $a_{j+1}^{i+1}$ in position $(i, j+1)$ of $U \leftarrow b$.  Then $b_j^{i+1}=b_j^i > a_{j+1}^i \geq a_{j+1}^{i+1}$.

\emph{Case 3:} Suppose $b_j^i$ does not bump in position $(i,j)$ of $U$, and $a_{j+1}^i$ does not bump in position $(i, j+1)$ of $U\leftarrow b$.  Then $b_j^{i+1}=b_j^i > a_{j+1}^i = a_{j+1}^{i+1}$.

\emph{Case 4:} Suppose $b_j^i$ bumps the entry $b_j^{i+1}$ in position $(i,j)$ of $U$, and $a_{j+1}^i$ does not bump in position $(i, j+1)$ of $U\leftarrow b$. Let $U^\prime=U \leftarrow b$.  Then 
\begin{displaymath}
U^\prime(i,j)=b_j^i \geq b_j^{i+1}=U(i,j) > U(i,j+1)=U^\prime(i,j+1).
\end{displaymath}  
Because $U^\prime(i,j)=b_j^i>a_{j+1}^i$ and $a_{j+1}^i$ does not bump, we must have $a_{j+1}^i < U^\prime(i,j+1)$.  This implies $b_j^{i+1}>a_{j+1}^i=a_{j+1}^{i+1}$.

The argument above shows that for fixed $j$,  $b_j^i > a_{j+1}^i$ for all $0 \leq i \leq k$.  This implies $b_{j-1}^0 > a_j^0$, which then implies the relations in part (2a) of the lemma.

\textbf{Proof of (2b):} Notice that row $i+1$ in $U^\prime$ will correspond to row $i$ in $U$ for all $i_b \leq i \leq k+1$.  Since row $i_b$ has only one box in it, then $a_{j+1}^{i_b}=a_{j+1}^{i_b+1}$ for $2\leq j \leq m+1$.  So assume $j$ is fixed such that $2 \leq j \leq m+1$.  The proof for part (2a) establishes $b_j^i > a_{j+1}^i$ for $0 \leq i \leq i_b$, which immediately implies $b_j^{i_b} > a_{j+1}^{i_b+1}$.  

Now suppose by induction that $b_j^p > a_{j+1}^{p+1}$ for all $p$ such that $i_b\leq p \leq i$ for some $i$.  We establish $b_j^{i+1} > a_{j+1}^{i+2}$ by considering the following cases.

\emph{Case 1:} Suppose $b_j^i$ bumps the entry $b_j^{i+1}$ in position $(i,j)$ of $U$, and $a_{j+1}^{i+1}$ bumps the entry $a_{j+1}^{i+2}$ in position $(i+1, j+1)$ of $U \leftarrow b$.  Then $U(i,j)=b_j^{i+1}> U(i,j+1)=U^\prime(i+1, j+1)=a_j^{i+2}$. 

\emph{Case 2:}  Suppose $b_j^i$ does not bump in position $(i,j)$ of $U$, and $a_{j+1}^{i+1}$ bumps the entry $a_{j+1}^{i+2}$ in position $(i+1, j+1)$ of $U \leftarrow b$. Then $b_j^{i+1}=b_j^i>a_{j+1}^{i+1} \geq a_{j+1}^{i+2}$.

\emph{Case 3:} Suppose $b_j^i$ does not bump in position $(i,j)$ of $U$, and $a_{j+1}^{i+1}$ does not bump in position $(i+1, j+1)$ of $U\leftarrow b$. Then $b_j^{i+1}=b_j^i>a_{j+1}^{i+1} = a_{j+1}^{i+2}$.

\emph{Case 4:} Suppose $b_j^i$ bumps the entry $b_j^{i+1}$ in position $(i,j)$ of $U$, and $a_{j+1}^{i+1}$ does not bump in position $(i+1, j+1)$ of $U\leftarrow b$. Let $U^\prime=U \leftarrow b$. Then
\begin{displaymath}
U^\prime(i+1, j)=b_j^i \geq b_j^{i+1} = U(i,j) > U(i,j+1)=U^\prime(i+1, j+1).
\end{displaymath}
Because $U^\prime(i+1,j)=b_j^i>a_{j+1}^{i+1}$ and $a_{j+1}^{i+1}$ does not bump, we must have $a_{j+1}^{i+1} < $$U^\prime(i+1,j+1)$. This implies $b_j^{i+1}> a_{j+1}^{i+1}=a_{j+1}^{i+2}$.

In the case $j=1$, the definition of RCT insertion forces each scanning value $b_1^i=b_1^0$ for all rows $0 \leq i \leq i_b$.  Since $b_1^0 > a_2^0$ by the argument above, and since the entries bumped by $a_2^0$ in the second column get weakly smaller we have $b_1^i>a_2^i$ for all $0 \leq i \leq i_b$ as needed.
\end{proof}

We can apply Lemma \ref{scanning values} to prove the following proposition, which describes where new boxes are added after consecutive insertions. 

\begin{proposition}{\label{bumping}}
Let $U$ be a RCT with $k$ rows, longest row length $m$.  Let $a$, $b$, and $c$ be positive integers with $a<b\leq c$.  Consider successive insertions $U_1 := (U \leftarrow b)\leftarrow c$ and $U_2:=(U \leftarrow b) \leftarrow a$.  Let $B_a=(i_a, j_a), B_b=(i_b, j_b),$ and $B_c=(i_c, j_c)$ be the new boxes created after inserting $a,b,$ and $c$, respectively, into the appropriate RCT.  Let $i_1$ be a row in $U_1$ which contains a box $(i_1, j_1)$ from $I(b)$ and a box $(i_1, j_1^\prime)$ from $I(c)$.  Similarly, let $i_2$ be a row in $U_2$ which contains a box $(i_2, j_2)$ from $I(b)$ and a box $(i_2, j_2^\prime)$ from $I(a)$.   Then 
\begin{enumerate}
\item In $U_1$, $j_c \leq j_b$.  In $U_2$, $j_a > j_b$.  
\item In $U_1$, $j_1^\prime \leq j_1$.  In $U_2$, $j_2^\prime > j_2$.
\end{enumerate}
\end{proposition}

\begin{remark} Informally, part (1) of Proposition \ref{bumping} states that if $a < b \leq c$, then in $U \leftarrow b \leftarrow c$ the new box created by $c$ is weakly left of the new box created by $b$, and in $U \leftarrow b \leftarrow a$ the new box created by $a$ is strictly right of the new box created by $b$. Part (2) of Proposition \ref{bumping} states that the insertion path of $c$ is weakly left of the insertion path of $b$, and the insertion path of $a$ is strictly right of the insertion path of $b$.
\end{remark}

\begin{proof}
\textbf{Proof of (1):} Lemma \ref{scanning values} part (1) shows that during the insertion of $c$ into $U \leftarrow b$, the scanning values $c_{j_b}^{i_b}$ is weakly greater than the entry occupying the box $B_b$, which forces the new box $B_c$ to be weakly left of $B_b$, that is $j_c\leq j_b$.  Lemma \ref{scanning values} part (2) show that during the insertion of $a$ into $U \leftarrow b$, the new box $B_a$ must occupy position $(i_b, j_b+1)$ if the insertion process reaches this position, implying that the new box $B_a$ is always strictly right of the box $B_b$, that is $j_a> j_b$.

\textbf{Proof of (2):} Suppose for a contradiction that there is a row $i_1$ of $U_1$ which contains a box $(i_1, j_1)$ from $I(b)$ and a box $(i_1, j_1^\prime)$ from $I(c)$, and that $j_1^\prime >j_1$.  Then 
\begin{displaymath}
U_1(i_1,j_1) \leq b_{j_1}^0 \leq c_{j_1}^0 \leq U_1(i_1, j_1^\prime)
\end{displaymath}  
which contradicts row-strictness in $U_1$.

Again, suppose for a contradiction that there is a row $i_2$ in $U_2$ which contains a box $(i_2, j_2)$ from $I(b)$ and a box $(i_2, j_2^\prime)$ from $I(a)$ and $j_2^\prime \leq j_2$.  If the boxes coincide, that is $j_2=j_2^\prime$, then $a_{j_2}^{i_2}$ bumped the entry $b_{j_2}^{i_2}$ in position $(i_2, j_2)$ of $U \leftarrow b$, and Lemma \ref{scanning values} shows $b_{j_2}^{i_2} > a_{j_2+1}^{i_2} \geq a_{j_2}^{i_2}$, which contradicts the definition of RCT insertion.  If $j_2^\prime < j_2$ then
\begin{displaymath}
U_2(i_2, j_2^\prime) \leq a_{j_2^\prime}^0< b_{j_2^\prime}^0 \leq U_2(i_2, j_2)
\end{displaymath}
where $a_{j_2^\prime}^0 < b_{j_2^\prime}^0$ is established by using Lemma \ref{scanning values}. But this contradicts row-strictness of $U_2$.
\end{proof}

The following lemma follows from Proposition \ref{bumping}.

\begin{lemma}{\label{readingorder}}
Consider the RCT obtained after $n$ successive insertions 
\begin{displaymath}
U_n := (\cdots ((U \leftarrow b_1) \leftarrow b_2) \cdots) \leftarrow b_n
\end{displaymath}
with $b_1\leq b_2 \leq \cdots \leq b_n$ positive integers. Let $B_1, B_2, \ldots, B_n$ be the corresponding new boxes.  Then in $U_n$, 
\begin{displaymath}
B_n <_{col} B_{n-1} <_{col} \cdots <_{col} B_1.
\end{displaymath}
\end{lemma}

\begin{proof}
Proposition \ref{bumping} implies the new boxes are added weakly right to left.  Let $i_1 <i_2$ and consider two boxes $B_{i_1}$ and $B_{i_2}$ in the same column.  Note that $B_{i_1}$ and $B_{i_2}$ cannot coincide.  Suppose for a contradiction that $B_{i_1}$ is (strictly) below $B_{i_2}$.  Suppose the row containing box $B_{i_1}$ has length $j$.  Once $B_{i_1}$ is added, the new boxes $B_k$  for $i_1<k<i_2$ cannot change the length of the row containing $B_{i_1}$.    Thus, when $B_{i_2}$ is added to the end of a row of length $j-1$ strictly above the row containing $B_{i_1}$, we contradict Lemma \ref{lastrow}.
\end{proof}

\subsection{Elementary Transformations}{\label{subsec:knuth}}

Knuth's contribution to the RSK algorithm included describing Schensted insertion in terms of two \emph{elementary transformations $\kone$ and $\ktwo$} which act on words $w$.  Let $a,b,$ and $c$ be positive integers.  Then
\begin{displaymath}
\begin{array}{lll}
\mathcal{K}_1: & bca \to bac & \text{ if } a<b \leq c \\
\mathcal{K}_2: & acb \to cab & \text{ if } a \leq b < c 
\end{array}.
\end{displaymath}

The relations $\kone, \ktwo,$ and their inverses $\kone^{-1}, \ktwo^{-1}$, act on words $w$ by transforming triples of consecutive letters.  Denote by $\stackrel{1}{\cong}$ the equivalence relation defined by using $\kone$ and $\kone^{-1}$.  That is, $w \stackrel{1}{\cong} w^\prime$ if and only if $w$ can be transformed into $w^\prime$ using a finite sequence of transformations $\kone$ or $\kone^{-1}$.

\begin{lemma}{\label{knuth}}
Let $U$ be a RCT and let $w$ and $w^\prime$ be two words such that $w \stackrel{1}{\cong} w^\prime$.  Then
\begin{displaymath}
U \leftarrow w = U \leftarrow w^\prime.
\end{displaymath}
\end{lemma}
\begin{proof}
It suffices to show
\begin{displaymath}
U \leftarrow b \leftarrow c \leftarrow a = U \leftarrow b \leftarrow a \leftarrow c
\end{displaymath}
for positive integers $a < b \leq c$.

To distinguish the two sets of boxes bumped by the entry $c$ we will let $I^{b}(c)$ be the insertion path created when inserting $c$ into $U \leftarrow b$, and we will let $I^{ba}(c)$ be the insertion path created when inserting $c$ into $U \leftarrow b \leftarrow a$. Consider the insertion path $I^{b}(c)$.  By Proposition \ref{bumping} we know that in $U \leftarrow b \leftarrow a$ the insertion of $a$ cannot end in a new box in the first column.  Thus, for each box $(i,j)$ in $I^{b}(c)$ we may consider the corresponding box $(i,j)$ in $U \leftarrow b \leftarrow a$.

We will inductively show that the insertion path of $c$ when inserting into $U \leftarrow b \leftarrow a$ is the exact same set of boxes $I^b(c)$ and bumps exactly the same set of entries in these boxes.  Let $(i_{\max},j_{\max})$ be the largest box with respect to $<_{col}$ in $I^b(c)$ in $U \leftarrow b \leftarrow c$, that is, the box $(i_{\max},j_{\max})$ is the first box bumped in the insertion of $c$ into $U \leftarrow b$.  

To show the base case, we need to show that the entry $c$ bumps in box $(i_{\max},j_{\max})$ of $U \leftarrow b \leftarrow a \leftarrow c$, and bumps the same entry as $c$ bumped in box $(i_{\max},j_{\max})$ of $U \leftarrow b \leftarrow c$.  We establish the base case is three steps.

\emph{Step 1:} We claim the entry $c$ cannot bump in a box $(p,q)$ such that $(i_{\max},j_{\max}) <_{col} (p,q)$ in $U \leftarrow b \leftarrow a$.  Suppose for a contradiction that $c$ does bump in box $(p,q)$. Even if $j_{\max}=1$, the box $(p,q)$ must have $q \neq 1$ by the definition of RCT insertion; that is, if $j_{\max}=1$ and $q=1$ then $p=i_{\max}$. The only way $c$ can bump in box $(p,q)$ of $U \leftarrow b \leftarrow a$ and not the corresponding box of $U \leftarrow b$ is if the insertion of $a$ changed the entry in the box $(p, q-1)$ (or added a new box with coordinates $(p, q-1)$) which the scanning value $c$ in column $q$ compares to.  Informally, the insertion of $a$ into $U \leftarrow b$ made the entry in $(p, q-1)$ larger, thus allowing $c$ to bump in box $(p,q)$.  Thus, the box $(p,q-1)$ is in the insertion path of $a$ in $U \leftarrow b \leftarrow a$, which places the insertion path of $a$ strictly to the left of the insertion path of $c$ in $U \leftarrow b \leftarrow a \leftarrow c$, which contradictions Proposition \ref{bumping}.

\emph{Step 2:} We claim the box $(i_{\max},j_{\max})$ in $U \leftarrow b \leftarrow a $ must be in the path of $c$.  Suppose for a contradiction that the box $(i_{\max},j_{\max})$ is not in the path of $c$. These assumptions imply $j_{\max} \neq 1$. The only reason why $c$ would not bump in $(i_{\max}, j_{\max})$ when inserting $c$ into $U \leftarrow b \leftarrow a$ would be because $a$ had bumped in $(i_{\max}, j_{\max})$.  Informally, the insertion of $a$ into $U \leftarrow b$ made the entry in $(i_{\max}, j_{\max})$ too large for $c$ to bump. Let $d$ be the entry in position $(i_{\max},j_{\max}-1)$ of $U \leftarrow b$, $U \leftarrow b \leftarrow c$, and $U \leftarrow b \leftarrow a$. The value $d$ is the same in each of these three RCT because of our assumptions and Lemma \ref{rows}. Since $(i_{\max},j_{\max})$ is in the path of $c$ when inserting $c$ into $U \leftarrow b $, then $d > c$.  Similarly, $d > a_{j_{\max}}^{i_{\max}}$ where $a_{j_{\max}}^{i_{\max}}$ is the scanning value of $a$.  Since $(i_{\max},j_{\max})$ is not in the path of $c$ in $U \leftarrow b \leftarrow a  \leftarrow c$ by assumption, and since $d$ is still the entry in position $(i_{\max},j_{\max}-1)$ when $c$ scans it, we must have $c=c_{j_{\max}}^{i_{\max}} < a_{j_{\max}}^{i_{\max}}$, which contradicts Lemma \ref{scanning values}.

\emph{Step 3:} We claim the entry $c$ bumps the same value in box $(i_{\max},j_{\max})$ in both $U \leftarrow b  \leftarrow c$ and $U \leftarrow b  \leftarrow a  \leftarrow c$.  Clearly, this claim only needs to be checked when $j_{\max} \geq 2$. Suppose for a contradiction that $c$ bumps a different value in $(i_{\max},j_{\max})$ of $U\leftarrow b  \leftarrow a \leftarrow c$. Then this implies the box $(i_{\max},j_{\max})$ is in the path of $a$ and (by assumption) in the path of $c$. Consider the following diagram, which depicts boxes $(i_{\max},j_{\max}-1)$ and $(i_{\max},j_{\max})$.
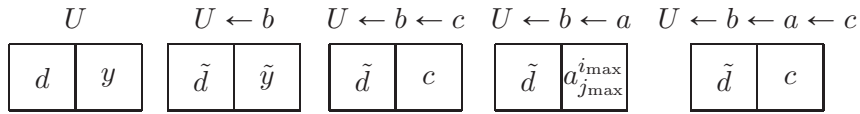
\[
\begin{array}{ccccc}
U & U \leftarrow b & U \leftarrow b \leftarrow c & U \leftarrow b  \leftarrow a & U \leftarrow b \leftarrow a \leftarrow c \\
\Tableau{ d & y} & \Tableau{ \tilde{d} & \tilde{y} } & \Tableau{ \tilde{d} & c } & \Tableau{\tilde{d} & a_{j_{\max}}^{i_{\max}} } & \Tableau{\tilde{d} & c}
\end{array}
\]
Above, $\tilde{d}$ and $\tilde{y}$ denote the fact that at most one of $d$ or $y$ could have been bumped during the insertion of $b$ into $U$, but not both.  Using Lemma \ref{scanning values} we get $b_{j_{\max}}^{i_{\max}} > a_{j_{\max}+1}^{i_{\max}}$ (or $b_{j_{\max}}^{i_{\max}} > a_{j_{\max}+1}^{i_{\max}+1}$ in the appropriate rows if the insertion of $b$ into $U$ created a new row). Thus  $b_{j_{\max}}^{i_{\max}} > a_{j_{\max}}^{i_{\max}} \geq \tilde{y} \geq y$.  We can also establish $b_{j_{\max}}^{i_{\max}} \leq c_{j_{\max}}^{i_{\max}}=c$ (or $b_{j_{\max}}^{i_{\max}} \leq c_{j_{\max}}^{i_{\max}+1} \leq c_{j_{\max}}^{i_{\max}}=c$).  In the case $y=\tilde{y}$ then either $\tilde{d}=b_{j_{\max}-1}^{i_{\max}} \leq b_{j_{\max}}^{i_{\max}} \leq c < \tilde{d}$ which is a contradiction, or $\tilde{d}=d>c \geq b_{j_{\max}}^{i_{\max}}$.  In the case $\tilde{y}=b_{j_{\max}}^{i_{\max}}$ then $d=\tilde{d}> c \geq b_{j_{\max}}^{i_{\max}}$. In all cases we have $d>b_{j_{\max}}^{i_{\max}}$ and $b_{j_{\max}}^{i_{\max}} \geq y$, which implies $b_{j_{\max}}^{i_{\max}}$ must bump in position $(i_{\max},j_{\max})$.  This immediately implies (by way of Proposition \ref{bumping}) that $a$ cannot have the box $(i_{\max},j_{\max})$ in its insertion path.

This completes the base case.  To finish the proof we induct on the length of the path $I^b(c)$. Suppose by induction that the path of $c$ in $U \leftarrow b \leftarrow c$ is identical to the path of $c$ in $U \leftarrow b \leftarrow a \leftarrow c$, and the bumped entries are the same in both paths, up to some box $(i,j)$ where the path, or the value bumped, is different.  Under the inductive hypothesis the scanning values $c_s^r$ obtained when inserting $c$ into $U \leftarrow b$ are equal to the scanning values, also denoted $c_s^r$, obtained when inserting $c$ into $U \leftarrow b \leftarrow a$ up to the box $(i,j)$. We show $(i,j)$ in $U \leftarrow b \leftarrow a \leftarrow c$ is in the path of $c$ if and only if $(i,j)$ in $U\leftarrow b \leftarrow c$ is in the path of $c$, and $c_j^i$ bumps the same valued entry.  We do this in three steps which are identical to the three steps above.

\emph{Step 1:} We claim if $(i,j)$ is in the path of $c$ in $U \leftarrow b \leftarrow a \leftarrow c$, then $(i,j)$ is in the path of $c$ in $U \leftarrow b \leftarrow c$.  This is clearly true if $j=1$.  When $j \geq 2$ and if this were not the case, that is $(i,j)$ is not in the path of $c$ in $U \leftarrow b \leftarrow c$, then the entry $a$ must have $(i,j-1)$ in its insertion path in $U \leftarrow b \leftarrow a$, which places the path of $a$ strictly to the left of the path of $c$ in $U \leftarrow b \leftarrow a \leftarrow c$ which is a contradiction.

\emph{Step 2:} We further claim that if $(i,j)$ is in the path of $c$ in $U \leftarrow b \leftarrow c$ then $(i,j)$ is in the path of $c$ in $U \leftarrow b \leftarrow a \leftarrow c$. Again, this is clearly true if $j=1$.  When $j \geq 2$ and if this were not the case, that is $(i,j)$ is not in the path of $c$ in $U \leftarrow b \leftarrow a \leftarrow c$, then the entry $a$ must have bumped in position $(i,j)$. Let $d$ be the entry in box $(i,j-1)$ of $U \leftarrow b$, $U \leftarrow b \leftarrow c$, and $U \leftarrow b \leftarrow a$. The fact that $(i,j)$ is in the path of $c$ in $U \leftarrow b \leftarrow c$ implies $d > c_j^i$.  Similarly, $d > a_j^i$.  Under our assumptions $(i,j)$ is not in the path of $c$ in $U \leftarrow b \leftarrow a \leftarrow c$, which implies $c_j^i < a_j^i$ which contradicts Lemma \ref{scanning values}.

\emph{Step 3:} We claim the same value is bumped in box $(i,j)$ of $U\leftarrow b \leftarrow c$ and $U \leftarrow b \leftarrow a \leftarrow c$. Observe this step only needs to be checked if $j \geq 2$.  If the claim were false then we must have that both $a$ and $c$ have the box $(i,j)$ in their respective insertion paths in $U \leftarrow b \leftarrow a \leftarrow c$.  Consider the following diagram which depicts boxes $(i,j-1)$ and $(i,j)$.
\[
\begin{array}{ccccc}
U & U \leftarrow b & U \leftarrow b \leftarrow c & U \leftarrow b  \leftarrow a & U \leftarrow b \leftarrow a \leftarrow c \\
\Tableau{ d & y} & \Tableau{ \tilde{d} & \tilde{y} } & \Tableau{ \tilde{d} & c_j^i } & \Tableau{\tilde{d} & a_j^i } & \Tableau{\tilde{d} & c_j^i}
\end{array}
\]
Using Lemma \ref{scanning values} we get $b_j^i > a_{j+1}^i$ (or $b_j^i > a_{j+1}^{i+1}$ in the appropriate rows if the insertion of $b$ into $U$ created a new row). Thus  $b_j^i > a_j^i \geq \tilde{y} \geq y$.  We can also establish $b_j^i \leq c_j^i$ (or $b_j^i \leq c_j^{i+1} \leq c_j^i$).  In the case $y=\tilde{y}$ then either $\tilde{d}=b_{j-1}^i \leq b_j^i \leq c_j^i < \tilde{d}$ which is a contradiction, or $\tilde{d}=d>c_j^i \geq b_j^i$.  In the case $\tilde{y}=b_j^i$ then $d=\tilde{d}> c_j^i \geq b_j^i$. In all cases we have $d>b_j^i$ and $b_j^i \geq y$, which implies $b_j^i$ must bump in position $(i,j)$.  This immediately implies (by way of Proposition \ref{bumping}) that $a$ cannot have the box $(i,j)$ in its bumping path.

Thus the path of $c$ in $U \leftarrow b \leftarrow c$, $I^b(c)$, is exactly the same set of boxes as the path of $c$ in $U \leftarrow b \leftarrow a \leftarrow c$, $I^{ba}(c)$. In both cases the same valued entries are bumped.

Now consider the insertion path $I^b(a)$ of $a$ in $U \leftarrow b \leftarrow a$.  By Proposition \ref{bumping} the insertion of $c$ into $U \leftarrow b$ may create a new box in the first column.  Despite this we may still consider the boxes in $U\leftarrow b \leftarrow c$ that correspond to the boxes in $I^b(a)$ since any particular box $(i,j)$ in $I^b(a)$ corresponds to the box $(i,j)$ in $U \leftarrow b \leftarrow c$ if row $i$ is above the new row created by $c$, or $(i,j)$ in $I^b(a)$ corresponds to $(i+1, j)$ in $U \leftarrow b \leftarrow c$ if row $i$ is weakly below the new row created by $c$.  With this in mind we will denote by $\widehat{(i,j)}$ the box in $U\leftarrow b \leftarrow c$ that corresponds to the box $(i,j)$ in $I^b(a)$.

Note that in both $U \leftarrow b \leftarrow a$ and $U \leftarrow b \leftarrow c \leftarrow a$ the path of $a$ cannot contain a box in the first column of the respective RCT.

We will inductively show that the path of $a$ in both $U \leftarrow b \leftarrow a$ and $U \leftarrow b \leftarrow c \leftarrow a$ consists of the the same (corresponding) boxes and the entries bumped in each path are equal entry by entry. Let $(i_{\max},j_{\max})$ be the largest box in $I^b(a)$ with respect to $<_{col}$. The base case can be established in three steps.

\emph{Step 1:} The entry $a$ cannot bump before the box $\widehat{(i_{\max},j_{\max})}$.  Suppose for a contradiction that $a$ bumped in some box $(p,q)$ with $\widehat{(i_{\max},j_{\max})}<_{col}(p,q)$. This implies that the box $(p, q-1)$ is in the path of $c$ in $U \leftarrow b \leftarrow c \leftarrow a$. 

Assume $q-1 \geq 2$.  Consider the diagram below, which depicts boxes $(p, q-2), (p, q-1)$, and $(p,q)$.
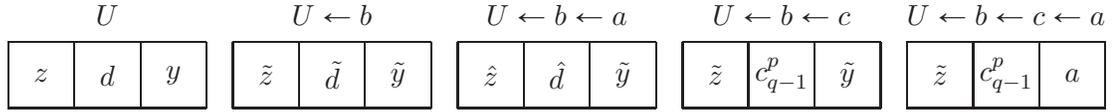
\[
\begin{array}{ccccc}
U & U \leftarrow b & U \leftarrow b \leftarrow a & U \leftarrow b  \leftarrow c & U \leftarrow b \leftarrow c \leftarrow a \\ 
\Tableau{ z & d & y} & \Tableau{ \tilde{z} & \tilde{d} & \tilde{y} } & \Tableau{ \hat{z} & \hat{d} & \tilde{y} } & \Tableau{\tilde{z} & c_{q-1}^p & \tilde{y} } & \Tableau{\tilde{z} & c_{q-1}^p & a}
\end{array}
\]
where either $\hat{z}$ or $\hat{d}$ could equal $a$ (but not both). We then get the inequalities $a \geq \tilde{y}$ which forces $a \geq \tilde{d}$. By Lemma \ref{scanning values} we know $b_{q-1}^p > a$, which implies $b_{q-1}^p > \tilde{d}\geq d$.  Thus the box $(p, q-1)$ is not in the path of $b$ and $\tilde{d}=d$.  On the other hand we see $\tilde{z} > c_{q-1}^p \geq b_{q-1}^p$, and since the path of $c$ cannot be strictly right of the path of $b$ we also see the box $(p, q-2)$ is not in the path of $b$ and thus $z=\tilde{z}$.  In the end we get the relations $z> b_{q-1}^p$ and $b_{q-1}^p> d$ which implies the box $(p, q-1)$ is in the path of $b$ and is a contradiction to the previously established condition on the box $(p, q-1)$.

Now we can assume $q-1=1$.  In this case the box $(p, q-1)=(p,1)$ is still in the path of $c$ and the position $(p,q)=(p, 2)$ is empty during the insertion of $a$.  With our assumptions that $(p,2)$ is in the path of $a$ in $U \leftarrow b \leftarrow c \leftarrow a$ and $\widehat{(i_{\max},j_{\max})} <_{col} (p,q)$ this forces $j_{\max}=2$ and the insertion of $b$ must have created a new box in the first column, say in position $(r, 1)$ with $r<p$.  This means position $(r, 2)$ is empty during the insertion of $a$ and by Lemma \ref{scanning values}, $b_q^r > a$ and thus $a$ must insert in position $(r,2)$. Which means $a$ cannot have $(p,q)$ in its path.

\emph{Step 2:} We claim the entry $a$ must bump in box $\widehat{(i_{\max},j_{\max})}$. Suppose for a contradiction that $a$ does not bump in box $\widehat{(i_{\max},j_{\max})}$ during the insertion of $a$ into $U \leftarrow b \leftarrow c$.  If $a$ does not bump in box $\widehat{(i_{\max},j_{\max})}$ then we must have $\widehat{(i_{\max},j_{\max})}$ in the path of $c$. As indicated above, $j_{\max} \neq 1$.

Consider the following diagram which depicts boxes $(i_{\max},j_{\max}-1)$ and $(i_{\max},j_{\max})$ and the corresponding boxes $\widehat{(i_{\max},j_{\max}-1)}$ and $\widehat{(i_{\max},j_{\max})}$.
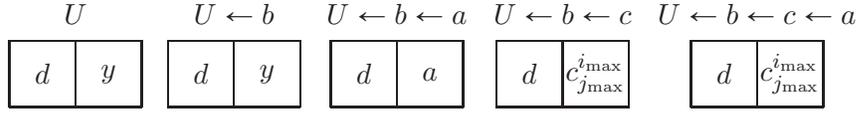
\[
\begin{array}{ccccc}
U&U \leftarrow b & U \leftarrow b \leftarrow a & U \leftarrow b  \leftarrow c & U \leftarrow b \leftarrow c \leftarrow a \\
 \Tableau{ d&y}&\Tableau{ d & y } & \Tableau{ d & a } & \Tableau{ d & c_{j_{\max}}^{i_{\max}} } & \Tableau{d & c_{j_{\max}}^{i_{\max}}}
\end{array}
\]
where neither box $(i_{\max},j_{\max})$ nor $(i_{\max},j_{\max}-1)$ can be in the path of $b$ since the box $(i_{\max},j_{\max})$ is in the path of $a$ and $\widehat{(i_{\max},j_{\max})}$ is in the path of $c$. From our assumptions we get the inequalities $d > a \geq y$ and $d > c_{j_{\max}}^{i_{\max}}>a$.  Now consider the scanning values obtained during the insertion of $b$.  Lemma \ref{scanning values} implies $b_{j_{\max}}^{i_{\max}} > a_{j_{\max}+1}^{i_{\max}}=a \geq y$ and $d > c_{j_{\max}}^{i_{\max}} \geq b_{j_{\max}}^{i_{\max}}$.  These inequalities force the box $(i_{\max},j_{\max})$ to be in the path of $b$, which contradicts properties previously established. This implies $a$ must have $\widehat{(i_{\max},j_{\max})}$ in its insertion path in $U \leftarrow b \leftarrow c \leftarrow a$.

\emph{Step 3:} The entry $a$ bumps the same entry in box $(i_{\max},j_{\max})$ in $U \leftarrow b \leftarrow a$ as $a$ bumps in box $\widehat{(i_{\max},j_{\max})}$ in $U \leftarrow b \leftarrow c \leftarrow a$.  Suppose for a contradiction that $a$ bumps a different entry in box $\widehat{(i_{\max},j_{\max})}$ during the insertion of $a$ into $U \leftarrow b \leftarrow c$.  This implies $\widehat{(i_{\max},j_{\max})}$ is in the path of $c$ (and by assumption in the path of $a$). But this contradicts Proposition \ref{bumping}, as the path of $a$ must be strictly rightly right of the path of $c$.

Now induct on the boxes in $I^b(a)$. Suppose by induction that the path of $a$ in $U \leftarrow b \leftarrow a$ is identical to the path of $a$ in $U \leftarrow b \leftarrow c \leftarrow a$, and the bumped entries are the same in both paths, up to some box $(i,j)$ where the path, or the value bumped, is different.  Under the inductive hypothesis the scanning values $a_s^r$ obtained when inserting $a$ into $U \leftarrow b$ are equal to the scanning values, also denoted $a_s^r$, obtained when inserting $a$ into $U \leftarrow b \leftarrow c$ up to the box $(i,j)$. We show $(i,j)$ in $U \leftarrow b \leftarrow a \leftarrow c$ is in the path of $a$ if and only if $\widehat{(i,j)}$ in $U\leftarrow b \leftarrow c$ is in the path of $a$, and $a_j^i$ bumps the same valued entry.  We do this in three steps which are identical to the three steps above.

\emph{Step 1:} If $\widehat{(i,j)}$ is in the path of $a$ in $U \leftarrow b \leftarrow c \leftarrow a$ then $(i,j)$ must be in the path of $a$ in $U \leftarrow b \leftarrow a$.  If this were not the case then the box $\widehat{(i,j-1)}$ is in the path of $c$.

Assume $j-1 \geq 2$.  Consider the diagram below, which depicts boxes $(i, j-2), (i, j-1)$, and $(i,j)$ and the corresponding boxes $\widehat{(i,j-2)}, \widehat{(i,j-1)}, \widehat{(i,j)}$.
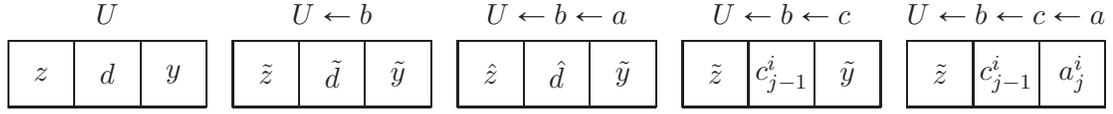
\[
\begin{array}{ccccc}
U & U \leftarrow b & U \leftarrow b \leftarrow a & U \leftarrow b  \leftarrow c & U \leftarrow b \leftarrow c \leftarrow a \\
\Tableau{ z & d & y} & \Tableau{ \tilde{z} & \tilde{d} & \tilde{y} } & \Tableau{ \hat{z} & \hat{d} & \tilde{y} } & \Tableau{\tilde{z} & c_{j-1}^i & \tilde{y} } & \Tableau{\tilde{z} & c_{j-1}^i & a_j^i}
\end{array}
\]
where either $\hat{z}$ or $\hat{d}$ could equal $a_j^i$ (but not both). We then get the inequalities $a_j^i \geq \tilde{y}$ which forces $a_j^i \geq \tilde{d}$. By Lemma \ref{scanning values} we know $b_{j-1}^i> a_j^i$, which implies $b_{j-1}^i > \tilde{d}\geq d$.  Thus the box $(i, j-1)$ is not in the path of $b$ and $\tilde{d}=d$.  On the other hand we see $\tilde{z} > c_{j-1}^i \geq b_{j-1}^i$, and since the path of $c$ cannot be strictly right of the path of $b$ we also see the box $(i, j-2)$ is not in the path of $b$ and thus $z=\tilde{z}$.  In the end we get the relations $z> b_{j-1}^i$ and $b_{j-1}^i> d$ which implies the box $(i, j-1)$ is in the path of $b$ and is a contradiction to the previously established condition on the box $(i, j-1)$.

Now we can assume $j-1=1$.  In this case the box $\widehat{(i, j-1)}=\widehat{(i,1)}$ is still in the path of $c$ and the position $\widehat{(i,j)}=\widehat{(i, 2)}$ is empty during the insertion of $a$.  With our assumption that $\widehat{(i,j)}=\widehat{(i,2)}$ is in the path of $a$ in $U \leftarrow b \leftarrow c \leftarrow a$ this implies the insertion of $b$ must have created a new box in the first column, say in position $(r, 1)$ with $r<i$.  This means position $(r, 2)$ is empty during the insertion of $a$ and by Lemma \ref{scanning values}, $b_j^r > a_j^i$ and thus $a_j^i$ must insert in position $(r,2)$. Which means $a$ cannot have $\widehat{(i,j)}=\widehat{(i,2)}$ in its path which is clearly a contradiction.

\emph{Step 2:} If $(i,j)$ is in the path of $a$ in $U \leftarrow b \leftarrow a$, then the box $\widehat{(i,j)}$ is in the path of $a$ in $U \leftarrow b \leftarrow c \leftarrow a$. As stated above, $j\neq 1$. Suppose for a contradiction that $\widehat{(i,j)}$ is not in the path of $a$. Then the box $\widehat{(i,j)}$ is in the path of $c$. As above consider the following diagram which depicts boxes $(i,j-1)$ and $(i,j)$ and the corresponding boxes $\widehat{(i,j-1)}$ and $\widehat{(i,j)}$.
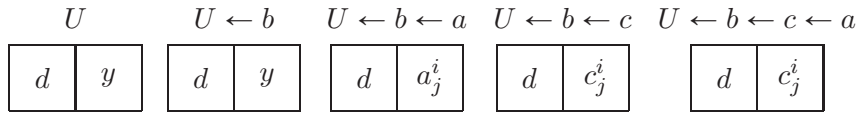
\[
\begin{array}{ccccc}
U&U \leftarrow b & U \leftarrow b \leftarrow a & U \leftarrow b  \leftarrow c & U \leftarrow b \leftarrow c \leftarrow a \\
 \Tableau{ d&y}&\Tableau{ d & y } & \Tableau{ d & a_j^i } & \Tableau{ d & c_j^i } & \Tableau{d & c_j^i}
\end{array}
\]
where neither box $(i,j)$ nor $(i,j-1)$ can be in the path of $b$ since the box $(i,j)$ is in the path of $a$ and $\widehat{(i,j)}$ is in the path of $c$. From our assumptions we get the inequalities $d > a_j^i \geq y$ and $d > c_j^i>a_j^i$.  Now consider the scanning values obtained during the insertion of $b$.  Lemma \ref{scanning values} implies $b_j^i > a_{j+1}^i\geq a_j^i \geq y$ and $d > c_j^i \geq b_j^i$.  These inequalities force the box $(i,j)$ to be in the path of $b$, which contradicts properties previously established. This implies $a$ must have $\widehat{(i,j)}$ in its insertion path in $U \leftarrow b \leftarrow c \leftarrow a$.

\emph{Step 3:} The values bumped by $a_j^i$ is the same in both $U \leftarrow b \leftarrow a$ and $U\leftarrow b \leftarrow c\leftarrow a$.  If this were not the case, then both $a$ and $c$ would have the box $\widehat{(i,j)}$ in their respective paths, which violates Proposition \ref{bumping}.

\end{proof}

\begin{remark}{\label{knuth-remark}}
In addition to showing $U \leftarrow b \leftarrow c \leftarrow a = U \leftarrow b \leftarrow a \leftarrow c$, the proof above shows that the new boxes added by $a, b,$ and $c$ occupy the same corresponding positions in each of $U \leftarrow b \leftarrow c \leftarrow a$ and $U \leftarrow b \leftarrow a \leftarrow c$.
\end{remark}

\section{A Littlewood-Richardson Type Rule}{\label{sec:LRrule}}

In this section we state and prove the main result of this chapter, which is

\begin{theorem}{\label{LRrule}}
Let $s_\lambda$ be the Schur function indexed by the partition $\lambda$, and let $\rsqschur_\alpha$ be the row-strict quasisymmetric Schur function indexed by the strong composition $\alpha$. We have
\begin{equation}{\label{LR-equation}}
\rsqschur_\alpha \cdot s_\lambda = \sum_{\beta} C_{\alpha, \lambda}^\beta \rsqschur_\beta
\end{equation}
where $C_{\alpha, \lambda}^\beta$ is the number of Littlewood-Richardson skew RCT of shape $\beta/\alpha$ and content $\reverse{\lambda}$.
\end{theorem}

\begin{proof}
It suffices to give a bijection $\rho$ between pairs $[U, T]$ and $[V,S]$ where $U$ is a RCT of shape $\alpha$, $T$ a tableau of shape $\lambda^t$, $V$ is a RCT of shape $\beta$, and $S$ is a LR skew RCT of shape $\beta/\alpha$ and content $\reverse{\lambda}$. Throughout this proof, $\lambda_1=m$.

Given a pair $[U,T]$, produce a pair $\rho([U,T])=[V,S]$ in the following way.  First use the classical RSK algorithm to produce a two-line array $A$ corresponding to the pair $(T, T_{\lambda^t})$. Next, successively insert $\widecheck{A}$ into $U$ while simultaneously placing the entries of $\widehat{A}$ into the corresponding new boxes of a skew shape with original shape $\alpha/\alpha$. This clearly produces a RCT $V$ of some shape $\beta$ and a skew filling $S$ of shape $\beta/\gamma$ where $\strongof{\gamma}=\alpha$ and the content of $S$ is $\reverse{\lambda}$.

To show that the skew filling $S$ is indeed a LR skew RCT, first note that since $\widehat{A}$ is weakly decreasing, no row of $S$ will have any instance of entries that strictly increase when read left to right.  Since $A$ is a two-line array, if $i_r=i_s$ for $r \leq s$ then $j_r\leq j_s$.  Lemma \ref{scanning values} then implies that each row of $S$ has distinct entries. Thus, the rows of $S$ strictly decrease when read left to right.

Consider the portion of $A$ where $\widehat{A}$ takes the value $i$.  The corresponding entries in $\widecheck{A}$, when read from left to right, are the entries appearing in the $(m-i+1)$st column of $T$ read from bottom to top. Now consider a different portion of the two-line array $A$ where $\widehat{A}$ takes values $i$ and $i-1$. For the moment, let this portion of $A$ be denoted $A(i,i-1)$ and suppose the number of $i$'s is $r_i$ and the number of $(i-1)$'s is $r_{i-1}$, where $r_i\geq r_{i-1}$ since $\widehat{A}$ is a regular reverse lattice word.  We will let the Knuth transformation $\kone$ act on $A(i,i-1)$ by letting $\kone$ act on $\widecheck{A}(i,i-1)$ and by considering each vertical pair as a bi-letter. We will apply a sequence $\tau$ of transformations $\kone$ to $A(i,i-1)$ until $\tau[\widehat{A}(i,i-1)]$ consists of $r_i-r_{i-1}$ number of $i$'s followed by $r_{i-1}$ pairs of the form $(i,i-1)$. Such a sequence $\tau$ exists because the entry in row $r_{i-1}-k+1$ (for $1\leq k \leq r_{i-1}$) and column $m-i+2$ of $T$ is strictly less than each entry in column $m-i+1$ which appears weakly higher in $T$. If we replace $A(i,i-1)$ with $\tau[A(i,i-1)]$ in $A$ to obtain some array $B$, then Lemma \ref{knuth} and Remark \ref{knuth-remark} imply $$U \leftarrow \widecheck{B} = U \leftarrow \widecheck{A} =V,$$ and the corresponding new box created by any entry $j$ in $\widecheck{B}$ is in the same position as the new box created by the same entry $j$ in $\widecheck{A}$. The advantage of replacing $A$ with $B$ is that now Proposition \ref{bumping} can be applied to each of the $r_{i-1}$ pairs $(i,i-1)$ and their corresponding entries in $\widecheck{B}$ to imply that in any prefix of the column word of $S$, the number of $i$'s will be at least the number of $(i-1)$'s. Hence $w_{col}(S)$ is a regular reverse lattice word.

Next we check that each Type A and Type B triple in $S$ is an inversion triple.  Below are the eight possible configurations of Type A triples in the skew filling $S$.

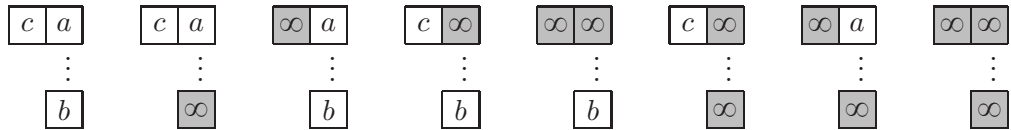
\begin{picture}(100,100)(-10,-5)

\put(0, 50){\smalltableau{c&a}}
\put(21, 36){\vdots}
\put(0, 18){\smalltableau{& b}}

\put(50, 50){\smalltableau{ c & a}}
\put(71, 36){\vdots}
\put(64, 18){\Sqone}
\put(64, 18){\smalltableau{ \infty}}

\put(100, 50){\Sqone}
\put(100,50){\smalltableau{\infty}}
\put(114, 50){\smalltableau{a}}
\put(121, 36){\vdots}
\put(100, 18){\smalltableau{& b}}

\put(150, 50){\smalltableau{c}}
\put(164, 50){\Sqone}
\put(164, 50){\smalltableau{\infty}}
\put(171, 36){\vdots}
\put(150, 18){\smalltableau{&b}}

\put(200, 50){\Sqtwo}
\put(200, 50){\smalltableau{\infty & \infty}}
\put(221, 36){\vdots}
\put(200, 18){\smalltableau{& b}}

\put(250, 50){\smalltableau{c}}
\put(264, 50){\Sqone}
\put(264, 50){\smalltableau{\infty}}
\put(271, 36){\vdots}
\put(264, 18){\Sqone}
\put(264, 18){\smalltableau{\infty}}

\put(300, 50){\Sqone}
\put(300, 50){\smalltableau{\infty}}
\put(314, 50){\smalltableau{a}}
\put(321, 36){\vdots}
\put(314, 18){\Sqone}
\put(314, 18){\smalltableau{\infty}}

\put(350,50){\Sqtwo}
\put(371, 36){\vdots}
\put(364, 18){\Sqone}
\put(350,50){\smalltableau{ \infty & \infty}}
\put(364,18){\smalltableau{\infty}}

\end{picture}

For each arrangement in the figure above, the higher row is weakly longer than the lower row because we only consider Type A triples for now.  Note that the fourth and sixth arrangements cannot exist in $S$ by Definition \ref{RCT Insertion}, and the seventh arrangement cannot exist in $S$ by Lemma \ref{weakly longer}. We can check the remaining arrangements to prove each are inversion triples.  For the first arrangement $c$ must clearly be added before $a$ and Lemma \ref{weakly longer} implies $a$ is added before $b$, which forces the relation $c > a \geq b$. The second arrangement is always an inversion triple.  In the third arrangement Lemma \ref{weakly longer} implies $a$ must have been added before $b$, hence this arrangement is an inversion triple. The fifth and eighth arrangements are always inversion triples. 

Below are the eight possible arrangements of Type B triples in the skew filling $S$.

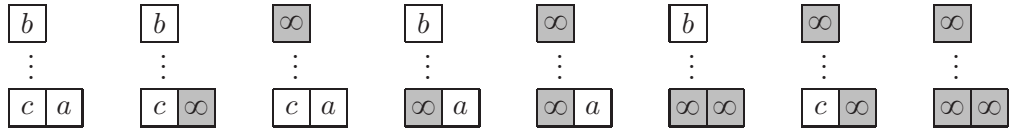
\begin{picture}(100,100)(-10,-5)

\put(0, 50){\smalltableau{b}}
\put(7, 36){\vdots}
\put(0, 18){\smalltableau{c& a}}

\put(50, 50){\smalltableau{ b}}
\put(57, 36){\vdots}
\put(64, 18){\Sqone}
\put(50, 18){\smalltableau{ c&\infty}}

\put(100, 50){\Sqone}
\put(100,50){\smalltableau{\infty}}
\put(107, 36){\vdots}
\put(100, 18){\smalltableau{c& a}}

\put(150, 50){\smalltableau{b}}
\put(157, 36){\vdots}
\put(150, 18){\Sqone}
\put(150, 18){\smalltableau{ \infty & a}}

\put(200, 50){\Sqone}
\put(200, 50){\smalltableau{ \infty}}
\put(207, 36){\vdots}
\put(200, 18){\Sqone}
\put(200, 18){\smalltableau{ \infty & a}}

\put(250, 50){\smalltableau{b}}
\put(257, 36){\vdots}
\put(250, 18){\Sqtwo}
\put(250, 18){\smalltableau{\infty&\infty}}

\put(300, 50){\Sqone}
\put(300, 50){\smalltableau{\infty}}
\put(307, 36){\vdots}
\put(314, 18){\Sqone}
\put(300, 18){\smalltableau{c & \infty}}

\put(350,50){\Sqone}
\put(350, 50){\smalltableau{\infty}}
\put(357, 36){\vdots}
\put(350, 18){\Sqtwo}
\put(350,18){\smalltableau{ \infty & \infty}}

\end{picture}

In each of the arrangements above the higher row is strictly shorter than the lower row. Note that the second and seventh arrangement cannot exist in $S$ by Definition \ref{RCT Insertion}.  For the first arrangement Lemma \ref{lastrow} implies the boxes must have been added in the order $b$, $c$, $a$ or $c$, $a$, $b$, giving the relations $b \geq c > a$ or $c > a \geq b$. The third arrangement is always an inversion triple.  For the fourth arrangement, Lemma \ref{lastrow} implies $a$ must have been added before $b$, hence this arrangement is an inversion triple.  The fifth, sixth, and eighth arrangements are always inversion triples.

Thus, we have shown the skew filling $S$ is indeed a LR skew RCT of shape $\beta/\alpha$ and content $\reverse{\lambda}$.

Given a pair $[V,S]$, produce a pair $\rho^{-1}([V,S])=[U, T]$ in the following way.  We can un-insert entries from $V$ by using $S$ as a sort of road map. Specifically, un-insert the entry in $V$ whose box is in the same position as the first occurrence (in the column reading order) of the value $1$ in $S$. This produces a pair $(1, j)$ which, when arranged as a vertical bi-letter, is the last entry of a two-line array.  Next, proceed inductively by, at the $i$th step, un-inserting each entry of $V$ which corresponds to each occurrence of the value $i$ in $S$. The row-stirctness of $S$, combined with the triple conditions imposed on $S$, ensure that after each un-insertion from $V$ the resulting figure is an RCT.

What remains after un-inserting the entries is an RCT $U$ of shape $\alpha$ since $S$ had shape $\beta/\alpha$.  The two line array produced is a valid two-line array $A$ by virtue of $w_{col}(S)$ being a regular reverse lattice word.  By RSK, $A$ corresponds to a pair $(T, T_{\lambda^t})$.  Thus we have a pair $[U, T]$ where $U$ is an RCT of shape $\alpha$ and $T$ is a tableau of shape $\lambda^t$.

\end{proof}

Figure \ref{fig-LRexample} gives an example of the bijection $\rho$ given in the proof of Theorem \ref{LRrule}.

\begin{figure}[ht]
\begin{center}
\[
\begin{array}{c}
\begin{pmatrix}
\;\tableau{  1 \\ 4 & 3 & 2 \\ 5 & 4 \\ 5 & 3 }  & , &
\tableau{  4 & 3 & 2 & 1 \\ 4 & 3 \\ 2} \; \vspace{6pt}\\ U & & T
\end{pmatrix} 
\qquad \stackrel{RSK}{\Leftrightarrow}
\vspace{12pt} \\

\begin{CD}
\begin{pmatrix} 
\smalltableau{  1 \\ 4 & 3 & 2 \\ 5 & 4 \\ 5 & 3  } &,&
\begin{picture}(110,40)
$
\begin{pmatrix}
4 & 4 & 4 & 3 & 3 & 2 & 1 \\
2 & 4 & 4 & 3 & 3 & 2 & 1
\end{pmatrix} 
$
\end{picture} 
\vspace{6pt} \\  U & & (T, T_{\lambda^t})
\end{pmatrix} @>\rho>>
\begin{pmatrix}
\smalltableau{
   1 \\
   3 & 2 \\
   4 & 3 & 2 \\
   4  \\ 
   5 & 4 & 3 & 2 & 1\\
   5 & 4 & 3   }
 &,&
\begin{picture}(80,40)(230,72.5)
\put(240,72){\Sqone}
\put(240,44){\Sqthree}
\put(240,16){\Sqtwo}
\put(240,2){\Sqtwo}
\put(222,72){\smalltableau{ \smallbas{\infty} &\infty \\
    \smallbas{\infty} & 4& 3  \\
     \smallbas{\infty} & \infty & \infty & \infty  \\ 
    \smallbas{\infty} & 4  \\
    \smallbas{\infty} & \infty & \infty &4&2&1 \\ 
    \smallbas{\infty} & \infty & \infty  & 3
}}
\end{picture}
\vspace{6pt} \\  V & & S
\end{pmatrix} 
\end{CD} 
\end{array}
\]
\caption{Example for a term of $\rsqschur_{(1,3,2,2)}\cdot s_{(3,2,1,1)}$}
\label{fig-LRexample}
\end{center}
\end{figure}

\section{A Basis for the Coinvariant Space for Quasisymmetric Functions}{\label{sec:coinv-space}}

In this section we follow \cite{Lauve2010QSym}, with one main exception.  In \cite{Lauve2010QSym} the authors are concerned with symmetric and quasisymmetric polynomials where the number of variables $n$ is finite and is at least the length of the compositions and partitions indexing the quasisymmetric and symmetric polynomials, respectively. We will continue to work with formal power series because the basis of row-strict quasisymmetric functions is not linearly independent when the number of variables $n$ is less than the degree of the function. For example, if we restrict to three variables one can easily compute $\rsqschur_{(3,2)}(x_1, x_2, x_3)=\rsqschur_{(1,3,1)}(x_1, x_2, x_3)$.  If $\alpha$ is any sequence of nonnegative integers, in particular a composition or a partition, then the problem of linear dependencies arising after restriction to a finite number of variables can be remedied by requiring $n\geq |\alpha|$, where $|\alpha|:=\sum_{k=1}^{\ell(\alpha)} \alpha_k$ is the size of alpha. If $|\alpha|=d$, we call $\alpha$ a composition (or partition) of $d$.

In this section we will consider $\Qsym$ as a module over $\Sym$. Throughout, let $\Qsym_d$ be quasisymmetric functions of homogeneous degree $d$.  Then $\Qsym=\bigoplus_{d \geq 0} \Qsym_d$. Let $(\mathcal{E})$ be the ideal in $\Qsym$ generated by the elementary symmetric functions of degree $d>0$.

The following definitions first appear in \cite{Bergeron_Reutenauer}.  Let $\alpha$ be a strong composition with largest part $m$.  We call $\alpha$ \emph{inverting} if and only if for each $1 < i \leq m$, there exists a pair of indices $s$ and $t$ with $s<t$ such that $\alpha_s =i$ and $\alpha_t=i-1$.  For example $\alpha=(1,3,1,2,3,1,2)$ is inverting but $\beta=(1,2,1,2,3,1,3)$ is not inverting.  Any composition $\alpha$ can be factored uniquely as 
\begin{equation*}{}
\alpha=(\alpha^\prime, k^{i_k}, \ldots, 2^{i_2}, 1^{i_1}) \; ,\qquad i_j \geq 1
\end{equation*}
where the prefix $\alpha^\prime$ has no parts of size $1, 2, \ldots, k$.  We call a composition $\alpha$ \emph{pure} if and only if $k$ is even.  Note that if the last part of $\alpha$ is not $1$, then $k=0$ and hence is pure.  As an example, $\alpha=(5,4,3,5,2,1,1)$ is pure with $k=2$ but $\beta=(5,4,3,5,1)$ is not pure because $k=1$.

Define $B$ to be the set of pure and inverting compositions and 
\begin{equation*}
\mathfrak{C}_d:=\{ s_\lambda \rsqschur_\alpha \mid |\lambda|+|\alpha|=d, \alpha \in B \}. 
\end{equation*}
Then we have

\begin{proposition}{\label{prop:coinvariant}}
The collection $\mathfrak{C}:=\bigsqcup_{d \geq 0} \mathfrak{C}_d$ is a basis for $\Qsym$.
\end{proposition}

Before we prove this result, we need to establish some preliminaries.  As in  \cite{Lauve2010QSym}, define
\[
\begin{array}{rcl}
PB_d&:=&\{ (\lambda, \alpha) \mid \lambda \text{ a partition, } \alpha \in B, |\lambda|+|\alpha|=d \} \\
C_d &:=& \{ \beta \mid \beta \text{ a composition, } |\beta|=d \}.
\end{array}
\]

Define a map $\phi \colon PB_d \rightarrow C_d$ as follows.  Let $(\lambda, \alpha) \in PB_d$.  Then $\phi((\lambda,\alpha))$ is the composition obtained by adding $\lambda^t_i$ (the $i$th part of $\lambda^t$) to the $i$th largest part of $\alpha$ for all $1 \leq i \leq \ell(\lambda)$.  If $\alpha_j=\alpha_k$ for some $j <k$, then we consider $\alpha_k$ larger than $\alpha_j$.  In the case where $\ell(\alpha) < \ell(\lambda^t)$, append $\ell(\lambda^t)-\ell(\alpha)$ zeros to $\alpha$ to make the two sequences the same length. See Figure \ref{fig:mapphi} for an example.

\begin{figure}[ht]{\label{fig:mapphi}}
\[
\begin{array}{rcl}
\lambda &=& (3,3,2,2,2,1,1) \\
\alpha &=& (1,3,1,2,3,1,2) \\
\phi((\lambda,\alpha)) &=& (1, 3+5, 1, 2, 3+7, 1, 2+2) =(1,8, 1,2,10,1,4)
\end{array}
\]
\caption{$\phi \colon PB_{27} \to C_{27}$}
\end{figure}

The following proposition was given in \cite{Lauve2010QSym} in the case where the lengths of each sequence in question were at most $n$, their motivation being to work with symmetric and quasisymmetric polynomials in variables $x_1, x_2, \ldots, x_n$. Since the authors' proof is given for arbitrary $n$ and $d$, the proposition holds in our situation for the infinite set of indeterminates $\xx=(x_1, x_2, \ldots)$.

\begin{proposition}(\cite{Lauve2010QSym})
The map $\phi$ is a bijection between $PB_d$ and $C_d$.
\end{proposition}

As in \cite{Lauve2010QSym} we will need the so called \emph{lexrev} order on compositions. First recall the lexicographic order $\lex$ on partitions of $n$, which states that $\lambda \lex \mu$ if and only if for some $k$, $\lambda_i=\mu_i$ for all $1 \leq i <k$ and $\lambda_k > \mu_k$.  Note that $\lex$ is defined on compositions as well. For compositions $\alpha$ and $\beta$ of $n$ we say $\alpha$ is greater than $\beta$ in the lexrev order, and write $\alpha \succeq \beta$ if and only if
\begin{enumerate}
\item $\partitionof{\alpha} \lex \partitionof{\beta},$ or
\item $\partitionof{\alpha}=\partitionof{\beta}$ and $\reverse{\alpha}$ is lexicographically greater than $\reverse{\beta}$.
\end{enumerate}

For example, if we take the eight compositions of $4$, we see
\begin{equation*}
(4) \succeq (1,3) \succeq (3,1) \succeq (2,2) \succeq (1,1,2) \succeq (1,2,1) \succeq (2,1,1) \succeq (1,1,1,1).
\end{equation*}

To prove Proposition \ref{prop:coinvariant} we will compute the leading term of $s_\lambda \rsqschur_\alpha$ relative to the lexrev order.  To do this, we will need a distinguished Littlewood-Richardson skew RCT of shape $\phi((\lambda, \alpha))/\alpha$ which we call, as in \cite{Lauve2010QSym}, the \emph{super filling} and denote it by $SU(\lambda, \alpha)$. We construct $SU(\lambda,\alpha)$ as follows.  

Let $\alpha$ be a strong composition and $\lambda$ a partition.  If $\ell(\lambda^t) > \ell(\alpha)$ append zeros to the end of $\alpha$ so that the resulting composition has the same length as $\lambda^t$. If $\lambda^t$ and $\alpha$ satisfy $\ell(\lambda^t) \leq \ell(\alpha)$ then no action is needed.   As in the definition of LR skew RCT, the cells of $\alpha$ will be filled with virtual $\infty$ symbols with the convention that given two boxes filled with $\infty$, if they are in the same row we define these entries to strictly decrease left to right, while two such boxes in the same column are defined to be equal. Recall that for an arbitrary partition $\mu$ with largest part $m$, $T_\mu$ is the (reverse row-strict) tableau that has the entire $i$th column filled with the entry $(m+1-i)$ for all $1 \leq i \leq m$. Now append the $i$th row of $T_{\lambda^t}$ to the $i$th longest row of $\alpha$.  If two rows of $\alpha$ are the same length, then the lower row is considered longer.

For example, if $\alpha=(1,3,1,2,3,1,2)$ and $\lambda=(3,3,2,2,2,1,1)$ then $SU(\lambda, \alpha)$ is
\[
\begin{picture}(100,130)(39,0)
\put(0,105){\sqone}
\put(0,87){\sqthree}
\put(0,69){\sqone}
\put(0, 51){\sqtwo}
\put(0, 33){\sqthree}
\put(0, 15){\sqone}
\put(0,-3){\sqtwo}
\put(0,105){\tableau{ \infty \\
		\infty & \infty & \infty &7 &6 &5 &4 & 3\\
		\infty \\
		\infty & \infty\\
		\infty & \infty & \infty & 7 & 6 & 5 & 4 & 3 & 2 & 1 \\
		\infty \\
		\infty &\infty & 7 & 6}}
\end{picture}
\]

\begin{lemma}{\label{lem:super}}
The super filling $SU(\lambda, \alpha)$ is a LR skew RCT of shape $\phi((\lambda, \alpha))/\alpha$ and content $\reverse{\lambda}$.
\end{lemma}
\begin{proof}
By construction $SU(\lambda, \alpha)$ has shape $\phi((\lambda, \alpha))/\alpha$ and content $\reverse{\lambda}$.  Clearly, the entries in the rows of $SU(\lambda, \alpha)$ which are not $\infty$ strictly decrease. Thus we only have to show that every Type A and Type B triple is an inversion triple, and that the column reading word is a regular reverse lattice word.  Throughout this proof, let $\beta$ be the shape of $SU(\lambda, \alpha)$.

Consider a Type A triple in rows $i_1$ and $i_2$, where $i_1 < i_2$ and $\beta_{i_1} \geq \beta_{i_2}$. Below are the only possible configurations of Type A triples in $SU(\lambda, \alpha)$:

\begin{picture}(100,100)(-10,-5)

\put(0, 50){\smalltableau{c&a}}
\put(21, 36){\vdots}
\put(0, 18){\smalltableau{& b}}

\put(50, 50){\smalltableau{ c & a}}
\put(71, 36){\vdots}
\put(64, 18){\Sqone}
\put(64, 18){\smalltableau{ \infty}}

\put(100, 50){\Sqone}
\put(100,50){\smalltableau{\infty}}
\put(114, 50){\smalltableau{a}}
\put(121, 36){\vdots}
\put(100, 18){\smalltableau{& b}}

\put(150, 50){\smalltableau{c}}
\put(164, 50){\Sqone}
\put(164, 50){\smalltableau{\infty}}
\put(171, 36){\vdots}
\put(150, 18){\smalltableau{&b}}

\put(200, 50){\Sqtwo}
\put(200, 50){\smalltableau{\infty & \infty}}
\put(221, 36){\vdots}
\put(200, 18){\smalltableau{& b}}

\put(250, 50){\smalltableau{c}}
\put(264, 50){\Sqone}
\put(264, 50){\smalltableau{\infty}}
\put(271, 36){\vdots}
\put(264, 18){\Sqone}
\put(264, 18){\smalltableau{\infty}}

\put(300, 50){\Sqone}
\put(300, 50){\smalltableau{\infty}}
\put(314, 50){\smalltableau{a}}
\put(321, 36){\vdots}
\put(314, 18){\Sqone}
\put(314, 18){\smalltableau{\infty}}

\put(350,50){\Sqtwo}
\put(371, 36){\vdots}
\put(364, 18){\Sqone}
\put(350,50){\smalltableau{ \infty & \infty}}
\put(364,18){\smalltableau{\infty}}

\end{picture}

Note that the fourth and sixth arrangements are not possible in our construction.  We see that in the seventh arrangement we must have $\alpha_{i_1} < \alpha_{i_2}$, which means we append cells first to the row $i_2$ and then append weakly fewer cells to row $i_1$.  Thus it is impossible to have $\beta_{i_1} \geq \beta_{i_2}$, so the seventh arrangement is impossible.

Next we check the remaining arrangements are all inversion triples. In each arrangement, $a=c-1$, so that in the first arrangement we must have either $b \geq c >a$ or $c>a \geq b$. The second arrangement is clearly an inversion triple. In the third arrangement if $\alpha_{i_1}=\alpha_{i_2}$ then $a=b$ and we have an inversion triple.  If on the other hand $\alpha_{i_1} >\alpha_{i_2}$, then $a >b$ and we have an inversion triple. (As was previously mentioned, it is impossible for $\alpha_{i_1} < \alpha_{i_2}$.)  The fifth and eight arrangements are clearly inversion triples.

Consider a Type B triple in rows $i_1$ and $i_2$, where $i_1 < i_2$ and $\beta_{i_1} < \beta_{i_2}$. Below are the eight possible arrangements of Type B triples in $SU(\lambda, \alpha)$:

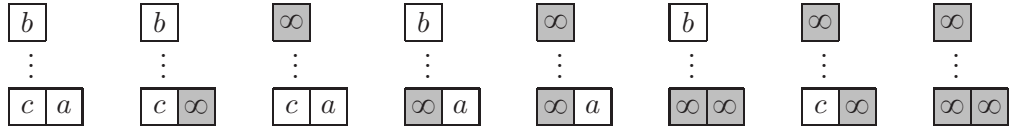
\begin{picture}(100,100)(-10,-5)

\put(0, 50){\smalltableau{b}}
\put(7, 36){\vdots}
\put(0, 18){\smalltableau{c& a}}

\put(50, 50){\smalltableau{ b}}
\put(57, 36){\vdots}
\put(64, 18){\Sqone}
\put(50, 18){\smalltableau{ c&\infty}}

\put(100, 50){\Sqone}
\put(100,50){\smalltableau{\infty}}
\put(107, 36){\vdots}
\put(100, 18){\smalltableau{c& a}}

\put(150, 50){\smalltableau{b}}
\put(157, 36){\vdots}
\put(150, 18){\Sqone}
\put(150, 18){\smalltableau{ \infty & a}}

\put(200, 50){\Sqone}
\put(200, 50){\smalltableau{ \infty}}
\put(207, 36){\vdots}
\put(200, 18){\Sqone}
\put(200, 18){\smalltableau{ \infty & a}}

\put(250, 50){\smalltableau{b}}
\put(257, 36){\vdots}
\put(250, 18){\Sqtwo}
\put(250, 18){\smalltableau{\infty&\infty}}

\put(300, 50){\Sqone}
\put(300, 50){\smalltableau{\infty}}
\put(307, 36){\vdots}
\put(314, 18){\Sqone}
\put(300, 18){\smalltableau{c & \infty}}

\put(350,50){\Sqone}
\put(350, 50){\smalltableau{\infty}}
\put(357, 36){\vdots}
\put(350, 18){\Sqtwo}
\put(350,18){\smalltableau{ \infty & \infty}}

\end{picture}

Note the second and seventh arrangements are impossible in our construction. In each arrangement we again have $a=c-1$, so in the first arrangement we must have either $b \geq c > a$ or $c >a \geq b$.  The third arrangement is clearly an inversion triple. In the fourth arrangement cells were appended to row $i_2$ before row $i_1$, so $b<a$ and we have an inversion triple. The fifth, sixth, and eighth arrangements are clearly inversion triples.

The last property to check is that the column reading word $w_{col}(SU(\lambda, \alpha))$ is a regular reverse lattice word.  Clearly, the first entry in $w_{col}(SU(\lambda, \alpha))$ is $m$, which is the largest part of $\lambda^t$.  By construction the entries in the rows of $SU(\lambda, \alpha)$ are $m, m-1, m-2, \ldots$, so prior to reading any entry $j<m$ in row $i$, one must have read the entry $j+1$ that is immediately to the left of $j$ in row $i$. Thus the column reading word is a reverse lattice word.  Also, the column reading word contains at least one $1$ by construction.  Therefore the column reading word is a regular reverse lattice word.
\end{proof}

\begin{proof}[Proof of Proposition \ref{prop:coinvariant}]
Order the compositions of $d$ by the lexrev order.  This puts an order on the basis elements $\rsqschur_\beta$ of $\Qsym_d$.  Order the elements of $\mathfrak{C}_d$ by mapping the pair $(\lambda, \alpha)$ under $\phi$ to a composition of $d$. We claim that the leading term in the row-strict quasisymmetric Schur expansion of $s_\lambda \rsqschur_\alpha$ is the function $\rsqschur_{\phi((\lambda, \alpha))}$.

First notice that $\rsqschur_{\phi((\lambda, \alpha))}$ is a term in the expansion of $s_\lambda \rsqschur_\alpha$, since by Lemma \ref{lem:super} the super filling $SU(\lambda, \alpha)$ is a LR skew RCT of shape $\phi((\lambda,\alpha))/\alpha$ and content $\reverse{\lambda}$. Thus it suffices to show that the composition $\beta=\phi((\lambda,\alpha))$ is the largest composition appearing in the expansion of $s_\lambda \rsqschur_\alpha$ with respect to the lexrev order.

Given $\lambda$ and $\alpha$, to construct the largest possible composition that might appear in the expansion of $s_\lambda \rsqschur_\alpha$ we first have to append as many cells as possible to the longest row of $\alpha$, taking into account that if two rows of $\alpha$ have equal length then the lower row is considered longer.  The last entry in this row must be $1$ since the column reading word is a regular reverse lattice word.  The entries of this row must also strictly decrease from left to right and the maximum entry in the row must be at most $L=\ell(\lambda)$, because the content of the filling is $\reverse{\lambda}$.  Therefore, to append as many cells as possible to the longest row of $\alpha$, one must append $\lambda_1^t$ cells filled with $L, L-1, \ldots, 2,1$.

Similarly, to append the maximum possible number of cells to the $i$th longest part of $\alpha$ (again, insuring we have the largest possible composition in lexrev order), one must append $\lambda_i^t$ cells filled with entries $L, L-1, \ldots, j$ where $j$ is the minimum positive entry in $\reverse{\lambda}-(i^{\ell(\lambda)})$.  If $\ell(\lambda^t) > \ell(\alpha)$, then append the extra parts of $\lambda^t$ from least to greatest and from top to bottom, after the last row of $\alpha$.  The resulting shape, which is largest in the lexrev order by construction, is precisely $\phi((\lambda, \alpha))$.

Since $\phi$ is a bijection and we have a triangular decomposition of the elements in $\mathfrak{C}_d$ in terms of the basis $\rsqschur_\beta$, we therefore have $\mathfrak{C}_d$ is a basis for $\Qsym_d$.  Notice that in constructing the largest possible composition $\beta$, the entries filling the shape $\phi((\lambda, \alpha))$ were uniquely determined.  Thus the super filling $SU(\lambda, \alpha)$ is the only LR skew RCT of that shape. Therefore, the transition matrix between $\mathfrak{C}_d$ and the basis $\rsqschur_\beta$ is uni-uppertriangular.

If we now union over the degree $d$, we have the required basis for $\Qsym$.
\end{proof}

From \cite{Lauve2010QSym} and \cite{Mason2011qsym-emails} we can established that the only row-strict quasisymmetric Schur functions $\rsqschur_\beta$ that appear in $\Sym$ are those where $\beta$ is of rectangular shape.  The only rectangles which are inverting are those of the form $(1^k)$, but these rectangles are not pure. Thus as a consequence of Proposition \ref{prop:coinvariant} we have the following.

\begin{corollary}{\label{cor:coinvariant}}
The set $\{\rsqschur_\alpha \mid \alpha \text{ pure and inverting}\}$ is a basis for $\Qsym/(\mathcal{E})$.
\end{corollary}

    \chapter[%
      Demazure Atoms
   ]{%
      Characterizations of Demazure Atoms
   }%
   \label{ch:3ndChapterLabel}
  In this Chapter we present six equivalent characterizations of Demazure atoms, four of which are known in the literature and two of which are new. The chapter is organized as follows.  In Section \ref{sec:divdiff} we record two characterizations of Demazure atoms first given in \cite{Lascoux1990Keys}, one of which is Definition \ref{def:atoms} which we use later in the chapter.   In Section \ref{sec:colstrictcomp} we present the characterization of Demazure atoms given in \cite{Mason2009An-explicit}, and we also record a bijection from \cite{Haglund2008Quasisymmetric} between semi-standard augmented fillings and certain column-strict composition tableaux; this bijection yields yet another characterization. We use this bijection in Section \ref{sec:comparray} to give our first new characterization of Demazure atoms. Specifically, we present Definition \ref{def:CTarraypattern}, which is a new Gelfand-Tsetlin type triangular array of nonnegative integers which we call composition array patterns. We show in Theorem \ref{thm-CT-array} that there is a bijection between certain column-strict composition tableaux and composition array patterns.  In Theorem \ref{thm:GTandCTbiject} we give a bijection $\Theta$ between composition array patterns of shape $\gamma$ and Gelfand-Tsetlin patterns of shape $\partitionof{\gamma}:=\lambda$. Finally, in Section \ref{sec:LSpaths} we use Lakshmibai-Seshadri paths to give our second characterization of Demazure atoms.

Throughout this chapter, all Young tableaux will follow the English convention.  Recall that a \emph{word} $w=w_1w_2\cdots w_n$ is any finite sequence of positive integers.  We will frequently use the Bruhat order on permutations of the symmetric group $S_n$ on $n$ letters \cite{Humphreys1990Reflection}. Let $\alpha_i$ be the simple root which has a $1$ in position $i$ and a $-1$ in position $i+1$ for $1 \leq i \leq n-1$.

\section{Keys and Divided Difference Operators}{\label{sec:divdiff}}

Demazure atoms first appeared in \cite{Lascoux1990Keys} under the name ``standard bases." In this section we present two characterizations of Demazure atoms from \cite{Lascoux1990Keys}.

\subsection{Divided Difference Operators}

Fix $n \in \Z_{\geq 0}$.  Let $P=\Z[x_1, x_2, \ldots, x_n]$ and let $S_n$ be the symmetric group on $[n]$.  For $1 \leq i < n$ define linear operators on $P$ by the formulas
\begin{equation}
\partial_i = \frac{1-s_i}{x_i-x_{i+1}}, \; \text{ and } \; \pi_i=\partial_i x_i.
\end{equation}
These operators obey the following relations.

\begin{align}{\label{relations}}
\partial_i^2&= 0 \notag \\
\partial_i \partial_j &= \partial_j \partial_i \; \; \text{ for } |i-j|>1 \notag \\
\partial_i \partial_{i+1} \partial_{i} &= \partial_{i+1} \partial_i \partial_{i+1} \\
\pi_i^2 &= \pi_i  \notag \\
\pi_i \pi_j &= \pi_j \pi_i \; \; \text{ for } |i-j|>1  \notag \\
\pi_i \pi_{i+1} \pi_i &= \pi_{i+1} \pi_i \pi_{i+1} \notag
\end{align}
The relations $\pi_i \pi_j = \pi_j \pi_i$ for $|i-j|>1$ and $\pi_i \pi_{i+1} \pi_{i} = \pi_{i+1} \pi_i \pi_{i+1}$ together are called \emph{braid relations}.

Let $\tau \in S_n$ and let $\tau= s_{i_1} \cdots s_{i_k}$ be a reduced word of $\tau$.  Define
\begin{equation}{\label{eq:demazure}}
\pi_\tau = \pi_{i_1} \cdots \pi_{i_k}.
\end{equation}
Due to \ref{relations} and the fact that any two reduced words are connected by a sequence of Coxeter relations, (\ref{eq:demazure}) is well defined.  The operator $\pi_\tau$ is called a \emph{Demazure operator.}

Given any partition $\lambda$ of length at most $n$, we can append zeros to the end of $\lambda$ so that the resulting sequence has length equal to $n$. We may then define $\xx^\lambda:=x_1^{\lambda_1}\cdots x_n^{\lambda_n}$; we briefly note that $\xx^\beta$, for any composition $\beta$, is defined similarly.  The polynomial defined by $\pi_\tau(\xx^\lambda)$ is the (type $A$) \emph{Demazure character} corresponding to the dominant weight $\lambda$ and permutation $\tau$.

For a fixed partition $\lambda$, the set of monomials in the Demazure characters $\pi_\tau(\xx^\lambda)$ and $\pi_\omega(\xx^\lambda)$ may in general intersect nontrivially.  This fact is clear from the definition of $\pi_\tau$, since if $\omega < \tau$ in Bruhat order then the set of monomials in $\pi_\omega(\xx^\lambda)$ are a subset of the monomials in $\pi_\tau(\xx^\lambda)$.  For example when $n=3$ and $\lambda=(2,1)$,
\begin{align*}
\pi_1\pi_2(x_1^2x_2) &=x_1^2x_2+x_1^2x_3+x_1x_2^2+x_1x_2x_3+x_2^2x_3 \\
\pi_2(x_1^2x_2) &= x_1^2 x_2 +x_1^2x_3.
\end{align*}
We see that the monomial $x_1^2x_3$, which has weight $(2,0,1)$, appears in both $\pi_1\pi_2(x_1^2x_2)$ and $\pi_2(x_1^2x_2)$ even though the weight space with weight $(2,0,1)$ is one dimensional.  

This motivates replacing the operator $\pi_i$ with $\bar{\pi}_i := \pi_i -1$. The operators $\bar{\pi}_i$ satisfy $\bar{\pi}_i^2=-\bar{\pi}_i$ and the braid relations, so $\bar{\pi}_\tau$ is still well defined.  This leads to the following definition.

\begin{definition}{\label{def:atoms}}
The polynomials $\bar{\pi}_\tau(\xx^\lambda)$, for $\lambda$ a partition and $\tau$ a permutation, are called \emph{Demazure atoms.}
\end{definition}

\begin{remark}
Let $\Stab(\lambda)$ be the stabilizer of $\lambda$ under the action of $S_n$ on the parts of $\lambda$, where $\ell(\lambda) \leq n$.  If $s_i \in \Stab(\lambda)$ then $\bar{\pi}_i(\xx^\lambda)=0$. There is a unique element $\tau$ of minimal length in $S_n/\Stab(\lambda)$ taking $\lambda$ to the weak composition $\gamma= \tau(\lambda)$.  Henceforth, we will implicitly assume we are always taking $\tau \in S_n/\Stab(\lambda)$ a minimal length coset representative and we will denote $\dzatom_{\tau(\lambda)}(x_1, \ldots, x_n) := \bar{\pi}_\tau(\xx^\lambda) $.
\end{remark}

Given a monomial $\xx^\beta$ we can describe the action of $\bar{\pi}_i$ as follows.
\begin{align}{\label{barpiaction}}
\bar{\pi}_i(\xx^\beta) &= \xx^{\beta-\alpha_i}+\xx^{\beta-2\alpha_i} + \cdots + \xx^{\beta-k\alpha_i} &&\text{ if } k=\beta_i-\beta_{i+1} >0,  \notag \\
\bar{\pi}_i(\xx^\beta) &= 0  && \text{ if } \beta_i=\beta_{i+1}, \\
\bar{\pi}_i(\xx^\beta) &= -(\xx^\beta+ \xx^{\beta + \alpha_i} \cdots + \xx^{\beta + (k-1)\alpha_i} ) &&\text{ if } k=\beta_{i+1}-\beta_i >0. \notag
\end{align}
Because the Demazure atom is defined as $\bar{\pi}_\tau(\xx^\lambda)$ where $\lambda$ is a partition, using the description of the action of $\bar{\pi}_i$ in (\ref{barpiaction}) we can see that Demazure atoms have positive integral coefficients. This is also apparent by using the equivalent definition of Demazure atoms in Lemma \ref{lem:atomaskey} below. 

The following lemma will be used in Section \ref{sec:LSpaths}. Let $[\xx^\beta]f$ be the coefficient of $\xx^\beta$ in the polynomial $f$.  If $[\xx^\beta]f \neq 0$ then we will say the monomial $\xx^\beta$ \emph{is in} the polynomial $f$.  In the case of Demazure atoms, if $\xx^\beta$ is in $\bar{\pi}_\tau(\xx^\lambda)$ then $[\xx^\beta]\bar{\pi}_\tau(\xx^\lambda)=c>0$.

\begin{lemma}{\label{lem:monomials}}
Let $\tau$ be a fixed permutation and let $\gamma$ be a fixed weak composition.  Let $i$ be such that $\ell(s_i \tau) > \ell(\tau)$.  We have $[\xx^\gamma]\bar{\pi}_{s_i \tau}(\xx^\lambda)=c > 0$ if and only if each of the following conditions are met.
\begin{enumerate}
\item $[\xx^\gamma]\bar{\pi}_i(\xx^\beta)=1$ for some $\xx^\beta$ in $\bar{\pi}_\tau(\xx^\lambda)$. Hence $\gamma=\beta-t\alpha_i$ for some $0 < t \leq (\beta_i-\beta_{i+1})$,
\item $\sum [\xx^\beta] \bar{\pi}_\tau(\xx^\lambda) - \sum [\xx^\mu]\bar{\pi}_\tau(\xx^\lambda) = c$.
\end{enumerate} 
where the first sum is over all $\beta$ such that $\xx^\beta$ is in $\bar{\pi}_\tau(\xx^\lambda)$ and $[\xx^\gamma]\bar{\pi}_i(\xx^\beta)=1$, and the second sum is over all  $\mu$ such that $\xx^\mu$ is in $\bar{\pi}_\tau(\xx^\lambda)$ and $\mu$ has the form 
\begin{itemize}
\item if $\gamma_i \geq \gamma_{i+1}$, then $\mu=s_i(\gamma)-r\alpha_i$ for some $0< r$, 
\item  if $\gamma_i < \gamma_{i+1}$, then $\mu=\gamma-r\alpha_i$ for some $0 \leq r$.
\end{itemize}
\end{lemma}
\begin{proof}
Suppose $\xx^\gamma$ is in $\bar{\pi}_{s_i \tau}(\xx^\lambda)$ with some coefficient $c$, where $c$ is necessarily positive.  By the definition of $\bar{\pi}_i$ and the description of its action given in (\ref{barpiaction}), we must have $\gamma=\beta-t\alpha_i$ for some $\beta$ such that $\beta_i > \beta_{i+1}$ and $\xx^\beta$ in $\bar{\pi}_\tau(\xx^\lambda)$.  Clearly, $0 < t \leq (\beta_i - \beta_{i+1})$.

To show (2), first suppose $\gamma_i \geq \gamma_{i+1}$.  Clearly $\sum [\xx^\beta]\bar{\pi}_\tau(\xx^\lambda) \geq c$, where the sum is over all $\beta$ such that $\xx^\beta$ is in $\bar{\pi}_\tau(\xx^\lambda)$ and $[\xx^\gamma]\bar{\pi}_i(\xx^\beta)=1$. The only monomials $\xx^\mu$ in $\bar{\pi}_\tau (\xx^\lambda)$ which have $[\xx^\gamma]\bar{\pi}_i(\xx^\mu)=-1$ are those monomials with $\mu= s_i(\gamma)-r\alpha_i$ and $0<r$.  Similarly, if $\gamma_i < \gamma_{i+1}$, the only monomials $\xx^\mu$ in $\bar{\pi}_\tau (\xx^\lambda)$ which have $[\xx^\gamma]\bar{\pi}_i(\xx^\mu)=-1$ are those monomials with $\mu= \gamma-r\alpha_i$ and $0\leq r$. Thus if $[\xx^\gamma]\bar{\pi}_{s_i \tau}(\xx^\lambda)=c$, we must have $\sum [\xx^\beta] \bar{\pi}_\tau(\xx^\lambda) - \sum [\xx^\mu]\bar{\pi}_\tau(\xx^\lambda) = c$ in both cases.

Conversely, suppose we have (1) and (2) of the lemma.  Since each monomial $\xx^\beta$ in the sum $\sum [\xx^\beta] \bar{\pi}_\tau(\xx^\lambda)$ satisfies $[\xx^\gamma]\bar{\pi}_i(\xx^\beta)=1$, and each monomial $\xx^\mu$ in the sum $\sum [\xx^\mu]\bar{\pi}_\tau(\xx^\lambda)$ satisfies $[\xx^\gamma]\bar{\pi}_i(\xx^\mu) = -1$, we conclude $[\xx^\gamma]\bar{\pi}_{s_i \tau}(\xx^\lambda)=c$.
\end{proof}

\begin{remark}
Intuitively, Lemma \ref{lem:monomials} says that $\xx^\gamma$ is in $\bar{\pi}_{s_i \tau}(\xx^\lambda)$ if and only if $\xx^\gamma$ is in $\bar{\pi}_i(\xx^\beta)$ for some $\xx^\beta$ in $\bar{\pi}_\tau(\xx^\lambda)$, and $\xx^\gamma$ is not cancelled by terms coming from $\bar{\pi}_i(\xx^\mu)$ for some $\xx^\mu$ in $\bar{\pi}_\tau(\xx^\lambda)$.
\end{remark}

\subsection{Keys} 
For the remainder of this section we will abandon the convention from Chapter \ref{ch:2ndChapterLabel} of using reverse tableaux.  Thus all Young tableaux in this section will have entries which weakly increase from left to right in its rows, and strictly increase from top to bottom in its columns. We will continue to use the English convention of drawing tableaux.  

In \cite{Lascoux1990Keys}, the authors use the notion of a \emph{key}, which is a Young tableau whose sets of column entries are ordered by containment, to give an alternate characterization of Demazure atoms $\dzatom_{\tau(\lambda)}$. The definitions we present here follow \cite{Lascoux1990Keys} and \cite{Reiner1995Key-polynomials}.

\begin{definition}
A \emph{key} is a semi-standard Young tableau such that the set of entries in the $(j+1)$st column form a subset of the set of entries in the $j$th column, for all $j$.
\end{definition}

There is an obvious bijection between weak compositions and keys given by $\gamma=(\gamma_1, \ldots, \gamma_n) \to key(\gamma)$, where $key(\gamma)$ is the key of shape $\partitionof{\gamma}$ whose first $\gamma_j$ columns contain the letter $j$, for all $j$. For example, if $\gamma=(1,0,3,2,0,1)$, then
\begin{equation*}
key(\gamma)=\tableau{1 & 3 & 3 \\ 3 & 4 \\ 4 \\ 6} \; .
\end{equation*}
The inverse of this map is given by $T \to \cont(T)$, sending a key $T$ to its content.

Knuth equivalence plays a crucial role in this section, so we reprise its definition from Section \ref{subsec:knuth}.

\begin{definition}
Let $a,b,$ and $c$ be positive integers.  Then
\begin{displaymath}
\begin{array}{lll}
\mathcal{K}_1: & bca \to bac & \text{ if } a<b \leq c \\
\mathcal{K}_2: & acb \to cab & \text{ if } a \leq b < c 
\end{array}.
\end{displaymath}
\end{definition}

The relations $\kone, \ktwo,$ and their inverses $\kone^{-1}, \ktwo^{-1}$, act on words $w$ by transforming triples of consecutive letters.  We will say two words $w$ and $w^\prime$ are \emph{Knuth equivalent}, and write $w \cong w^\prime$, if and only if $w$ can be transformed to $w^\prime$ through a sequence of transformations using only $\kone, \ktwo, \kone^{-1},$ and $\ktwo^{-1}$.

Let $T$ be a Young tableau and let $w_{col}(T)$ be the \emph{column reading word} of $T$, which is obtained by reading the entries of $T$ in each column from bottom to top and from left to right.  We will occasionally write $w_{col}(T)=v^{(1)}v^{(2)} \cdots$ where each $v^{(j)}$ is the strictly decreasing word comprising the $j$th column of $T$. In general any word which strictly decreases will be called a \emph{column word.} Similarly, let $w_{row}(T)$ be the \emph{row reading word} of $T$, which is obtained by reading the entries of $T$ in each row from left to right and from bottom to top.  We will write $w_{row}(T)=\cdots u^{(2)}u^{(1)}$ where each $u^{(i)}$ is the weakly increasing word comprising $i$th row of $T$. In general any word which weakly increases will be called a \emph{row word}. A standard fact about Knuth equivalence \cite{Stanley1999EC2} is that there is a unique word $v$ in each Knuth equivalence class such that $v=w_{col}(T)$ for some Young tableau $T$.  

Let $w$ be a word.  The \emph{column word factorization} of $w$, written $w=w^{(1)}w^{(2)}\cdots$, is the factorization where each $w^{(j)}$ is a maximal column word.  Denote by $\colform(w)$ the \emph{column form} of a word $w$, which is the composition whose $j$th part is the length $w^{(j)}$.  For example if $w=134214$, then $w=1\cdot 3 \cdot 421 \cdot 4$ is its column word factorization, and $\colform(w)=(1,1,3,1)$.

If a word $w$ is equivalent to $w_{col}(T)$ we will write $w \cong T$.  Let $w$ be an arbitrary word such that $w \cong T$ where the shape of $T$ is $\lambda$. The word $w$ is called \emph{column-frank} if $\partitionof{\colform(w)}=\lambda^t$; that is, if $\colform(w)$ is a rearrangement of the parts of $\lambda^t$.

\begin{definition}{\label{def:leftandrightkey}}(\cite{Reiner1995Key-polynomials})
Let $T$ be a Young tableau of shape $\lambda$.  The \emph{right key} of $T$, denoted $K_+(T)$, is the key of shape $\lambda$ whose $j$th column is given by the last column word of any column-frank word $v$ such that $v \cong T$ and $\colform(v)=(\ldots, \lambda_j^t)$.
\end{definition}

\begin{example}
Let $\lambda=(2,2,1,1)$.  Then $\lambda^t=(4,2)$. If 
\begin{equation*}
T=\tableau{1 & 2 \\ 2 & 4 \\ 3 \\ 5}
\end{equation*}
then $w_{col}(T)=532142$.  To compute $K_+(T)$ we can use the words $v_1=5321 \cdot 42$ and $v_2=32 \cdot 5421$.  Then
\begin{equation*}
K_+(T)=\tableau{1 & 2 \\ 2 & 4 \\ 4 \\ 5} \; .
\end{equation*}
\end{example}

Let $\sigma$ be a permutation written in one-line notation and let $\lambda$ be a partition.  There exists a key $K(\sigma, \lambda)$ associated to $\sigma$ and $\lambda$ which is defined by setting the $j$th column of $K(\sigma, \lambda)$ to be the first $\lambda_j$ letters of $\sigma$ in increasing order. 

Our first characterization of Demazure atoms first appeared in \cite{Lascoux1990Keys}.  Let $\sigma \in S_n/\Stab(\lambda)$ be a minimal length coset representative such that $\sigma(\lambda)=\gamma$.

\begin{lemma}{\label{lem:atomaskey}}(\cite{Lascoux1990Keys})
The Demazure atom $\dzatom_{\tau(\lambda)}=\bar{\pi}_\tau(\xx^\lambda)$ is the sum of the weights of all Young tableaux whose right key is equal to $K(\sigma, \lambda^t)$.
\end{lemma}   

\section{Nonsymmetric Macdonald Polynomials at $q=t=0$ and Column-Strict Composition Tableaux}{{\label{sec:colstrictcomp}}}

Following \cite{Ion2003Nonsymmetric} and \cite{Sanderson2000On-the}, in \cite{Mason2009An-explicit} the author shows that Demazure atoms $\dzatom_{\tau(\lambda)}$ are specialized nonsymmetric Macdonald polynomials, where the author in \cite{Mason2009An-explicit} uses a version of the nonsymmetric Macdonald polynomials studied in \cite{Marshall1999Symmetric}. See Chapter \ref{ch:4ndChapterLabel} for details on nonsymmetric Macdonald polynomials.

\begin{theorem}(\cite{Mason2009An-explicit}){\label{thm:atommac}}
The Demazure atom $\dzatom_\gamma(x_1, \ldots, x_n)=\bar{\pi}_\tau(\xx^\lambda)$ is equal to the nonsymmetric Macdonald polynomial $E_\gamma(x_1, \ldots, x_n; 0,0)$ at $q=t=0$, where $\tau(\lambda)=\gamma$.
\end{theorem}

The proof of this theorem relied first on classifying the fillings of diagrams in the combinatorial formula of $E_\gamma(\xx;0,0)$ as so-called \emph{semi-standard augmented fillings,} or SSAF. Then the author provided certain bijections to Young tableaux to establish Theorem \ref{thm:atommac}. These semi-standard augmented fillings of shape $\gamma$ are defined as non-attacking augmented fillings of $\widehat{\dg}(\gamma)$ such that there are no descents and every triple of Type I and Type II is an inversion triple.  The reader can find the definitions of these terms in Section \ref{ch4sec:combdef}. 

In a later work \cite{Haglund2008Quasisymmetric} the authors give a bijection between semi-standard augmented fillings and \emph{column-strict composition tableaux.}  The reader should compare the following definition to Definition \ref{def_RCT}.

\begin{definition}{\label{def:CT}}
Let $\alpha$ be a strong composition with $k$ parts and largest part size $m$.  A {\it{column-strict composition tableau}} (CT) $U$ is a filling of the diagram $\alpha$ such that
\begin{enumerate}
\item The first column is strictly increasing when read top to bottom.
\item Each row weakly decreases when read left to right.
\item \textbf{Triple Rule:} Supplement $U$ with zeros added to the end of each row so that the resulting filling $\hat{U}$ is of rectangular shape $k \times m$.  Then for $1 \leq i_1 < i_2 \leq k$ and $2 \leq j \leq m$,
\begin{equation*}
\left( \hat{U}(i_2,j) \neq 0 \text{ and } \hat{U}(i_2,j)\geq \hat{U}(i_1,j) \right)\Rightarrow \hat{U}(i_2,j) > \hat{U}(i_1, j-1).
\end{equation*}
\end{enumerate}
\end{definition}

If we let $\hat{U}(i_2,j)=b$, $\hat{U}(i_1,j)=a$, and $\hat{U}(i_1, j-1)=c$, then the Triple Rule ($b\neq 0$ and $b \geq a$ implies $b > c$) can be pictured as

\[
\begin{array}{ccc} \vspace{6pt}
\tableau{   c & a   \\  & \bas{{\vdots}} }  
 \\
 \tableau{ & b}
\end{array}.
\]

\begin{lemma}(\cite{Haglund2008Quasisymmetric}){\label{lem:ssafasCT}}
There exists a weight preserving bijection between column-strict composition tableaux of shape $\alpha$ and semi-standard augmented fillings of shape  $\gamma$ such that $\strongof{\gamma}=\alpha$.
\end{lemma}

\begin{example} The following pair consisting of a column-strict composition tableau $U$ of shape $(1,3,2,2)$ and a semi-standard augmented filling $V$ of shape $(1,0,3,0,0,2,2)$ illustrates the bijection whose existence is claimed in Lemma \ref{lem:ssafasCT}.
\begin{equation*}
U = \tableau{1 \\
	           3 & 2 & 2 \\
	           6 & 4 \\
	           7 & 7}
\; \Leftrightarrow \;
V = \tableau{ \bas{1} & 1 \\
		     \bas{2} \\
		     \bas{3} & 3 & 2 & 2 \\
		     \bas{4} \\
		     \bas{5} \\
		     \bas{6} & 6 & 4 \\
		     \bas{7} & 7 & 7}
		     \; .
\end{equation*}
\end{example}

Given a composition $\gamma$, define the \emph{foundation} of $\gamma$ as the set
\begin{equation}
\mathcal{F}(\gamma)=\{ i \mid \gamma_i >0\}.
\end{equation}
The proof of Lemma \ref{lem:ssafasCT} in \cite{Haglund2008Quasisymmetric} shows that semi-standard augmented fillings of shape $\gamma$ are in bijection with column-strict composition tableaux whose first column consists of the entries $\mathcal{F}(\gamma)$. Now we can give another characterization of $\dzatom_{\tau(\gamma)}$.

\begin{lemma}(\cite{Haglund2008Quasisymmetric}){\label{lem:atomasCT}}
Let $\tau(\lambda)=\gamma$. The Demazure atom is given by
\begin{equation}
\dzatom_{\gamma} = \sum_{U} \xx^U
\end{equation}
where the sum is over all column-strict composition tableaux $U$ of shape $\strongof{\gamma}$ whose first column is $\mathcal{F}(\gamma)$.

\end{lemma}

\section{Composition Array Patterns}{\label{sec:comparray}}

In this section we first recall the classical bijection between Young tableaux and Gelfand Tsetlin patterns. We then give a characterization of $\dzatom_{\tau(\lambda)}$ which parallels the construction in classical symmetric function theory of GT-patterns. In this section we return to our previous convention of using reverse Young tableaux.

\subsection{Definitions}

Let $\lambda$ and $\mu$ be partitions such that $\mu \subseteq \lambda$, that is $\ell(\mu) \leq \ell(\lambda)$ and $\mu_i < \lambda_i$ for all $1 \leq i \leq \ell(\mu)$.  A reverse column-strict Young tableaux of shape $\lambda/\mu$ is a filling of the boxes of the skew diagram $\lambda/\mu$ such that each row weakly decreases and each column strictly decreases.  

A \emph{Gelfand-Tsetlin pattern} of shape $\lambda$ is a triangular array $GT_n^\lambda=(x_{i,j})$ of positive integers such that $x_{i,j} \geq x_{i+1, j} \geq x_{i, j+1}$ for all $i$ and $j$ satisfying $1 \leq i \leq n-1$ and $1 \leq j \leq n-i$, where we use the indexing convention depicted below: 

\[
\begin{matrix}
x_{1,1} &        & \cdots &        & \cdots &           & \cdots &           & x_{1,n} \\
       & \ddots &        & \ddots &        & \revddots &        & \revddots &        \\
       &        & x_{n-2,1} &        & x_{n-2,2} &           & x_{n-2,3} &           &        \\
       &        &        & x_{n-1,1} &        & x_{n-1,2}    &        &           &        \\
       &        &        &        & x_{n,1} &           &        &           &
\end{matrix}.
\]

A standard fact \cite{Stanley1999EC2} from combinatorics is that GT-patterns of shape $\lambda$ are in bijection with semi-standard Young tableaux.  This bijection, which we will denote by $\varphi$, can be described within our definitions as follows.  Given a GT-pattern $GT_n^\lambda$ the reverse column-strict Young tableau $\varphi(GT_n^\lambda)$ is given as follows. Given a Ferrers diagram of shape $\lambda$, fill the boxes of the skew shape 
\begin{displaymath}
(x_{i,1}, x_{i, 2}, \ldots, x_{i, n-i+1})/(x_{i+1, 1}, x_{i+1,2}, \ldots, x_{i+1, n-i})
\end{displaymath}
with the entry $i$. Given a reverse column-strict Young tableau $T$ of shape $\lambda$, the $i$th row of the array $\varphi^{-1}(T)$ will be the shape of the resulting diagram when one deletes all boxes from $T$ filled with entries strictly less than $i$. 

Let $\alpha$ and $\beta$ be (strong or weak) compositions. Recall that we say $\alpha \subset \beta$ if $\ell(\alpha) \leq \ell(\beta)$ and $\alpha_i \leq \beta_i$ for $1\leq i \leq \ell(\alpha)$. Let $\reverse{\alpha}$ be the reversal of $\alpha$.  We write $\alpha \Subset \beta$ if $\reverse{\alpha} \subset \reverse{\beta}$.

The following definition of skew composition shaped diagrams was first given in \cite{Bessenrodt2011Skew-quasischur}. Given compositions $\alpha \Subset \beta$, the diagram of skew composition shape $\beta // \alpha$ are the boxes that are in $\beta$ but not in $\alpha$ when the diagram of $\alpha$ is positioned in the lower left corner of the diagram of $\beta$.

Next we will define the Gelfand-Tsetlin type triangular arrays which are in bijection to composition tableaux.

\begin{definition}{\label{def:CTarraypattern}}
Let $\gamma$ be a weak composition of length $n$.  A \emph{composition array pattern} of shape $\gamma$ is a triangular array $X_n^\gamma=(x_{i,j})$ of nonnegative integers satisfying:
\begin{enumerate}
\item $\gamma_i = x_{1,i}$,
\item $x_{i,j} \leq x_{i-1, j+1}$, and
\item for all $1\leq i \leq n-1$ and for all $0\leq r < s \leq n-i$ we have
\[
\begin{array}{ccc}
(x_{i+1,r} \geq x_{i+1,s} \text{ and } x_{i+1,r} \geq x_{i, s+1}) & \text{ or } & (x_{i+1,r} < x_{i+1,s} < x_{i, s+1} \text{ and } x_{i, r+1} < x_{i+1,s})
\end{array}
\]
where we understand $x_{j,0}=0$ for all $j$.
\end{enumerate}

\end{definition}

Our first new characterization of Demazure atoms is given in the following theorem. The bijection $\psi$ of Theorem \ref{thm-CT-array} is illustrated in Figure \ref{fig:ssafasarray}.

\begin{theorem}\label{thm-CT-array}
The set of column-strict composition tableaux of shape $\strongof{\gamma}=\alpha$ with first column $\mathcal{F}(\gamma)$ are in bijection with the set of composition arrays of shape $\gamma$.
\end{theorem}
\begin{proof}
We need to construct a bijection $\psi$ mapping from composition arrays of shape $\gamma$ to composition tableaux of shape $\strongof{\gamma}=\alpha$ with first column $\mathcal{F}(\gamma)$. First we describe $\psi$.  Let $X_n$ be a composition array of shape $\gamma$. First create a diagram of shape $x_{n,1}$ and fill each box with the entry $n$.  Then fill the boxes of $(x_{n-1,1}, x_{n-1,2})//(x_n)$ with the entry $n-1$.  Continuing inductively, fill the boxes of $(x_{i,1}, \ldots, x_{i,n-i+1})//(x_{i+1,1}, \ldots, x_{i+1,n-i})$ with $i$.  This clearly creates a filling $U=\psi(X_n)$ of composition shape $\gamma$, where $\strongof{\gamma}=\alpha$.  Since the entries of the array $X_n$ satisfy $x_{i,j} \leq x_{i-1,j+1}$, then by construction the entries in the rows of $U$ weakly decrease from left to right.  It is also clear that the first column of $U$ weakly increases top to bottom.

To see that the first column actually strictly increases, we need to show that for any $i$ there is at most one occurrence of $i$ in the first column of $U$.  Suppose that this is not the case.  Then for some $i$ we have at least two occurrences in the first column, say is rows $i_1$ and $i_2$ with $i \leq i_1<i_2 \leq n$.  That means $x_{i+1, i_1-i}=0$, $x_{i, i_1-i+1} >0$, $x_{i+1, i_2-i}=0$, and $x_{i,i_2-i+1}>0$.  But this is clearly in violation of the defining inequalities since $x_{i+1,i_1-i}=x_{i+1,i_2-i}=0$ but $x_{i+1, i_1-i}=0<x_{i, i_2-i+1}$. Note that this argument also shows that the entry $U(i,1)$ in row $i$ and column $1$ must be either undefined (in which case $\gamma_i=0$) or $U(i,1)=i$. Thus the first column of $U$ is $\mathcal{F}(\gamma)$ and strictly increases.

Now we check that $U$ satisfies the Triple Rule for column-strict composition tableaux.  First we note that there can be at most one occurrence of $i$ in any column $j$.  To see this suppose it is not the case.  Then for some rows $i_1=i+{k_1}$ and $i_2=i+{k_2}$, with $k_1 < k_2$, we have $x_{i, k_1+1}\geq j$ and $x_{i, k_2+1} \geq j$ while $x_{i+1, k_1}<j$ and $x_{i+1, k_2} <j$.  These inequalities violate the defining inequalities of the array.  In detail: If $x_{i+1, k_1} \geq x_{i+1, k_2}$ we have $x_{i+1, k_1} < j \leq x_{i, k_2+1}$.  If $x_{i+1,k_1} < x_{i+1,k_2}$ (and necessarily $x_{i+1, k_2}<x_{i, k_2+1}$) we have $x_{i, k_1+1} \geq j > x_{i+1,k_2}$.  So this shows there is at most one occurrence of $i$ in any column.

Now we check the Triple Rule.  Append zeros to the end of each row of $U$, as in Definition \ref{def:CT}, and call this filling $\hat{U}$. Consider an entry $b=\hat{U}(i_2, j)$.  Every entry in the same column as $b$ is different than $b$, so the entry $a=\hat{U}(i_1,j)$ (which may be zero) satisfies $b \neq a$.  If $b < a$, then there is nothing to check.  If $b >a$, then we have relations $x_{b,i_2-b+1} >0$ and $x_{a, i_1-a+1}>0$ in the array.

Suppose that $c=\hat{U}(i_1, j-1)$ satisfies $c\geq b$.  In the case where $c=b$, then we have $x_{c, i_1-c+1}<x_{c,i_2-c+1}$, and $x_{c+1,i_1-c}< x_{c,i_1-c+1}$, and $x_{c+1, i_2-c} < x_{c, i_2-c+1}$.  The last two inequalities come from the fact that there is at least one entry $c$ in both rows $i_1$ and $i_2$. Considering row $c+1$ of $X_n$, assume we have $x_{c+1, i_1-c} \geq x_{c+1,i_2-c}$.  Because $X_n$ satisfies all the appropriate inequalities, we must have $x_{c+1,i_1-c}\geq x_{c,i_2-c+1}$ which clearly contradicts our assumptions that $x_{c+1,i_1-c}< x_{c,i_1-c+1} < x_{c,i_2-c+1}$.  Suppose instead we have $x_{c+1, i_1-c} < x_{c+1,i_2-c}$.  Note that $x_{c,i_1-c+1} =j-1$, and $x_{c+1,i_2-c} \leq j-1$.  Because $x_{c+1, i_1-c} < x_{c+1,i_2-c}$ we must have $x_{c, i_1-c+1}=j-1 < x_{c+1, i_2-c} \leq j-1$, which is a contradiction.

Now consider the case when $c>b$.  Thus in the array, $x_{c, i_1-c+1}>0$.  Since $c>b$ and appears in column $j-1$, we must have $x_{c, i_1-c+1} \geq x_{c, i_2-c+1}$.  By the definition of the array, we then have $x_{c, i_1-c+1}\geq x_{c-1,i_2-c+2}$.  Since $c>b$ and $a<b$ and $c$ and $a$ are adjacent in the diagram $\hat{U}$, we know then that $x_{c, i_1-c+1}=x_{c-1,i_1-c+2}$.  Thus $x_{c-1,i_1-c+2} \geq x_{c-1,i_2-c+2}$.  This in turn forces $x_{c-1,i_1-c+2} \geq  x_{c-2,i_2-c+3}$.  Again we have $x_{c-1, i_1-c+2}=x_{c-2,i_1-c+3}$.  Thus $x_{c-2,i_1-c+3} \geq x_{c-2,i_2-c+3}$.  Continuing in this way, we eventually see $x_{b+1,i_1-b}=x_{c, i_1-c+1} \geq x_{b+1, i_2-b}$, which forces $x_{b+1, i_1-b}=x_{c, i_1-c+1} \geq x_{b, i_2-b+1}$.  Thus it is impossible for an entry $b$ to be in column $j$ with $a<b$ while an entry $c>b$ is in column $j-1$.  This proves that the Triple Rule is satisfied.

Now to describe the inverse $\psi^{-1}$. Given a column-strict composition tableau $U$ of shape $\strongof{\gamma}=\alpha$ with first column $\mathcal{F}(\gamma)$, expand the diagram to the weak composition shape $\gamma$ according to the bijection in Lemma \ref{lem:ssafasCT}. Thus the diagram has shape $\gamma$, where $\gamma$ has length $n=$(maximum entry in $U$), and where $\{ i \mid \gamma_i >0\} = \{ \hat{U}(i,1) \mid 1\leq i \leq \ell(\alpha)\}$ and for $j \in \{i \mid \gamma_i >0\}$ we have $\gamma_j =\alpha_j$.

Now construct the array $X_n$ as follows.  The $n$th row of $X_n$ will be the number of boxes containing $n$ in the last row of $U$. Note that $n$ can only appear in the last row.  The $(n-1)$st row of $X_n$ will be the weak composition $(x_{n-1,1}, x_{n-1,2})$ where $x_{n-1,1}$ is the number of $(n-1)$'s in the $(n-1)$st row and 
\begin{displaymath}
x_{n-1,2} =  x_{n,1}+(\text{the number of entries $n-1$ in row $n$}).
\end{displaymath}
By construction $x_{n,1} \leq x_{n-1,2}$.  Because $U$ is a column-strict composition tableaux and thus satisfies the Triple Rule, we know that in the case $x_{n,1} < x_{n-1,2}$ we have $x_{n-1,1}<x_{n,1}$.  This is one of the defining inequalities for when $i=n-1$, $r=0$ and $s=1$.

Continuing inductively, suppose we have constructed rows $i+1, i+2, \ldots, n$ of the array $X_n$ and each of these rows satisfy the defining inequalities of a composition array pattern.  The $i$th row will then be the weak composition $(x_{i,1}, \ldots, x_{i, n-i+1})$, where 
\begin{displaymath}
x_{i,j}= x_{i+1, j-1}+(\text{the number of entries $i$ in row $i+j-1$}).
\end{displaymath}
We now need to show that row $i$ of $X_n$ satisfies the correct inequalities with row $i+1$. By construction, $x_{i+1,j} \leq x_{i,j+1}$.  Now assume we are in the case where, for $0 \leq r<s<n-i$, we have $x_{i+1,r} \geq x_{i+1,s}$. In terms of the filling $U$, this relation means that if we look at the truncated portions of rows $i+r$ and $i+s$ which contain entries $\geq i+1$, then the length of the truncated row $i+r$  is weakly greater than the length of the truncated row $i+s$.  Because $U$ satisfies the Triple Rule, we know that the number of $i$'s in row $i+s$ cannot exceed the length of the truncated row $i+r$.  Which is to say $x_{i+1,r}\geq x_{i,s+1}$.  

Suppose now we are in the case $x_{i+1,r} < x_{i+1,s} < x_{i, s+1}$.  In terms of the filling $U$, the relation $x_{i+1,r} < x_{i+1,s}$ mean that the length of the truncated row $i+r$ (again, consider only the portion of the row with entries $\geq i+1$) is strictly shorter than the length of the truncated row $i+s$.  The relation $x_{i+1,s} < x_{i, s+1}$ means there is at least one $i$ in row $i+s$.  Since $U$ satisfies the Triple Rule, the number of $i$'s in row $i+r$ is strictly bounded by the length of the truncated row $i+s$.  That is $x_{i,r+1}<x_{i+1,s}$.

Finally, we see that $\gamma_i=x_{1,i}$ by construction.  Thus the array $X_n$ satisfies all the defining relations for a composition array pattern.  Therefore we have the required bijection.
\end{proof}

\begin{remark}
If we take the set of all composition arrays of shape $\gamma$ where $\strongof{\gamma}=\alpha$ for some fixed strong composition $\alpha$, then Theorem \ref{thm-CT-array} can be viewed as a bijection between column-strict composition tableaux of shape $\alpha$ and composition arrays of shape $\gamma$ with $\strongof{\gamma}=\alpha$.
\end{remark}

\begin{figure}[ht]
\begin{center}
\begin{picture}(100,100)(53,0)
\put(-50,90){$\smalltableau{\smallbas{1} & 1 \\
		     \smallbas{2} \\
		     \smallbas{3} & 3 & 2 & 2 \\
		     \smallbas{4} \\
		     \smallbas{5} \\
		     \smallbas{6} & 6 & 4 \\
		     \smallbas{7} & 7 & 7}$}
\put(40, 50){$\leftrightarrow$}
\put(60, 50){$
\begin{matrix}
1 & {} & 0 & {} & 3 & {} & 0 & {} & 0 & {} & 2 & {} & 2 \\
{} & 0 & {} & 3 & {} & 0 & {} & 0 & {} & 2 & {} & 2 & {} \\
{} & {} & 1 & {} & 0 & {} & 0 & {} & 2 & {} & 2 & {} & {} \\
{} & {} & {} & 0 & {} & 0 & {} & 2 & {} & 2 & {} & {} & {} \\
{} & {} & {} & {} & 0 & {} & 1 & {} & 2 & {} & {} & {} & {} \\
{} & {} & {} & {} & {} & 1 & {} & 2 & {} & {} & {} & {} & {} \\
{} & {} & {} & {} & {} & {} & 2 & {} & {} & {} & {} & {} & {} 
\end{matrix}
$}
\end{picture}
\caption{Instance of the bijection from SSAF to composition array patterns}
\label{fig:ssafasarray}
\end{center}
\end{figure}
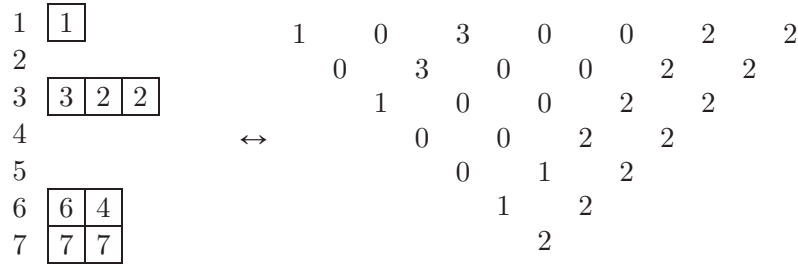

In \cite{Mason2006A-decomposition} the author describes a bijection, which we denote $\theta$, between semi-standard augmented fillings whose shape is a rearrangement of $\lambda$ and reverse column-strict Young tableaux of shape $\lambda$.  The latter are partition shaped diagrams filled with positive integers such that the entries in each row weakly decrease when read left to right, and the entries in each column strictly decrease when read top to bottom. 

In light of Lemma \ref{lem:ssafasCT} we can describe the bijection $\theta$ as follows. Given a column-strict composition tableaux $U$, the $i$th column of $\theta(U)$ is defined to be the $i$th column of $U$ in strictly decreasing order.  The inverse $\theta^{-1}$ is given as follows.  Given a reverse column-strict Young tableaux $T$, the first column of $\theta^{-1}(T)$ is the first column of $T$ in increasing order.  Then place the entries of the $i$th column of $T$, starting with the largest entry, in the highest row of the leftmost column whose rightmost entry is weakly greater.

\begin{example}
Below is an example of the bijection $\theta$.
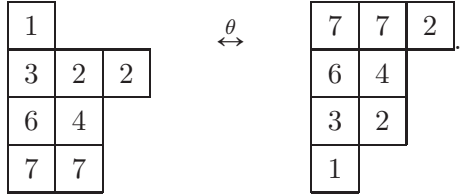
\begin{equation*}
\tableau{1 \\ 
	      3 & 2 & 2 \\
	      6 & 4 \\
	      7 & 7 }
\qquad \stackrel{\theta}{\leftrightarrow} \qquad 
\tableau{ 7 & 7 & 2 \\
	       6 & 4 \\
	       3 & 2 \\
	       1}.
\end{equation*}
\end{example}

One natural question to ask is how to define $\theta$ on composition array patterns such that the image is a Gelfand-Tsetlin pattern.  Let us define the map $\Theta$ acting on composition array patterns to be the map which sorts the $i$th row of a composition array pattern $X_n^\gamma$ into weakly decreasing order. Also define $\tilde{\Theta}$ to be the map which transforms a GT-pattern $GT_n^\lambda$ into a new triangular array by the following procedure.  First, $\tilde{\Theta}$ fixes the $n$th row of $GT_n^\lambda$. Then, inductively, $\tilde{\Theta}$ takes the entries of the $i$th row of $GT_n^\lambda$  from least to greatest and places then as far to the right as possible, such that if the entry $b_j$ from row $i$ is placed at coordinates $(i,k)$, then the entry at coordinates $(i+1, k-1)$ is weakly less than $b_j$. If no such coordinates $(i,k)$ with $k>1$ exist, place $b_j$ in coordinate $(i,1)$.  Thus the entries from row $i$ of $GT_n^\lambda$ are, from least to greatest, placed as far to the right as possible so that the columns running in the north-east direction weakly increase.

The map $\tilde{\Theta}$ is well-defined for the following reason.  Consider two consecutive rows $i-1$ and $i$ of $GT_n^\lambda$.  In row $i-1$ index the entries from left to right so that $a_1\geq a_2 \geq a_3 \geq \cdots$.  Similarly, in row $i$ index the entries so that $b_1 \geq b_2 \geq b_3 \geq \cdots$.  Because $GT_n^\lambda$ is a GT pattern, we have by definition $a_j \geq b_j \geq a_{j+1}$. To show $\tilde{\Theta}$ is well-defined we need to show that after inductively constructing rows $n, n-1, \ldots, i$, and after placing $a_{n-i+2}, a_{n-i+1}, \ldots, a_{j+1}$ in the new array, the element $a_j$ in row $i-1$ of $GT_n^\lambda$ can be placed in a well-defined position in $\tilde{\Theta}(GT_n^\lambda)$. 

If we assume rows $n, n-1, \ldots, i$ of $\tilde{\Theta}(GT_n^\lambda)$ have been constructed, then we can place elements $a_{n-i+2}, a_{n-i+1}, \ldots, a_{l+1}$ in a well defined-position up to some entry $a_l$, where $a_l$ is the first entry in our inductive process in row $i-1$ which is strictly less than all the entries $x_{i,k-1}$ adjacent to each unoccupied coordinate $(i-1,k)$. If $a_l$ is as just described, we place $a_l$ in coordinates $(i-1,1)$. To show $\tilde{\Theta}$ is well-defined, it suffices to show inductively that after placing entries $a_{l-1}, a_{l-2}, \ldots, a_{j+1}$ in well-defined positions, the entry $a_j$, with $j<l$, satisfies $x_{i,k-1} \leq a_j$ for at least one unoccupied coordinate $(i-1,k)$.

To prove this claim, recall that $a_l$ is placed at the coordinate $(i-1,1)$ because $a_l$ is strictly less than the entries $x_{i,k-1}$ adjacent to each unoccupied coordinate $(i-1,k)$. If some $a_s$, with $s >l$, is adjacent to a $b_q$ with $q <l $, then there exists an unoccupied coordinate $(i-1, t)$, with $t > 1$, such that $x_{i, t-1}=b_p$ for some $p\geq l$.  Thus $a_l \geq b_l \geq b_p$, contradicting the fact that $a_l$ is strictly less than the entries $x_{i,k-1}$ adjacent to each unoccupied coordinate. Thus at the time $a_l$ is placed in coordinates $(i-1,1)$, all the remaining unoccupied coordinates $(i-1,k)$ have $x_{i,k-1}=b_q$ for $q <l$. 

We clearly have $|\{b_q \mid q < l \}|=|\{a_r \mid r < l\}|$.  At the time $a_l$ is placed in coordinate $(i-1,1)$ any entry $a_{l-r}$, where $1 \leq r \leq l-1$, will be greater than or equal to $r$ entries $x_{i,k-1}$ adjacent to the unoccupied coordinate $(i-1,k)$; namely $a_{l-r}\geq b_{l-r} \geq b_{l-r+1} \geq \cdots \geq b_{l-1}$. After inductively placing each entry $a_{l-r}$ at most one more of the coordinates adjacent to one of $b_{l-r}, b_{l-r+1}, \ldots, b_{l-1}$ gets occupied.  So at any time $j=l-r$, the entry $a_j$ is greater than or equal to at least one entry $x_{i,k-1}$ adjacent to an unoccupied coordinate. Thus at each step in our inductive process, there is always a well-defined position for each entry $a_j$, and so the map $\tilde{\Theta}$ is well-defined.

\begin{theorem}{\label{thm:GTandCTbiject}}
The map $\Theta$ is a bijection, with inverse $\Theta^{-1}=\tilde{\Theta}$, between composition array patterns of shape $\gamma$ and Gelfand-Tsetlin patterns of shape $\partitionof{\gamma}=\lambda$.  Moreover, the following diagram commutes:
\[
\begin{CD}
A @>\Theta>> B \\
@VV\psi V @VV\varphi V \\
C @>\theta>> D
\end{CD}
\]
where $A$ is  the set of composition array patterns of shape $\gamma$, $B$ is the set of Gelfand-Tsetlin patterns of shape $\partitionof{\gamma}=\lambda$, $C$ is the set of column-strict composition tableaux of shape $\strongof{\gamma}$ with first column $\mathcal{F}(\gamma)$, and $D$ is the set of reverse column-strict Young tableaux of shape $\partitionof{\gamma}=\lambda$.
\end{theorem}
\begin{proof}
Recall that $\Theta$ is the map which sorts the rows of $X_n^\gamma$ into weakly decreasing order. First we will show that, given a composition array pattern $X_n^\gamma$, the image $\Theta(X_n^\gamma)$ is a Gelfand-Tsetlin pattern of shape $\partitionof{\gamma}=\lambda$. 

Consider two consecutive rows $i-1$ and $i$ of $X_n^\gamma$.  In row $i-1$ let $a_j$ be the $j$th largest entry so that $a_1\geq a_2 \geq a_3 \geq \cdots$.  Similarly, in row $i$ let $b_j$ be the $j$th largest entry so that $b_1 \geq b_2 \geq b_3 \geq \cdots$.  To show that $\Theta(X_n^\gamma)$ is a GT-pattern we need to show $a_j \geq b_j  \geq a_{j+1}$ for all $1 \leq j \leq n-i$.

First we show $a_j \geq b_j$. For the case $j=1$, since $a_1$ is the largest entry of row $i-1$, if $a_1 < b_1$ then $b_1=x_{i,s_1} \leq x_{i-1,s_1 +1}$.  Thus $a_1 < b_1 \leq x_{i-1, s_1+1}$, which is a contradiction.  Now assume $j \geq 2$.  Consider $b_k$ for $1 \leq k <j$.  Let $b_k=x_{i,s_k}$.  Then if any $b_k=x_{i,s_k}$ has some $a_l=x_{i-1,s_k+1}$ with $l \geq j$ then $a_j \geq a_l \geq b_k \geq b_j$.  Thus we may assume 
\begin{displaymath}
\{x_{i-1,s_k+1} \mid x_{i, s_k}=b_k \text{ for some $1 \leq k <j$}\} = \{a_1, a_2, \ldots, a_{j-1} \}.
\end{displaymath}
Thus $b_j = x_{i, s_j}$ must have $a_l =x_{i-1, s_j+1}$ with $l \geq j$.  Thus $a_j \geq a_l \geq b_j$.

Next we will show $b_j \geq a_{j+1}$.  For the case $j=1$ we proceed as follows.  Suppose $b_1=x_{i, r}$ and $a_1=x_{i-1, s}$ for some $r < s$, that is $a_1$ is strictly right of $b_1$ in the composition array pattern $X_n^\gamma$.  Since $b_1 \geq x_{i, s-1}$ then we must have $b_1 \geq a_1$, and hence $b_1 \geq a_2$.  

Suppose instead $b_1=x_{i,r}$ and $a_1=x_{i-1, r+1}$, that is $a_1$ and $b_1$ are adjacent in $X_n^\gamma$.  If $a_2=x_{i-1,s}$  is strictly to the right of $a_1$ in row $i-1$, then since $b_1 \geq x_{i,s+1}$ we have $b_1 \geq a_2$.  If $a_2=x_{i-1,s}$ is strictly left of $a_1$ in row $i-1$, that is $s<r$, we have three cases. For the first case, suppose $x_{i,s-1} < b_1 < a_1$.  Then $b_1 > a_2$ by (3) of Definition \ref{def:CTarraypattern}. For the second case suppose $x_{i,s-1} < b_1=a_1$.  Then $b_1=a_1 \geq a_2$.  For the third case suppose $x_{i,s-1}=b_1$.  Then $b_1=x_{i,s-1} \geq a_1$. This $b_1 \geq a_2$.  For the case where $b_1=x_{i,r}$, $a_2=x_{i-1, r+1}$, and $a_1=x_{i-1,s}$ with $s<r$, we can repeat the previous three cases with the roles of $a_1$ and $a_2$ switched.

To finish showing $b_1\geq a_2$, we need to consider the case when both $a_1$ and $a_2$ are strictly left of $b_1$ in $X_n^\gamma$.  Let $b_1=x_{i,r}$.  We may assume that $a_1$ is to the left of $a_2$, and $a_2$ is to the left of $x_{i-1,r+1}$, because the case when $a_2$ is to the left of $a_1$, and $a_1$ is to the left of $x_{i-1,r+1}$ can be completed by reindexing the present case.  Let $a_1=x_{i-1, s_1}$ and $a_2=x_{i-1, s_2}$, with $s_1<s_2$.  We wish to consider the elements $x_{i, s_1-1}:=b_{k_1}$ and $x_{i, s_2-1}:=b_{k_2}$ in the following three cases.  For the first case, if $b_{k_1} \geq b_{k_2}$ then we have $b_{k_1} \geq a_2$.  Thus we have $b_1 \geq b_{k_1} \geq a_2$.  For the second case, suppose $b_{k_1} < b_{k_2} < a_2$.  Then $b_{k_2} > a_1$, from which it follows that $b_1 \geq b_{k_2} > a_1 \geq a_2$.  For the last case, if $b_{k_1} < b_{k_2}=a_2$ then clearly $b_1 \geq b_{k_2} \geq a_2$.

Next we show $b_j \geq a_{j+1}$ for $j\geq 2$.  Let 
\begin{displaymath}
\{x_{i, s_k} \mid a_l = x_{i-1, s_k+1} \text{ for some $1 \leq l \leq j+1$}  \}.
\end{displaymath}
Then by the pigeonhole principle, $x_{i, s_k}=b_k$ with $k > j$ for at least one element in this set.  If $b_k$ is to the right of $b_j$ in $X_n^\gamma$, then $b_j=x_{i, s_j} \geq b_k=x_{i,s_k}$ implies $b_j \geq a_l:=x_{i-1, s_k+1}$ where $1 \leq l \leq j+1$.  Thus $b_j \geq a_l \geq a_{j+1}$ as needed.  

If instead $b_k=x_{i,s_k}$ is to the left of $b_j=x_{i, s_j}$ in $X_n^\gamma$ then we have the following.  Let $a_r = x_{i-1,s_j+1}$ and assume $b_k < b_j < a_r$. Then $b_j > x_{i-1, s_k+1}=a_l$ where $1 \leq l \leq j+1$. This implies a contradiction in the case when $1 \leq l \leq j$.  In the case $l=j+1$ we have $b_j > a_{j+1}$ as needed.  

Next, assume $b_k < b_j = a_r$.  If $r \leq j+1$, then $b_j = a_r \geq a_{j+1}$.  If $r > j+1$, then at least two elements of the set
\begin{displaymath}
\{a_l \mid 1 \leq l \leq j+1 \},
\end{displaymath}
say $a_{l_1}:=x_{i-1, s_1+1}$ and $a_{l_2}:=x_{i-1,s_2+1}$, must have $x_{i, s_1}:=b_{k_1}$ and $x_{i, s_2}:=b_{k_2}$ with $k_1 > j$ and $k_2 > j$.  If either $b_{k_1}$ or $b_{k_2}$ appear to the right of $b_j$ in $X_n^\gamma$, then $b_j \geq b_{k_i}$ forces $b_j \geq a_{l_i} \geq a_{j+1}$.  Thus we may assume both $b_{k_1}$ and $b_{k_2}$ are to the left of $b_j$.  Let us first assume $b_{k_1}$ is left of $b_{k_2}$ is left of $b_j$.  With this arrangement, if $b_{k_1}  \geq b_{k_2}$ then $b_{k_1} \geq a_{l_2}$.  Since $b_j \geq b_{k_1}$ and $a_{l_2} \geq a_{j+1}$ we get $b_j \geq a_{j+1}$.  If instead $b_{k_1} < b_{k_2} < a_{l_2}$ then $b_{k_2} > a_{l_1}$.  Thus $b_j > a_{j+1}$.  Finally, if $b_{k_1} < b_{k_2}=a_{l_2}$ then $b_j \geq b_{k_2}=a_{l_2}\geq a_{j+1}$. The arrangement when $b_{k_2}$ is left of $b_{k_1}$ is left of $b_j$ is exactly analogous.

The last case to consider is when $b_k =b_j$.  Then we have $b_k \geq a_r$.  If $r \leq j+1$ then $b_j=b_k \geq a_r \geq a_{j+1}$.  Now assume $r > j+1$. As in the previous paragraph we must have elements $a_{l_1}$ and $a_{l_2}$, with $l_i \leq j+1$, and elements $b_{k_1}$ and $b_{k_2}$ diagonally adjacent to $a_{l_1}$ and $a_{l_2}$, with $k_i > j$.  If either $b_{k_1}$ or $b_{k_2}$ is right of $b_j$, then $b_j \geq b_{k_i}$ forces $b_j \geq a_{l_i} \geq a_{j+1}$.  Thus we may assume both $b_{k_1}$ and $b_{k_2}$ are left of $b_j$.  Examining the same cases as in the previous paragraph proves $b_j  \geq a_{j+1}$.

Thus $\Theta$ maps composition array patterns to Gelfand-Tsetlin patterns.  Next we show $\Theta$ is a bijection by showing $\tilde{\Theta}=\Theta^{-1}$.

First we show that $\tilde{\Theta}(GT_n^\lambda)$ is a composition array pattern. Relations (1) and (2) of Definition \ref{def:CTarraypattern} are satisfied by construction.  It remains to be shown that the two relations in (3) of Definition \ref{def:CTarraypattern} are satisfied.

Suppose $b_j:=x_{i+1, r} \geq b_k := x_{i+1,s}$ with $r<s$ in $\tilde{\Theta}(GT_n^\lambda)$.  By the definition of the map $\tilde{\Theta}$ we have $j <k$.  Define $a_q:=x_{i, r+1}$ and $a_p:=x_{i, s+1}$.  If $p \geq j+1$ then $b_j \geq a_{j+1} \geq a_p$, as needed.  If instead $p < j+1$, then $q< p$ by the definition of $\tilde{\Theta}$. Considering the sets
\begin{displaymath}
\{a_1, \ldots, a_{p-1}\} \qquad \text{ and } \qquad \{b_1, \ldots, b_{p-1}\}
\end{displaymath}
and noting that $a_q$ is already adjacent to $b_j$, we see that by the pigeonhole principle there must be at least one $a_{t_1} := x_{i, j_1}$, with $t_1 > p$, such that $b_{q_1}:= x_{i+1, j_1-1}$ has $q_1 < p$.  Considering now the two sets 
\begin{displaymath}
\{a_1, \ldots, a_{t_1-1}\} \qquad \text{ and } \qquad \{b_1, \ldots, b_{t_1-1} \}
\end{displaymath}
we again use the pigeonhole principle to get at least one $a_{t_2}:=x_{i,j_2}$, with $t_2>t_1$, such that $b_{q_2}:=x_{i+1,j_2-1}$ has $q_2 < t_1$.  We can continue in this way until we get some $a_{t_k}:=x_{i, j_k}$, with $t_k > j \geq t_{k-1}$, such that $b_{q_k}:=x_{i+1, j_k-1}$ has $q_k < t_{k-1} \leq j$.  These relations forces the equality $b_{q_1}=a_{t_k}$.  In particular, $b_j=a_p$ as needed. 

Now assume $x_{i+1,r} < x_{i+1,s} < x_{i, s+1}$.  Then we must show $x_{i+1,s}> x_{i, r+1}$.  Let $b_k=x_{i+1,r}$ and $b_j=x_{i+1,s}$.  Since $b_k < b_j$, we have $k>j$.  Let $x_{i,s+1}=a_p$ and let $x_{i,r+1}=a_q$. If $p \geq j+1$ then $b_j \geq a_{j+1} \geq a_p$ from the GT pattern.  But we also have $b_j < a_p$ by assumption which produces a contradiction.  Thus we may assume $p < j+1$.  If $q \geq k+1$ then $b_k \geq a_{k+1} \geq a_q$ from the GT pattern.  Since $b_k \leq a_q$ by assumption, we get $b_k=a_q$ and hence $b_j > a_q$ as needed.

Next consider when $p < j+1 \leq q \leq k$.  Then $b_j \geq a_{j+1} \geq a_q$ from the GT pattern.  Because $p<q$, in constructing $\tilde{\Theta}(GT_n^\lambda)$ the entry $a_q$ would have been placed before $a_p$. If $b_j=a_q$ then the entry $a_q$ should have been placed at coordinates $(i,s+1)$ instead of $(i,r+1)$, contradiction our assumptions.  Thus we have $b_j > a_q$ as needed.

We are now in the case when both $p$ and $q$ are both strictly less than $j+1$. If $p < q \leq j$ then 
\begin{displaymath}
a_p \geq a_q \geq a_j \geq b_j.
\end{displaymath}
Again we have, since $p < q$, the entry $a_q$ is placed in $\tilde{\Theta}(GT_n^\lambda)$ before $a_p$. Since $a_q \geq b_j$, the entry $a_q$ should have been placed in coordinates $(i,s+1)$, contradiction our assumptions.

The final case to consider is when $q < p \leq j$. Consider the two sets
\begin{displaymath}
\{a_1, \ldots, a_{p-1}\} \qquad \text{ and } \qquad \{b_1, \ldots, b_{p-1}\}
\end{displaymath}
and note that since $a_q$ is already adjacent to $b_k$ by assumption, the pigeonhole principle guarantees at least one entry $a_{t_1} := x_{i, j_1}$, with $t_1 > p$, such that $b_{q_1}:= x_{i+1, j_1-1}$ has $q_1 < p$.  Considering now the two sets 
\begin{displaymath}
\{a_1, \ldots, a_{t_1-1}\} \qquad \text{ and } \qquad \{b_1, \ldots, b_{t_1-1} \}.
\end{displaymath}
We can use the pigeonhole principle again to get at least one $a_{t_2}:=x_{i,j_2}$, with $t_2>t_1$, such that $b_{q_2}:=x_{i+1,j_2-1}$ has $q_2 < t_1$.  We can continue in this way until we get some $a_{t_k}:=x_{i, j_k}$, with $t_k > j \geq t_{k-1}$, such that $b_{q_k}:=x_{i+1, j_k-1}$ has $q_k < t_{k-1} \leq j$.  These relations forces the equality $b_{q_1}=a_{t_k}$. In particular, $b_j=a_p$, which is a contradiction.

Thus, the map $\tilde{\Theta}$ maps Gelfand-Tsetlin patterns to composition array patterns.  Clearly, the composition $\Theta \circ \tilde{\Theta}$ is the identity on GT patterns. Hence $\tilde{\Theta}$ is injective and $\Theta$ is surjective. To show $\Theta$ is injective, we will show $\theta=\varphi \circ \Theta \circ \psi^{-1}$ which will also show that the diagram in the statement of Theorem \ref{thm:GTandCTbiject} commutes.

Recall that the bijection $\theta$ mapping between column-strict composition tableaux and reverse column-strict Young tableaux sorts the column entries of a column-strict composition tableau $U$ into decreasing order to form a reverse column-strict Young tableau $\theta(U):=T$.   Hence $\theta$ preserves column entries in $U$ and $T$ \cite{Haglund2008Quasisymmetric}.  That is, a positive integer $i$ is in column $j$ of $U$ if and only if $i$ is in column $j$ of $\theta(U)=T$.  

By the definition of the map $\psi$, a positive integer $i$ is in column $j$ of a column-strict composition tableau if and only if there exists an entry $x_{i,k} \geq j$ in $\psi^{-1}(U)$ such that $x_{i+1,k-1}<j$.  Since an entry $i$ can appear at most once in any column of a column-strict composition tableau $U$, we can use the method of skew composition diagrams and $\psi^{-1}$ to see that if there exists an entry $x_{i,k} \geq j$ in $\psi^{-1}(U)$ such that $x_{i+1,k-1}<j$, then any other $x_{i,k^\prime} \geq j$ must have $x_{i+1, k^\prime -1} \geq j$.

The same can be said for reverse column-strict Young tableaux.  That is, a positive integer $i$ is in column $j$ of a reverse column-strict Young tableau $T$ if and only if there exists an entry $y_{i,l}$ of the GT pattern $\varphi^{-1}(T)$ with $y_{i,l} \geq j$ and $y_{i+1,l} <j$, where we understand $y_{i,n-i+2}=0$ for all $i$.

Let $U$ be a column-strict composition tableau.  By the reasoning above, to show $\theta=\varphi \circ \Theta \circ \psi^{-1}$, it is enough to show 
\begin{equation}{\label{eq:iinCT}}
x_{i,k} \geq j \text{ and } x_{i+1,k-1}<j \text{ for some $k$ in } \psi^{-1}(U)
\end{equation}
implies
\begin{equation}{\label{eq:iinYT}}
y_{i,l} \geq j \text{ and } y_{i+1,l}<j \text{ for some $l$ in } (\Theta \circ \psi^{-1}) (U),
\end{equation}
which in turn implies a positive integer $i$ appearing in column $j$ of $(\varphi \circ \Theta \circ \psi^{-1})(U)$. The fact that (\ref{eq:iinCT}) implies (\ref{eq:iinYT}) can be readily checked through a parity argument which we presently describe.

If we assume for some $k$ that $x_{i,k} \geq j$ and $x_{i+1,k-1} <j$ in $\psi^{-1}(U)$, then any other $x_{i,k^\prime}\geq j$ in $\psi^{-1}(U)$ must have $x_{i+1,k-1} \geq j$ (again, this is because the set of entries in any column $j$ of a column-strict composition tableau $U$ are distinct).  Thus we have 
\begin{displaymath}
|\{ x_{i,t} \geq j \mid x_{i,t} \text{ in } \psi^{-1}(U) \} | - |\{x_{i+1,t} \geq j \mid x_{i+1,t} \text{ in } \psi^{-1}(U) \} =1.
\end{displaymath}
Thus, in the GT pattern $(\Theta \circ \psi^{-1} )(U)$ where we sort the entries of $\psi^{-1}(U)$, the right most occurrence $y_{i,l}$ such that $y_{i,l} \geq j$ in $(\Theta \circ \psi^{-1}) (U)$ must have $y_{i+1,l}<j$. Thus there exists a positive integer $i$ in column $j$ of $(\phi \circ \Theta \circ \psi^{-1})(U)$.  

Thus we have shown $(\phi \circ \Theta \circ \psi^{-1})(U)$ is a reverse column-strict Young tableau whose column entries are exactly those of $U$.  Thus, $(\phi \circ \Theta \circ \psi^{-1})(U)=\theta(U)$.

\end{proof}

\section{Lakshmibai-Seshadri Paths}
{\label{sec:LSpaths}}

The following definitions hold for symmetrizable Kac-Moody algebras \cite{Littelmann1994A-Littlewood} \cite{Littelmann1995Paths-and-root}, but since we are interested in type $A$ objects, namely Demazure atoms, we will give all of our definitions in this specific case.

Fix a partition $\lambda$, and as above let $S_n$ denote the symmetric group on $[n]$ and $\Stab(\lambda)$ the stabilizer of $\lambda$. We can identify $S_n/\Stab(\lambda)$ with the subset of $S_n$ consisting of elements $\omega$ such that $\omega s_i > \omega$ for all simple reflections $s_i$ such that $\langle \lambda, \alpha_i \rangle =0$; these elements $\omega$ are exactly minimal length coset representatives for $S_n/\Stab(\lambda)$ \cite{Humphreys1990Reflection}. Let $X$ denote the weight lattice of $\mathfrak{sl}_{n+1}$ and define $X_\R := X \otimes_\Z \R$.  Define the pair $(\underline{\tau}; \underline{a})$ as
\begin{align}
\underline{\tau} &: \tau_1 > \cdots > \tau_r \\
\underline{a} &: 0=a_{0} < a_1 < \cdots < a_r =1
\end{align}
where $\underline{\tau}$ is a strictly decreasing (in Bruhat order) sequence of elements in $S_n/\Stab(\lambda)$ and $\underline{a}$ is a strictly increasing sequence of rational numbers.  

As in \cite{Littelmann1994A-Littlewood}, we will identify the pair $(\underline{\tau}; \underline{a})$ with a piecewise linear map $\pi \colon [0,1] \to X_\R$, and write $\pi=(\underline{\tau}, \underline{a})$, defined by
\begin{align}
&\pi(t) = \sum_{i=1}^{j-1} (a_i -a_{i-1}) \tau_i(\lambda) + (t-a_{j-1})\tau_{j}(\lambda)  \\
&\text{for } t \in [a_{j-1}, a_j], \text{ where } j=1, \ldots, r. \notag
\end{align}
We will call $\pi$ a \emph{rational path of shape} $\lambda$.  This map $\pi$ is a path in $X_\R$ which starts at $0$ and moves initially in the direction of $\tau_1(\lambda)$ for $a_1$ units, then moves in the direction of $\tau_{2}(\lambda)$ for $a_{2}-a_1$ units, and so on.  We will call $\pi(1)$ the \emph{weight} of $\pi$.

\begin{definition}
Let $\tau$ and $\sigma$ be in $S_n/\Stab(\lambda)$ with $\tau>\sigma$, and let $a \in \Q$ with $0 <a <1$.  An \emph{$a$-chain} for the pair $(\tau, \sigma)$ is a sequence of elements $\kappa_0, \kappa_1, \ldots, \kappa_s$ in $S_n/\Stab(\lambda)$ such that
\begin{enumerate}
\item $\tau=\kappa_0 > \kappa_1 > \cdots > \kappa_s=\sigma$,
\item $\ell(\kappa_i) = \ell(\kappa_{i-1})-1$, and
\item $a \langle \kappa_i(\lambda), \beta_i \rangle \in \Z$,
\end{enumerate}
where the $\beta_i$ are positive roots such that $\kappa_i=s_{\beta_i} \kappa_{i-1}$.
\end{definition}

We can now give the crucial definition in our final characterization of Demazure atoms.

\begin{definition}{\label{def:LSpath}}
The rational path $\pi=(\underline{\tau}; \underline{a})$ of shape $\lambda$ is called a \emph{Lakshmibai-Seshadri path} (LS-path) of shape $\lambda$ if for all $i$, $1 \leq i \leq r-1$, there exists an $a_i$-chain for the pair $(\tau_i, \tau_{i+1})$.
\end{definition}

Notice that by definition, if $\pi$ is a LS-path then $\pi(1)$ is in $X$. Denote by $\Pi(\lambda)$ the set of all LS-paths of shape $\lambda$. If $\pi=(\underline{\omega}; \underline{a})$ is an LS-path of shape $\lambda$ with $\underline{\omega}:\omega_1> \cdots > \omega_r$, we define
\begin{align}
\Pi_\tau(\lambda)&= \{ \pi=(\underline{\omega}; \underline{a}) \in \Pi(\lambda) \mid \tau \geq \omega_1 \}, \\
\widehat{\Pi}_\tau(\lambda) &= \{ \pi =(\underline{\omega}; \underline{a})\in \Pi(\lambda) \mid \tau = \omega_1 \}.
\end{align}
The subset $\Pi_\tau(\lambda)$ is the set of all LS-paths $\pi$ of shape $\lambda$ that begin in the direction $\omega_1(\lambda)$ where $\tau \geq \omega_1$, and the subset $\widehat{\Pi}_\tau(\lambda)$ is the subset of all LS-paths $\pi$ of shape $\lambda$ which begin in the direction $\tau(\lambda)$.

In \cite{Littelmann1994A-Littlewood} the author shows that the Demazure character $\pi_\tau(\xx^\lambda)$ is equal to $\sum \xx^{\pi(1)}$, where the sum is over all paths in $\Pi_\tau(\lambda)$.  In the remainder of this chapter we will prove the following, which is our last new characterization of Demazure atoms.

\begin{proposition}{\label{prop:atomaspath}}
With the notation above,
\begin{equation}
\bar{\pi}_\tau(\xx^\lambda)=\dzatom_{\tau(\lambda)}(x_1, \ldots, x_n) = \sum_{\pi \in \widehat{\Pi}_\tau(\lambda)} \xx^{\pi(1)}.
\end{equation}
That is, the Demazure atom $\bar{\pi}_\tau(\xx^\lambda)$ is the sum of the weights of all LS-paths beginning in the direction $\tau(\lambda)$.
\end{proposition}

To prove Proposition \ref{prop:atomaspath} we first need to define the \emph{root operators} of Littelmann \cite{Littelmann1994A-Littlewood}. Let $\pi$ be an LS-path and for a fixed simple root $\alpha$, define the function $h_\alpha \colon [0,1] \to \R$ by $t \mapsto \langle \pi(t), \alpha \rangle$. Define $Q$ to be the absolute minimum attained by the function $h_\alpha$ and let $P=h_\alpha(1)-Q$.  Notice that $P \geq 0$ and since $h_\alpha(0)=0$ we must have $Q \leq 0$.  Since $\pi$ is a piecewise linear path, the minimum $Q$ must be attained at some value $t=a_i$ for some $i$, where $1 \leq i \leq r$.  Let $p$ be maximal such that $0 \leq p \leq r$ and $h_\alpha(a_p) =Q$ and let $q$ be minimal such that $0 \leq q \leq r$ and $h_\alpha(a_q)=Q$.

If $Q \leq -1$, let $y \leq p$ be maximal such that $h_\alpha(t) \geq Q+1$ for all $t \leq a_y$.  Similarly, if $P \geq 1$ let $x \geq p$ be minimal such that $h_\alpha(t) \geq Q+1$ for all $t \geq a_x$.  For any simple root $\alpha$, Littelmann defines operators $e_\alpha$ and $f_\alpha$ which act on LS-paths $\pi$.  The author, having defined $e_\alpha$ and $f_\alpha$ in a more general context, gives the following as a proposition in \cite{Littelmann1994A-Littlewood}.  Since we are in the finite type $A$ case, we present the following as our definition for the action of $e_\alpha$ and $f_\alpha$.

\begin{definition}{(\cite{Littelmann1994A-Littlewood} Proposition 4.2)}{\label{def:root}}
\begin{enumerate}
\item If $P>0$, then $f_\alpha(\pi)$ is the LS-path
\begin{align*}
&(\tau_1, \ldots, \tau_{p-1}, s_\alpha \tau_{p+1}, \ldots, s_\alpha \tau_x, \tau_{x+1}, \ldots, \tau_r ; a_0, \ldots, a_{p-1}, a_{p+1}, \ldots, a_r), \\
&\; \; \qquad \text{ if } h_\alpha(a_x) = Q+1 \text{ and } s_\alpha \tau_{p+1} =\tau_p ; \\
&(\tau_1, \ldots, \tau_{p}, s_\alpha \tau_{p+1}, \ldots, s_\alpha \tau_x, \tau_{x+1}, \ldots, \tau_r ; a_0, \ldots, a_r), \\
&\; \; \qquad \text{ if } h_\alpha(a_x) = Q+1 \text{ and } s_\alpha \tau_{p+1} < \tau_p ; \\
&(\tau_1, \ldots, \tau_{p-1}, s_\alpha \tau_{p+1}, \ldots, s_\alpha \tau_x, \tau_{x}, \ldots, \tau_r ; a_0, \ldots, a_{p-1}, a_{p+1}, \ldots, a_{x-1}, a, a_x, \ldots a_r), \\
&\; \; \qquad \text{ if } h_\alpha(a_x) > Q+1 \text{ and } s_\alpha \tau_{p+1} =\tau_p ; \\
&(\tau_1, \ldots, \tau_{p}, s_\alpha \tau_{p+1}, \ldots, s_\alpha \tau_x, \tau_{x}, \ldots, \tau_r ; a_0, \ldots, a_{x-1},a, a_{x}, \ldots, a_r), \\
&\; \; \qquad \text{ if } h_\alpha(a_x) > Q+1 \text{ and } s_\alpha \tau_{p+1} <\tau_p ;
\end{align*}
where $a_{x-1} < a< a_x$ is such that $h_\alpha(a)=Q+1$. If $P=0$ then $f_\alpha(\pi)=0$.
\item If $Q < 0$ then $e_\alpha(\pi)$ is the LS-path
\begin{align*}
&(\tau_1, \ldots, \tau_{y}, s_\alpha \tau_{y+1}, \ldots, s_\alpha \tau_q, \tau_{q+2}, \ldots, \tau_r ; a_0, \ldots, a_{q}, a_{q+2}, \ldots, a_r), \\
&\; \; \qquad \text{ if } h_\alpha(a_y) = Q+1 \text{ and } s_\alpha \tau_{q} =\tau_{q+1} ; \\
&(\tau_1, \ldots, \tau_{y}, s_\alpha \tau_{y+1}, \ldots, s_\alpha \tau_q, \tau_{q+1}, \ldots, \tau_r ; a_0, \ldots, a_r), \\
&\; \; \qquad \text{ if } h_\alpha(a_y) = Q+1 \text{ and } s_\alpha \tau_{q} > \tau_{q+1} ; \\
&(\tau_1, \ldots, \tau_{y+1}, s_\alpha \tau_{y+1}, \ldots, s_\alpha \tau_q, \tau_{q+2}, \ldots, \tau_r ; a_0, \ldots, a_{y},a, a_{y+1}, \ldots, a_{q}, a_{q+2}, \ldots a_r), \\
&\; \; \qquad \text{ if } h_\alpha(a_y) > Q+1 \text{ and } s_\alpha \tau_{q} =\tau_{q+1} ; \\
&(\tau_1, \ldots, \tau_{y+1}, s_\alpha \tau_{y+1}, \ldots, s_\alpha \tau_q, \tau_{q+1}, \ldots, \tau_r ; a_0, \ldots, a_{y},a, a_{y+1}, \ldots, a_r), \\
&\; \; \qquad \text{ if } h_\alpha(a_y) > Q+1 \text{ and } s_\alpha \tau_{q} >\tau_{q+1} ;
\end{align*}
where $a_{y} < a< a_{y+1}$ is such that $h_\alpha(a)=Q+1$. If $Q=0$ then $e_\alpha(\pi)=0$.
\end{enumerate}
\end{definition}

In \cite{Littelmann1994A-Littlewood} it is also shown that if $f_\alpha(\pi) \neq 0$ then $f_\alpha(\pi)(1) = \pi(1) -\alpha$, and if $e_\alpha(\pi) \neq 0$ then $e_\alpha(\pi)(1) = \pi(1) +\alpha$. Observe that Definition \ref{def:root} shows that if $\pi$ is an LS-path beginning in the direction of $\tau_1(\lambda)$, then $f_\alpha(\pi)$ will be a path beginning in either the direction $\tau_1(\lambda)$ or $s_\alpha \tau_1(\lambda)$. The latter case occurs if and only if $p=0$.  When $p=0$, then $h_\alpha(t)>0$ for all $0<t \leq 1$.  In particular, $h_\alpha(a_1)=\langle a_1 \tau_1(\lambda), \alpha \rangle >0$.  Thus $\langle \tau_1(\lambda), \alpha \rangle> 0$, hence $s_\alpha \tau_1 >  \tau_1$. 

Let $\pi$ be an LS-path beginning in the direction $\tau_1(\lambda)$ and assume $e_\alpha(\pi)=0$.  Furthermore, assume $f_\alpha(\pi) \neq 0$ and $f_\alpha(\pi)$ is an LS-path beginning in the direction $\tau_1(\lambda)$.  Then we claim $f_\alpha^k(\pi)$ is either $0$ or an LS-path beginning in the direction $\tau_1(\lambda)$ for all $k \geq 2$. If the claim is false, then there exists some $k$ such that $f_\alpha^k(\pi)$ begins in the direction $\tau_1(\lambda)$ but $f_\alpha^{k+1}(\pi)$ begins in the direction $s_\alpha\tau_1(\lambda)$.  Thus $p=0$ for $f_\alpha^k(\pi)$, and so $Q=h_{\alpha}(a_p)=h_{\alpha}(0)=0$. Thus $e_\alpha(f_\alpha^k(\pi))=0$, which is a contradiction.

As above, let $\pi$ be an LS-path beginning in the direction $\tau_1(\lambda)$ and assume $e_\alpha(\pi)=0$. Now assume $f_\alpha(\pi) \neq 0$ and $f_\alpha(\pi)$ is an LS-path beginning in the direction $s_\alpha \tau_1(\lambda)$.  Then we claim $f_\alpha^k(\pi)$ is either $0$ or an LS-path beginning in the direction $s_\alpha \tau_1(\lambda)$ for all $k \geq 2$.  To prove this, again suppose the claim is false.  Then as above there exists some $k$ such that $f_\alpha^k(\pi)$ begins in the direction $s_\alpha \tau_1(\lambda)$ but $f_\alpha^{k+1}(\pi)$ begins in the direction $\tau_1(\lambda)$. Again we conclude $Q=0$ for $f_\alpha^k(\pi)$, and thus $e_\alpha(f_\alpha^k(\pi))=0$, which is a contradiction.

The previous two paragraphs show that each LS-path in a given $\alpha$-root string all begin in the same direction, except for possibly the first path in the $\alpha$-root string. 

We now have the tools to prove Proposition \ref{prop:atomaspath}

\begin{proof}[Proof of Proposition \ref{prop:atomaspath}]
Fix $\lambda$ a partition.  We will induct on the length of the permutation $\tau \in S_n/\Stab(\lambda)$.  When $\tau = \epsilon$ is the identity, $\bar{\pi}_\tau(\xx^\lambda)=\xx^\lambda$, and there is only one LS-path $\pi$ of shape $\lambda$ beginning in the direction $\lambda$.  This path is given by $\underline{\tau}: \epsilon$ and $\underline{a}:0 < 1$.  Clearly $\pi(1)=\lambda$.  

Let $s_i$ be a simple reflection such that $\ell(s_i \tau) > \ell(\tau)$ and $s_i \tau \in S_n/\Stab(\lambda)$. Assume the inductive hypothesis that $\bar{\pi}_\tau(\xx^\lambda)= \sum_{\pi \in \widehat{\Pi}_\tau(\lambda)} \xx^{\pi(1)}$. For any fixed $\pi^\prime \in  \widehat{\Pi}_\tau(\lambda)$ let $m$ be the maximal integer such that $f_{\alpha_i}^m(\pi^\prime) \neq 0$. By Definition \ref{def:root}, 
\begin{equation}
f_{\alpha_i}^t \left( \pi^\prime \right) \in \widehat{\Pi}_\tau(\lambda) \cup \widehat{\Pi}_{s_i \tau}(\lambda)
\end{equation}
for all $0 \leq t \leq m$.  

Let $\pi \in  \widehat{\Pi}_{s_i \tau}(\lambda)$.  Then since $h_{\alpha_i}(a_1)=\langle \pi(a_1), \alpha_i \rangle =a_1 \langle s_i \tau(\lambda), \alpha_i \rangle <0$, we have $Q<0$.  Thus $e_{\alpha_i}(\pi) \neq 0$, and  $e_{\alpha_i}(\pi)$  begins in either the direction $\tau(\lambda)$ or the direction $s_i\tau(\lambda)$ by Definition \ref{def:root}.  In the former case, $\pi=f_{\alpha_i}(\pi^\prime)$, where $\pi^\prime=e_{\alpha_i}(\pi) \in  \widehat{\Pi}_{ \tau}(\lambda)$ and $e_{\alpha_i}(\pi^\prime)=0$. If we are in the latter case, $Q<0$ for $\pi^\prime$ and so $e_{\alpha_i}^2(\pi)\neq 0$. Thus $\pi^\prime:=e_{\alpha_i}^2(\pi)$ begins in either the direction $\tau(\lambda)$ or in the direction $s_i\tau(\lambda)$. Repeat this process the maximum number of times, so that we can write $e_{\alpha_i}^t(\pi)=\pi^\prime$ where $\pi^\prime$ begins in the direction $\tau(\lambda)$ and $Q=0$ for $\pi^\prime$.  Thus every path $\pi  \in  \widehat{\Pi}_{s_i \tau}(\lambda)$ can be written as $f_{\alpha_i}^t(\pi^\prime)$ for some $\pi^\prime \in \widehat{\Pi}_{ \tau}(\lambda)$ and some integer $t > 0$, with $Q=0$ for $\pi^\prime$.

Let $\pi \in  \widehat{\Pi}_{s_i \tau}(\lambda)$ and set $\pi(1):=\gamma$. By the last paragraph, $\pi=f_{\alpha_i}^t(\pi^\prime)$ for some $\pi^\prime \in \widehat{\Pi}_{ \tau}(\lambda)$ and some $0<t\leq \langle \pi^\prime(1), \alpha_i \rangle$ and $Q=0$ for $\pi^\prime$. Thus $\pi(1)=\pi^\prime(1)-t\alpha_i$.  But since $\pi^\prime \in \widehat{\Pi}_{ \tau}(\lambda)$, we have $\xx^{\pi^\prime(1)}$ is in the Demazure atom $\bar{\pi}_\tau(\xx^\lambda)$ by the inductive hypothesis.  Furthermore, since $Q=0$ for $\pi^\prime$ and $f_{\alpha_i}(\pi^\prime) \neq 0$, we have for $\pi^\prime$
\begin{equation*}
0<P=h_{\alpha_i}(1)-Q =h_{\alpha_i}(1) = \langle \pi^\prime(1), \alpha_i \rangle
\end{equation*}
Thus if we set $\beta=\pi^\prime(1)$, then $\beta_i > \beta_{i+1}$ and $[\xx^\gamma]\bar{\pi}_i(\xx^\beta)=1$.  Thus condition (1) of Lemma \ref{lem:monomials} is satisfied.  It remains to be shown that condition (2) of Lemma \ref{lem:monomials} is also satisfied.

Now we want to look at all the paths $\pi^\prime \in \widehat{\Pi}_{\tau(\lambda)}$ which, after some number $t$ of applications of $f_{\alpha_i}$ yield $f_{\alpha_i}^t(\pi^\prime)(1)=\gamma$.  By the inductive hypothesis we have a correspondence
\begin{align*}
\{\xx^\beta &\mid \xx^\beta \text{ is in } \bar{\pi}_\tau(\xx^\lambda) \text{ and } [\xx^\gamma]\bar{\pi}_i(\xx^\beta)=1\} \\
&\leftrightarrow  \{\pi \in  \widehat{\Pi}_{s_i \tau}(\lambda) \mid \pi(1)=\gamma\} \cup \{\pi^\prime \in  \widehat{\Pi}_{ \tau}(\lambda) \mid f_{\alpha_i}^t(\pi^\prime)(1)=\gamma \text{ for some $t$}, f_{\alpha_i}^t(\pi^\prime) \in  \widehat{\Pi}_{ \tau}(\lambda) \}.
\end{align*}
Given a path 
\begin{equation}{\label{eq:piin}}
\pi_0^\prime \in \{\pi^\prime \in  \widehat{\Pi}_{ \tau}(\lambda) \mid f_{\alpha_i}^t(\pi^\prime)(1)=\gamma \text{ for some $t$}, f_{\alpha_i}^t(\pi^\prime) \in  \widehat{\Pi}_{ \tau}(\lambda) \}
\end{equation}
we may map it bijectively to a path $\eta \in \widehat{\Pi}_{ \tau}(\lambda)$ such that $\eta(1)=\mu$ for some $\mu$ in the sum $\sum [\xx^\mu]\bar{\pi}_\tau(\xx^\lambda)$ where $\mu$ is as in Lemma \ref{lem:monomials}, and this can be done as follows. The path $\pi_0^\prime$ satisfies $\langle \pi_0^\prime(1), \alpha_i \rangle > 0$ by virtue of the correspondence above.  Let $k$ be maximal such that $e_{\alpha_i}^k(\pi_0^\prime) \neq 0$.  Set $\phi = \langle e_{\alpha_i}^k(\pi_0^\prime)(1) , \alpha_i \rangle $ and let $\eta=f_{\alpha_i}^{\phi-2k}(\pi_0^\prime)$. The path $\eta=f_{\alpha_i}^{\phi-2k}(\pi_0^\prime)$ is non-zero and begins in the direction $\tau(\lambda)$ by the discussion immediately preceding this proof; that is, $\pi_0^\prime=f_{\alpha_i}^t(\pi^\prime)$ for some $t$ and some $\pi^\prime$, and both $\pi^\prime$ and $\pi_0^\prime$ begin in the same direction $\tau(\lambda)$. Then $\eta$ is a path such that $\eta(1)$ is a weight satisfying  $\xx^{\eta(1)}$ is in $\bar{\pi}_\tau(\xx^\lambda)$ and $[\xx^\gamma]\bar{\pi}_i(\xx^{\eta(1)})=-1$. Conversely, given any $\eta$ contributing to the sum $\sum [\xx^\mu]\bar{\pi}_\tau(\xx^\lambda)$ we may map it to an LS-path $\pi^\prime$ as in \ref{eq:piin} as follows.  Let $l$ be maximal such that $f_{\alpha_i}^l(\eta)\neq 0$ and set $\epsilon= \langle f_{\alpha_i}^l(\eta)(1), \alpha_i \rangle$.  Now let $\pi^\prime=e_{\alpha_i}^{\epsilon-2l}(\eta)$.

This shows that the second condition of Lemma \ref{lem:monomials} is met.  Thus $\xx^{\pi(1)}$ is in $\bar{\pi}_{s_i \tau}(\xx^\lambda)$ as needed.

Now let $\xx^\gamma$ be in $\bar{\pi}_{s_i \tau}(\xx^\lambda)$.  Then the conditions of Lemma \ref{lem:monomials} are satisfied.  Thus $[\xx^\gamma]\bar{\pi}_i(\xx^\beta)=1$ for at least one $\beta$ and $\sum [\xx^\beta] \bar{\pi}_\tau(\xx^\lambda) - \sum [\xx^\mu]\bar{\pi}_\tau(\xx^\lambda) = c$ as in Lemma \ref{lem:monomials}.  For each $\beta$ such that $\xx^\beta$ is in $\bar{\pi}_\tau(\xx^\lambda)$ and $[\xx^\gamma]\bar{\pi}_i(\xx^\beta)=1$ by the inductive hypothesis there are $[\xx^\beta] \bar{\pi}_\tau(\xx^\lambda)$ distinct LS-paths $\pi^\prime \in \widehat{\Pi}_{ \tau}(\lambda)$ such that $\pi^\prime(1)=\beta$. Thus we can write $\gamma=f_{\alpha_i}^t(\pi^\prime)(1)$ for one of these paths $\pi^\prime$ and some positive integer $t$, where $0<t\leq \beta_i -\beta_{i+1}$ is determined by the action of $\bar{\pi}_i$ on $\xx^\beta$. It remains to be shown that $f_{\alpha_i}^t(\pi^\prime):=\pi$ is an LS-path beginning in the direction $s_i \tau(\lambda)$.  If this were not the case, then $\pi \in  \widehat{\Pi}_{ \tau}(\lambda)$ and maps uniquely to a path $\eta$ as described above.  This forces $\xx^\gamma$ to not be in $\bar{\pi}_{s_i \tau}(\xx^\lambda)$, which contradicts our first assumption. Thus $\gamma=\pi(1)$ for an LS-path in $\widehat{\Pi}_{s_i \tau}(\lambda)$.
\end{proof}

 \chapter[%
      Permuted Basement Polynomials
   ]{%
       Permuted Basement Nonsymmetric Macdonald Polynomials
   }%
   \label{ch:4ndChapterLabel}
   In this chapter we show that the nonsymmetric functions obtained by permuting the basement in the combinatorial formula for nonsymmetric Macdonald polynomials \cite{Haglund2008A-combinatorial} are eigenfunctions of a family of commuting operators.  We begin in Section \ref{sec:hecke} by reviewing the necessary definitions of Hecke algebras.  In Section \ref{ch4sec:combdef} we review the combinatorial formula for nonsymmetric Macdonald polynomials given in \cite{Haglund2008A-combinatorial}. In Section \ref{sec:cotripandcoinv} we present Proposition \ref{prop:coinv} which is a relation between two combinatorial statistics given in the previous section. Finally, in Section \ref{sec:permbasepoly} we present Proposition \ref{prop:permbaseeigen} which is the main result of this chapter.  

All notations in this chapter conform to those found in \cite{Haglund2008A-combinatorial}.  For this chapter let $\N=\{0,1,2,\ldots\}$ and let $\xx:=(x_1, x_2, \ldots, x_n)$ be a finite set of indeterminates. 

\section{Hecke Algebras}{\label{sec:hecke}}

Nonsymmetric Macdonald polynomials $E_\gamma$ can be defined as eigenfunctions of certain commuting operators in the double affine Hecke algebra $\D(q;t)$ of type $A$ defined by Cherednik in \cite{Cherednik1995Double-affine}.  This algebra can be defined as follows.

\begin{definition}{\label{def:doubleaffine}}
The double affine Hecke algebra $\D := \D(q;t)$ for $n \geq 3$ is the $\Q(q,t)$-algebra with generators $T_0, T_1, T_2, \ldots, T_{n-1}, \pi$ and $X_1^{\pm1}, X_2^{\pm1}, \ldots, X_n^{\pm1}$ and relations 
$$\begin{array}{ll} 	(T_i-t)(T_i+1)=0  				& \textrm{for $0 \leq i \leq n-1$}\\
			T_iT_j=T_jT_i 	& \textrm{for $i-j \not \equiv \pm 1 mod (n)$}\\
			T_iT_jT_i=T_jT_iT_j				& \textrm{for $i-j \equiv \pm 1 mod(n)$}\\
			\pi T_i \pi^{-1} = T_{i+1} 			&\textrm{indices taken $mod(n)$}
\end{array}$$
and 
$$\begin{array}{ll} 	X_iX_j=X_jX_i    				& \textrm{for all $i,j$}\\
				T_i X_j=X_jT_i 					& \textrm{for $j \neq i, i+1$ mod($n$)}\\
				T_iX_iT_i=tX_{i+1}				& \textrm{for $1 \leq i <n$}\\
				\pi X_i \pi^{-1} = X_{i+1}			& \textrm{for $1 \leq i < n$}\\
				\pi qX_n \pi^{-1} = X_1		 	&\\
					T_0 qX_n T_0 = tX_1.

\end{array}$$
\end{definition}

We will make repeated use of the relations in Definition \ref{def:doubleaffine}.  We will also need Cherednik's representation of $\D$ on the vector space $\Q(q,t)[x_1^{\pm 1}, \ldots, x_n^{\pm 1}]$. 
\begin{definition}{\label{def:cherednik}}
Let $f \in \Q(q,t)[x_1^{\pm 1}, \ldots, x_n^{\pm 1}]$. Then 
\begin{enumerate}
\item $X_i^{\pm 1}$ acts by multiplication by $x_i^{\pm 1}$ for $1 \leq i \leq n$,
\item $T_i f = t(s_i \cdot f) + (t-1) \frac{f-(s_i\cdot f)}{1-x_i/x_{i+1}}$ for $0 < i < n$, 
\item $T_0 f = t(s_i \cdot f) + (t-1) \frac{f-(s_i \cdot f)}{1-qx_n/x_1}$, where $s_0 \cdot f(x_1, \ldots, x_n)$$=f(qx_n, x_2, \ldots, x_{n-1}, q^{-1}x_1)$,
\item $\pi f(x_1, \ldots, x_n) = f(x_2, \ldots, x_n, q^{-1} x_1)$.
\end{enumerate}
\end{definition}

\begin{remark}{\label{rem:symmetric}}
We will make use of the fact that a function $f \in \Q(q,t)[x_1^{\pm 1}, \ldots, x_n^{\pm 1}]$ is symmetric in $x_i, x_{i+1}$ if and only if $T_i f = tf$.
\end{remark}

Now define $Y_i = t^{i-1} T_{i-1}^{-1} \cdots T_1^{-1} \pi T_{n-1} \cdots T_i$.  These elements of $\D$ satisfy $Y_iY_j=Y_jY_i$ for all $i$ and $j$, and $Y_i$ acts on $\Q(q,t)[\xx]$ via Cherednik's representation.  If we let $\gamma$ be a weak composition, define $\tilde{\gamma}_i=q^{-\gamma_i} t^{k_i}$ where 
\begin{equation*}
k_i = |\{j=1, \ldots, i-1 \mid \gamma_j > \gamma_i \} | - |\{ j=i+1, \ldots, n \mid \gamma_j \geq \gamma_i\}|.
\end{equation*}
We have the following Theorem from the literature, see \cite{Knop1997Integrality} for example. We can take the follow theorem as our definition of nonsymmetric Macdonald polynomials.  We note that Macdonald \cite{Macdonald1996Affine-Hecke} defines these polynomials in terms of triangularity and orthogonality conditions.

\begin{theorem}
The operators $Y_i$ admit a simultaneous eigenbasis $E_\gamma(x_1, \ldots, x_n; q,t)$, called nonsymmetric Macdonald polynomials, with eigenvalue $\tilde{\gamma}_i$. That is
\begin{equation*}
Y_i E_\gamma(\xx; q,t)=q^{-\gamma_i} t^{k_i} E_\gamma(\xx; q,t).
\end{equation*}

\end{theorem}

\section{Combinatorial Definitions}{\label{ch4sec:combdef}}

Now we will present the definitions necessary to state the combinatorial formula for nonsymmetric Macdonald polynomials given in \cite{Haglund2008A-combinatorial}. 

Let $\gamma=(\gamma_1, \ldots, \gamma_n)$ be a weak composition with $n$ parts.  We will visualize $\gamma$ as a \emph{skyline diagram} or \emph{column diagram}, which is the set
\begin{equation*}
\dg(\gamma)=\{ (i,j) \in \N \times \N \mid 1 \leq i \leq n, 1 \leq j \leq \gamma_i \}
\end{equation*}
where $i$ indexes the columns and $j$ indexes the rows.  Thus, the diagram's coordinates are in Cartesian coordinates.

The \emph{augmented diagram} of $\gamma$, denote $\widehat{\dg}(\gamma)$, will be the diagram obtained by adjoining $n$ extra boxes in row $0$, thus adding a box at the bottom of every column of $\dg(\gamma)$.

\begin{definition}
Given $\gamma$ a weak composition with $n$ parts and a box $u=(i,j) \in \dg(\gamma)$, define
\begin{enumerate}
\item $\leg(u) = \{ (i,j') \in \dg(\gamma) \mid j' > j \}$
\item $\arm^{\text{left}}(u) = \{ (i', j) \in \dg(\gamma) \mid i' < i , \lambda_{i'} \leq \lambda_i \}$
\item $\arm^{\text{right}}(u) = \{ (i', j-1) \in \widehat{\dg}(\gamma) \mid i' > i , \lambda_{i'} < \lambda_i \}$
\item $\arm(u)=\arm^{\text{left}}(u) \sqcup \arm^{\text{right}}(u)$
\item $l(u)= |\leg(u)| = \gamma_i-j$
\item $a(u) = |\arm(u)|$
\end{enumerate}
\end{definition}

\begin{example}
If $\gamma=(3,1,2,4,3,0,4,2,3)$ and $u=(5,2)$ then the cells belonging to $\leg(u)$, $\arm^{\text{left}}(u)$, and  $\arm^{\text{right}}(u)$ are marked by $x$, $y$, $z$ in the following figure:
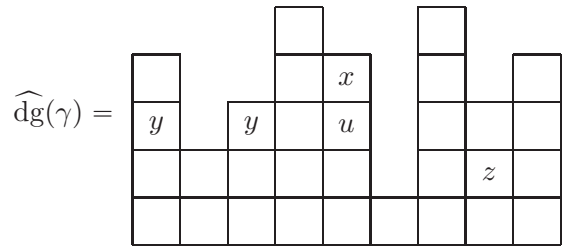
\begin{picture}(100,120)(-90,-10)
\put(0,50){$\widehat{\dg}(\gamma) = $}
\put(45,75){$\tableau{   &   &   &{ }&   &   &{ }&   &   \\
         { }&   &   &{ }&{x}&   &{ }&   &{ }\\
         {y}&   &{y}&{ }&{u}&   &{ }&{ }&{ }\\
         { }&{ }&{ }&{ }&{ }&   &{ }&{z}&{ } \\
         { }&{ }&{ }&{ }&{ }&{ }&{ }&{ }&{ }}\; $}
\end{picture}

giving $\ell(u)=1$, $a(u)=3$. 
\end{example}

A \emph{filling} of $\gamma$ is an assignment of positive integers to the boxes of $\dg(\gamma)$. We will denote fillings by the associated map $\sigma \colon \dg(\gamma) \to [n]$.  Let $\tau \in S_n$, then the associated \emph{augmented filling with basement $\tau$} is the map $\widehat{\sigma}^\tau \colon \widehat{\dg}(\gamma) \to [n]$ such that $\widehat{\sigma}^\tau$ agrees with $\sigma$ on $\dg(\gamma)$, and row $0$ has $\widehat{\sigma}^\tau((j,0))=\tau(j)$, for $1 \leq j \leq n$.  That is, row $0$ has the permutation $\tau$ written in one-line notation. It is at this point that we depart slightly from the definitions in \cite{Haglund2008A-combinatorial}, where the authors only consider the identity permutation $\tau=\epsilon$ in row $0$.  Instead, we present their definitions for arbitrary $\tau$.

Two cells in $\widehat{\dg}(\gamma)$ are said to be \emph{attacking} if either
\begin{enumerate}
\item they are in the same row, or
\item they are in consecutive rows, and the cell in the higher row is strictly to the left of the cell in the lower row.
\end{enumerate}
A filling $\widehat{\sigma}^\tau$ is called \emph{non-attacking} if $\augfill^\tau(u)\neq \augfill^\tau(v)$ for all pairs of attacking boxes $u,v \in \widehat{\dg}(\gamma)$.

\begin{example}
Below is a non-attacking filling $\widehat{\sigma}^\tau$ of shape $\gamma=(2,1,3,0,0,2)$ with basement entries given by $\tau=s_2s_1s_3$.

\begin{picture}(100,100)(-110,10)
\put(0,50){$\widehat{\dg}(\gamma) = $}
\put(45,75){$\tableau{ & & {2} &&& \\
				 {6} & &{4} & & & {5} \\
				 {3} & {1} & {4} & & & {2} \\
				 {3} & {1} & {4} & {2} & {5} & {6}}$}
\end{picture}

\end{example}

Let $d(u)=(i,j-1)$ be the box directly below $u=(i,j)$.  For any $\gamma$, a \emph{descent} in a filling $\augfill^\tau$ of $\gamma$ is a box $u\in \widehat{\dg}(\gamma)$ such that $d(u) \in \widehat{\dg}(\gamma)$ and $\augfill^\tau(u) > \augfill^\tau(d(u))$.  We define

\begin{equation}
\Des(\augfill^\tau)=\{\text{ descents of } \augfill^\tau \},
\end{equation}
and
\begin{equation}
\maj(\augfill^\tau)=\sum_{u \in \Des(\augfill^\tau)} (l(u)+1).
\end{equation}

The \emph{reading order} of a diagram $\widehat{\dg}(\gamma)$ is the total order of the boxes of $\widehat{\dg}(\gamma)$ obtained by reading the cells row by row from left to right, starting in the top row and working downward.  An \emph{inversion} is a pair of boxes $u, v \in \widehat{\dg}(\gamma)$ such that
\begin{enumerate}
\item $u$ and $v$ are attacking,
\item $u<v$ in reading order, and
\item $\augfill^\tau(u) > \augfill^\tau(v)$.
\end{enumerate}
Notice that the inversions with both $u$ and $v$ in row $0$ are exactly the inversion of the permutation $\tau^{-1}$.  Define

\begin{equation}
\Inv(\augfill^\tau) = \{ \text{ inversions of } \augfill^\tau \},
\end{equation}
\begin{equation}
\inv (\augfill^\tau)= |\Inv(\augfill^\tau)| - |\{ i<j \mid \gamma_i \leq \gamma_j \}| - \sum_{u \in \Des(\augfill^\tau)} a(u),
\end{equation}
and
\begin{equation}
\coinv(\augfill^\tau) = \left(\sum_{u \in \dg(\gamma)} a(u) \right) - \inv(\augfill^\tau).
\end{equation}

The statistics $\inv (\augfill^\tau)$ and $\coinv(\augfill^\tau)$ will henceforth be called the \emph{inversion} and \emph{coinversion statistics}, respectively, to differentiate them from the following notions of inversion and coinversion triples.

A \emph{triple} is three boxes $(u,v,w) \in \widehat{\dg}(\gamma)$ such that $w=d(u)$ and $v \in \arm(u)$.  That is, the boxes have one of the orientations

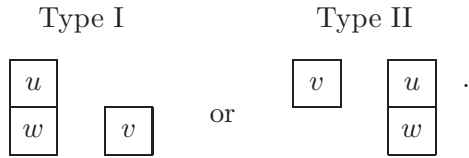
\[
\begin{array}{ccc} \vspace{6pt}
\text{Type I} & & \text{Type II} \\

\begin{array}{c} \vspace{6pt}
\tableau{   u \\
		w  }
\tableau{ \\
		&   v }  
\end{array}
& \; \text{ or } \; &
\begin{array}{c} \vspace{6pt}
\tableau{   v }
\tableau{ &  u \\
		&  w}  
\end{array}
\end{array}.
\]
The total number of triples in $\augfill^\tau$ is equal to $\sum_{u \in \dg(\gamma)}a(u)$. 

Informally, we say that a triple $(u,v,w)$ in $\augfill^\tau$ is a \emph{coinversion triple} if its entries increase clockwise in Type I or counterclockwise in Type II.  If two entries in the triple are equal, we say the entry read first in reading order is smaller.  We say a triple $(u,v,w)$ in $\augfill^\tau$ is an \emph{inversion triple} if it is not a coinversion triple.

Inversion triples can be defined formally as follows.  Given a filling $\augfill^{\tau}$ and boxes $x,y \in \widehat{\dg}(\gamma)$ with $x<y$ in reading order, define
\begin{equation}
\chi_{xy}(\augfill^{\tau}) = \begin{cases} 1 & \text{ if } \augfill^{\tau}(x) > \augfill^{\tau}(y) \\
						              0 & \text{ else.} \end{cases}
\end{equation}

Let $(u,v,w)$ be a triple.  Then we see
\begin{itemize}
\item $\chi_{uv}(\augfill^{\tau}) =1$ if and only if $(u,v) \in \Inv(\augfill^{\tau})$,
\item $\chi_{vw}(\augfill^{\tau}) =1$ if and only if $(v,w) \in \Inv(\augfill^{\tau})$, and
\item $\chi_{uw}(\augfill^{\tau})=1$ if and only if $u \in \Des(\augfill^{\tau})$.
\end{itemize}
From this it follows that $\chi_{uv}(\augfill^{\tau})+\chi_{vw}(\augfill^{\tau})-\chi_{uw}(\augfill^{\tau}) \in \{0,1\}$. Whenever $\chi_{uv}(\augfill^{\tau})+\chi_{vw}(\augfill^{\tau})-\chi_{uw}(\augfill^{\tau})=1$ we call the triple $(u,v,w)$ an inversion triple.  Similarly, if $\chi_{uv}(\augfill^{\tau})+\chi_{vw}(\augfill^{\tau})-\chi_{uw}(\augfill^{\tau})=0$ the triple $(u,v,w)$ is a coinversion triple. Let $\invtrip(\augfill^\tau)$ and $\cotrip(\augfill^\tau)$ be the number of inversion and coinversion triples, respectively, in $\augfill^\tau$.  

\section{Determining the Coinversion Statistic}{\label{sec:cotripandcoinv}}

Before stating the combinatorial formula for $E_\gamma$, we show how the coinversion statistic relates to the number of coinversion triples.

For integers $a$ and $b$, define
\begin{equation}
\chi(a\leq b) = \begin{cases} 1 & \text{ if } a \leq b \\
					     0 & \text{ else.} \end{cases}
\end{equation}
The authors in \cite{Haglund2008A-combinatorial} show that $\coinv(\augfill^\epsilon)=\cotrip(\augfill^\epsilon)$.  Recall that a pair $(i,j)$ is an \emph{inversion of} $\tau$ if and only if $i<j$ and $\tau(i) > \tau(j)$.  If $(i,j)$ is an inversion of $\tau$, then in a filling $\augfill^\tau$ of $\widehat{\dg}(\gamma)$ we see that $\gamma_i$ and $\gamma_j$ are the heights of the columns above the basement entries $\tau(i)$ and $\tau(j)$, respectively. In this section we prove the following.

\begin{proposition}{\label{prop:coinv}}
With the notation above,
\begin{equation}
\coinv(\augfill^{\tau})=\cotrip(\augfill^{\tau}) + \sum_{\substack{(i,j) \text{ is an }\\ \text{ inversion of } \tau}} \chi(\gamma_i \leq \gamma_j).
\end{equation}
\end{proposition}

Before we prove this proposition we need the following lemma from \cite{Haglund2008A-combinatorial} which is readily checked.
\begin{lemma}{\label{lem:intriple}}
Every pair of attacking boxes in $\widehat{\dg}(\gamma)$ occurs as either $\{u,v\}$ or $\{v,w\}$ is a unique triple $(u,v,w)$, except that an attacking pair $\{(i, 0), (j, 0) \}$ in row $0$, with $i < j$ and $\gamma_i \leq \gamma_j$, is not in any triple.
\end{lemma}
\begin{proof}[Proof of Proposition \ref{prop:coinv}]

The sum over all triples,
\begin{equation}
\sum_{(u,v,w)} (\chi_{uv}(\augfill^{\tau})+\chi_{vw}(\augfill^{ \tau})-\chi_{uw}(\augfill^{ \tau}))
\end{equation}
is equal to the number of inversion triples, $\invtrip(\augfill^\tau)$. We claim that
\begin{equation*}
\sum_{\substack{(u,v,w) \\ \text{a triple} }} (\chi_{uv}(\augfill^{ \tau})+\chi_{vw}(\augfill^{\tau}))=|\Inv(\augfill^{\tau})| - |\{ i<j \mid \gamma_i \leq \gamma_j \}| + \sum_{\substack{(i,j) \text{ is an } \\ \text{ inversion of } \tau}} \chi(\gamma_i \leq \gamma_j).
\end{equation*}

To see this, note that $\sum_{(u,v,w)} (\chi_{uv}(\augfill^{ \tau})+\chi_{vw}(\augfill^{\tau}))$ counts all the inversions in $\augfill^{ \tau}$ except those of the form  $\{(i, 0), (j, 0) \}$, with $\tau(i) < \tau(j)$ and $\gamma_i \leq \gamma_{j}$. Consider an inversion $(i,j)$ of $\tau$ and assume $\gamma_i > \gamma_j$.  Then the boxes in the basement containing $\tau(i)$ and $\tau(j)$ are \emph{not} an inversion of $\augfill^\tau$ because $\tau(i) > \tau(j)$. When $\gamma_i> \gamma_j$ then the pair of attacking boxes $t=(i,0)$ and $s=(j,0)$ is in a unique triple $(r,s,t)$ by Lemma \ref{lem:intriple}, where $r$ is the box $(i,1)$. Since $(s,t)$ is not an inversion of $\augfill^\tau$, then $(r,s) \in \Inv(\augfill^{\tau})$ if and only if $\chi_{rs}(\augfill^{\tau}) + \chi_{st}(\augfill^{\tau})=1$. The inversion $(i,j)$ of $\tau$ is not counted in $\sum_{(i,j)} \chi(\gamma_i \leq \gamma_j)$.

If $\gamma_i \leq \gamma_j$, then the boxes $(i,0)$ and $(j,0)$ are still not an inversion and they do not appear in any triple. The term $ |\{ i<j \mid \gamma_i \leq \gamma_j \}|$ counts all pairs of boxes $(x,y)$ in row $0$ with the corresponding column heights weakly increasing as inversions.  But clearly this term over-counts because each inversion $(i,j)$ of $\tau$ that satisfies $\gamma_i \leq \gamma_j$, which is counted in $ |\{ i<j \mid \gamma_i \leq \gamma_j \}|$, is not an inversion of $\augfill^\tau$. Thus we have
\begin{equation*}
\sum_{\substack{(u,v,w) \\ \text{ a triple }}} (\chi_{uv}(\augfill^{s_i \tau})+\chi_{vw}(\augfill^{s_i \tau}) )=|\Inv(\augfill^{\tau})| - |\{ i<j \mid \gamma_i \leq \gamma_j \}| + \sum_{\substack{(i,j) \text{ is an }\\ \text{ inversion of } \tau}} \chi(\gamma_i \leq \gamma_j).
\end{equation*}

Finally, we see that $\sum_{(u,v,w) }\chi_{uw}(\augfill^{ \tau})$ is equal to $\sum_{u \in \Des(\augfill^\tau)} a(u)$.  Thus we have
\begin{align*}
\invtrip(\augfill^\tau)&= |\Inv(\augfill^{\tau})| - |\{ i<j \mid \gamma_i \leq \gamma_j \}| + \sum_{\substack{(i,j) \text{ is an } \\ \text{ inversion of } \tau}} \chi(\gamma_i \leq \gamma_j)-\sum_{u \in \Des(\augfill^\tau)} a(u) \\ 
				&= \inv (\augfill^\tau) + \sum_{\substack{(i,j) \text{ is an }\\ \text{ inversion of } \tau}} \chi(\gamma_i \leq \gamma_j).
\end{align*}

It follows that
\begin{align*}
\coinv(\augfill^\tau) &= \left(\sum_{u \in \dg(\gamma)} a(u) \right) - \inv(\augfill^\tau) \\
&= \left(\sum_{u \in \dg(\gamma)} a(u) \right) -\left( \invtrip(\augfill^\tau) -  \sum_{\substack{(i,j) \text{ is an }\\ \text{ inversion of } \tau}} \chi(\gamma_i \leq \gamma_j) \right) \\
&= \cotrip(\augfill^\tau) + \sum_{\substack{(i,j) \text{ is an }\\ \text{ inversion of } \tau}} \chi(\gamma_i \leq \gamma_j).
\end{align*}

\end{proof}

\section{Permuted Basement Nonsymmetric Macdonald Polynomials}{\label{sec:permbasepoly}}

In this section we first state the combinatorial formula for nonsymmetric Macdonald polynomials $E_\gamma$ given in \cite{Haglund2008A-combinatorial}.  Then we will use the relation in Proposition \ref{prop:coinv} to state how the nonsymmetric Macdonald polynomials transform under the action of the generator $T_i \in \D$; this result was first communicated in \cite{Haglund2010The-action}. We then use this result to show that the nonsymmetric polynomials obtained by permuting the basement in the combinatorial formula for $E_\gamma$ are eigenfunctions of a family of commuting operators.

In \cite{Haglund2008A-combinatorial} the authors prove the following.

\begin{theorem}(\cite{Haglund2008A-combinatorial}){\label{thm:nsmac}}
The nonsymmetric Macdonald polynomials $E_\gamma(x_1, \ldots, x_n; q,t) $ are given by
\begin{equation}
E_\gamma(\xx; q,t) = \sum_{\substack{\sigma \colon \gamma \to [n] \\ \text{non-attacking}}} x^\sigma q^{\maj(\augfill^\epsilon)} t^{\coinv(\augfill^\epsilon)} \prod_{\substack{u \in \dg(\gamma) \\ \augfill^\epsilon(u) \neq \augfill^\epsilon(d(u))}} \frac{1-t}{1-q^{l(u)+1}t^{a(u)+1}}
\end{equation}
where $\epsilon$ is the identity permutation and $x^\sigma=\prod_{u \in \dg(\gamma) } x_{\sigma(u)}$ is the weight of the filling $\sigma$.
\end{theorem}

Note that in Theorem \ref{thm:nsmac} the statistic $\coinv(\augfill^\epsilon)$ can be replaced by $\cotrip(\augfill^\epsilon)$, and in fact this is what the authors in \cite{Haglund2008A-combinatorial} use to prove Theorem  \ref{thm:nsmac}.

The definitions needed to describe the combinatorial formula above (non-attacking fillings, $\maj$, $\coinv$, and so on) are still valid when working with an arbitrary basement $\tau$.  Thus we define \emph{permuted basement nonsymmetric polynomials} as follows.

\begin{definition}
Let $\tau \in S_n$. The permuted basement nonsymmetric polynomials $E_{\gamma, \tau}(x_1, x_2, \ldots, x_n; q,t)$ are given by
\begin{equation}
E_{\gamma,\tau}(\xx; q,t) = \sum_{\substack{\sigma \colon \gamma \to [n] \\ \text{non-attacking}}} x^\sigma q^{\maj(\augfill^\tau)} t^{\cotrip(\augfill^\tau)} \prod_{\substack{u \in \dg(\gamma) \\ \augfill^\tau(u) \neq \augfill^\tau(d(u))}} \frac{1-t}{1-q^{l(u)+1}t^{a(u)+1}}.
\end{equation}
\end{definition}

We note that $E_{\gamma, \tau}$ is defined using the number of coinversion triples $\cotrip$ instead of the coinversion statistic $\coinv$. This is so that the study of permuted basement nonsymmetric polynomials in this dissertation conform to \cite{Haglund2010The-action} and a forthcoming paper \cite{Haglund2011Properties}.

We wish to establish two results in the remainder of this chapter. The first is to record the relationship between $E_{\gamma, \tau}$ and $E_{\gamma, \omega}$ for $\gamma > \omega$ in weak Bruhat order. This was first done in \cite{Haglund2010The-action}.  The second is to show the functions $E_{\gamma,\tau}$ are eigenfunctions for a family of commuting operators in $\D$. One lemma needed to establish the first result is the following.

\begin{lemma}{\label{lem:jims}}
For $0 < i < n$ and any $F,G \in \Q(q,t)[x_1, \ldots, x_n]$ the following are equivalent.
\begin{enumerate}
\item $T_i F =G $,
\item $F+G$ and $tx_{i+1}F + x_i G$ are both symmetric in $x_i, x_{i+1}$.
\end{enumerate}
\end{lemma}  
\begin{proof}
From the defining quadratic relation $(T_i -t)(T_i+1)=0$ in $\D$ we see $T_i^2=(t-1)T_i+t$.  We will also make use of the defining relation $T_i X_i T_i =t X_{i+1}$ and  Remark \ref{rem:symmetric}.  

Assume $T_iF=G$.  Then we compute
\begin{align*}
T_i(F+G) &= T_i(F+T_iF) = T_i(1+T_i)F \\
&= t(1+T_i)F = t(F+G).
\end{align*}
Thus $F+G$ is symmetric in $x_i, x_{i+1}$.

We also see
\begin{align*}
T_i(tx_{i+1}F+x_iG) &= T_i(tx_{i+1} F + x_i (T_i F)) = T_i\left(tx_{i+1} F + (T_i+1-t)(x_{i+1}F) \right) \\
&= T_i ( t+T_i +1 -t)x_{i+1} F = t(T_i+1)x_{i+1}F = t(T_i x_{i+1} F + x_{i+1} F) \\
&= t\left( x_i T_i F + (t-1)x_{i+1} F + x_{i+1} F\right) = t(tx_{i+1}F+x_iG).
\end{align*}

For the converse, assume $F+G$ and $tx_{i+1}F + x_i G$ are both symmetric in $x_i, x_{i+1}$. Thus we have
\begin{align*}
F+G &= (s_i F)+(s_i G) \text{ and } \\
tx_{i+1}F+x_iG &= tx_i(s_i F) + x_{i+1}(s_i G).
\end{align*}
So we can compute
\begin{align*}
T_iF &= t(s_iF)+(t-1)x_{i+1}\frac{F-s_iF}{x_{i+1}-x_i} \\
&= t(s_iF) +\frac{tx_{i+1}F-tx_{i+1}s_iF}{x_{i+1}-x_i} -\frac{x_{i+1}F-x_{i+1}s_iF}{x_{i+1}-x_i} \\
&= t(s_iF) +\frac{tx_is_iF+x_{i+1}s_iG-x_iG-tx_{i+1}s_iF}{x_{i+1}-x_i}+\frac{x_{i+1}s_iF-x_{i+1}F}{x_{i+1}-x_i} \\
&= t(s_iF)+\frac{(x_i-x_{i+1})ts_iF}{x_{i+1}-x_i}+\frac{x_{i+1}s_iG-x_iG}{x_{i+1}-x_i}+\frac{x_{i+1}s_iF-x_{i+1}F}{x_{i+1}-x_i} \\
&= \frac{x_{i+1}s_iG-x_iG+x_{i+1}s_iF-x_{i+1}F}{x_{i+1}-x_i}\\
&= \frac{x_{i+1}(s_iG+s_iF)-x_iG -x_{i+1}F}{x_{i+1}-x_i} \\
&= G.
\end{align*}
\end{proof}

Lemma \ref{lem:jims} can be used to prove the following proposition, which was first communicated to this author in \cite{Haglund2010The-action}. Recall, the polynomials $E_{\gamma, \tau}$ are defined using $\cotrip$ instead of $\coinv$.

\begin{proposition}(\cite{Haglund2010The-action}){\label{prop:theaction}}
Let $\tau\in S_n$ and $i$ be such that $0 < i < n$ and  $\ell(s_i \tau) > \ell(\tau)$.  In the augmented filling with basement given by $\tau$, let $a$ be the height of the column above the entry $i$ in row $0$ and let $b$ be the height of the column above the entry $i+1$ in row $0$. Then
\begin{equation}
T_i  E_{\gamma, \tau} = t^{\chi (a \leq b)} E_{\gamma, s_i \tau}.
\end{equation}
\end{proposition}

\begin{remark}{\label{rem:howtouse}}
Let $\tau \in S_n$ be reduced.  Using Proposition \ref{prop:theaction} inductively we see
\begin{equation}
T_\tau E_{\gamma, \epsilon}=t^c E_{\gamma, \tau}
\end{equation}
where $c$ is the constant $c=\sum \chi(\gamma_i \leq \gamma_j)$, where the sum is over all inversions $(i,j)$ of $\tau$.

\end{remark}

The remainder of this chapter will be used to show the operators $Y_i^\tau$, whose definition follows, are the family of commuting operators in the double affine Hecke algebra for which $E_{\gamma, \tau}$ are eigenfunctions.

\begin{definition}
Let $Y_i^\tau=t^{\varepsilon_i^\tau}T_{i-1}^{\varepsilon_{i-1}}\cdots T_1^{\varepsilon_1} \pi T_{n-1}^{\varepsilon_{n-1}}\cdots T_i^{\varepsilon_i}$ with $\varepsilon_j \in \{-1, 1\}$ and $\varepsilon_i^\tau=|\{j | 1 \leq j \leq n-1 \text{ such that } \varepsilon_j =-1\}|$.  Define $\varepsilon_j$ via:
\begin{enumerate}
\item For indices $j$ with $i \leq j < n$: $\varepsilon_j=-1$ if $(i,j+1)$ form an inversion in $\tau$, and $\varepsilon_j=1$ else.
\item For indices $j$ with $1 \leq j < i$: $\varepsilon_j=1$ if $(j,i)$ form an inversion in $\tau$, and $\varepsilon_j=-1$ else.
\end{enumerate}
\end{definition}

The following lemma gives the relation between the DAHA generator $Y_i$ and the operator $Y_i^\tau$ in $\widehat{H_n}(q,t)$.
\begin{lemma}{\label{lem:TsandYs}}
For all $1 \leq i \leq n$ and $\tau$ reduced we have 
\begin{equation}\label{eq:tandy}
Y_i^\tau = T_\tau Y_{\tau^{-1}(i)}T_\tau^{-1}.
\end{equation}
\end{lemma}
\begin{proof}
We induct on the length of $\tau$.  If $\tau=\epsilon$ is the identity then $T_\tau=1$ and then (\ref{eq:tandy}) is $Y_i^\epsilon=Y_i$.  If $\tau=s_j$ has length one, then we will prove $Y_i^{s_j} = T_j Y_{s_j(i)}T_j^{-1}$ in the following four cases.

Assume $j < i-1$.  Then
\begin{align*}
T_j Y_{s_j(i)}T_j^{-1} &= T_j Y_i T_j^{-1} \\
&= T_j \left( t^{i-1} T_{i-1}^{-1} \cdots T_1^{-1} \pi T_{n-1}\cdots T_i \right)T_j^{-1} \\
&= t^{i-1}T_{i-1}^{-1} \cdots T_j T_{j+1}^{-1}T_j^{-1} \cdots T_1^{-1} \pi T_{n-1} \cdots T_i T_j^{-1} \\
&= t^{i-1} T_{i-1}^{-1} \cdots T_{j+1}^{-1}T_j^{-1} T_{j+1} \cdots T_1^{-1} \pi T_{n-1} \cdots T_i T_j^{-1} \\
&= t^{i-1} T_{i-1}^{-1} \cdots T_1^{-1} T_{j+1} \pi T_{n-1} \cdots T_i T_j^{-1} \\
&= t^{i-1} T_{i-1}^{-1} \cdots T_1^{-1} \pi T_j T_{n-1} \cdots T_i T_j^{-1} \\
&= t^{i-1} T_{i-1}^{-1} \cdots T_1^{-1} \pi T_{n-1}\cdots T_i \\
&= Y_i^{s_j}.
\end{align*}

Similarly, if $j >i$ then 
\begin{align*}
T_j Y_{s_j(i)}T_j^{-1} &= T_j Y_i T_j^{-1} \\
&= T_j \left( t^{i-1} T_{i-1}^{-1} \cdots T_1^{-1} \pi T_{n-1}\cdots T_i \right) T_j^{-1} \\
&= t^{i-1} T_{i-1}^{-1} \cdots T_1^{-1} T_j \pi T_{n-1} \cdots T_i T_j^{-1} \\
&= t^{i-1} T_{i-1}^{-1} \cdots T_1^{-1} \pi T_{j-1} T_{n-1} \cdots T_i T_j^{-1} \\
&=t^{i-1} T_{i-1}^{-1} \cdots T_1^{-1} \pi T_{n-1}\cdots T_{j-1}T_jT_{j-1} \cdots T_i T_j^{-1} \\
&= t^{i-1} T_{i-1}^{-1} \cdots T_1^{-1} \pi T_{n-1}\cdots T_{j}T_{j-1}T_{j} \cdots T_i T_j^{-1} \\
&= Y_i^{s_j}.
\end{align*}

If $j=i$ then
\begin{align*}
T_j Y_{s_j(i)}T_j^{-1} &= T_i Y_{i+1} T_i^{-1} \\
&=T_i \left( t^i T_i^{-1} \cdots T_1^{-1} \pi T_{n-1} \cdots T_{i+1}\right) T_i^{-1} \\
&= t^i T_{i-1}^{-1} \cdots T_1^{-1} \pi T_{n-1} \cdots T_{i+1} T_i^{-1} \\
&= Y_i^{s_i}=Y_i^{s_j}.
\end{align*}

Similarly, if $j=i-1$ then 
\begin{align*}
T_j Y_{s_j(i)}T_j^{-1} &= T_{i-1} Y_{i-1} T_{i-1}^{-1} \\
&= T_{i-1} \left( t^{i-2} T_{i-2}^{-1} \cdots T_1^{-1} \pi T_{n-1} \cdots T_{i-1} \right) T_{i-1}^{-1} \\
&= t^{i-2} T_{i-1} T_{i-2}^{-1} \cdots T_1^{-1} \pi T_{n-1} \cdots T_i \\
&= Y_i^{s_{i-1}}=Y_i^{s_j}.
\end{align*}

Now assume $Y_i^\tau=T_\tau Y_{\tau^{-1}(i)} T_\tau^{-1}$ for all $\tau$ of length at most $L$. Assume that $\ell(s_j \tau)=L+1$. Using the inductive hypothesis we see $Y_{s_j(i)}^\tau=T_\tau Y_{\tau^{-1}(s_j(i))} T_\tau^{-1}$, and so
\begin{equation*}
T_j Y_{s_j(i)}^\tau T_j^{-1} = T_j T_\tau Y_{\tau^{-1}(s_j(i))} T_\tau^{-1} T_j^{-1} = T_{s_j\tau} Y_{\tau^{-1}s_j(i)} T_{s_j\tau}^{-1}.
\end{equation*}
Again we proceed by cases.

Assume $j <i-1$.  Then 
\begin{align*}
T_j Y_{s_j(i)}^\tau T_j^{-1} &= T_j \left( t^{\varepsilon_i^\tau} T_{i-1}^{\varepsilon_{i-1}} \cdots T_1^{\varepsilon_1} \pi T_{n-1}^{\varepsilon_{n-1}} \cdots T_i^{\varepsilon_i} \right) T_j^{-1} \\
&= t^{\varepsilon_i^\tau} T_{i-1}^{\varepsilon_{i-1}} \cdots T_j T_{j+1}^{\varepsilon_{j+1}}T_{j}^{\varepsilon_{j}} \cdots T_1^{\varepsilon_1} \pi T_{n-1}^{\varepsilon_{n-1}} \cdots T_i^{\varepsilon_i} T_j^{-1} .
\end{align*}
Since $(j,j+1)$ is not an inversion in $\tau$ and $j < i-1$, we see that $T_j T_{j+1}^{\varepsilon_{j+1}}T_{j}^{\varepsilon_{j}}$ is equal to
\begin{itemize}
\item $T_j T_{j+1}^{-1} T_j^{-1} \Leftrightarrow \tau= \tau(1) \cdots  j \cdots j+1 \cdots i \cdots \tau(n)$, or
\item $T_j T_{j+1} T_j^{-1} \Leftrightarrow \tau= \tau(1) \cdots  j \cdots i \cdots j+1 \cdots \tau(n)$, or
\item $T_j T_{j+1} T_j \Leftrightarrow \tau= \tau(1) \cdots  i \cdots j \cdots j+1 \cdots \tau(n)$.   
\end{itemize}
In each case we can apply a braid relation yielding $T_{j+1}^{\varepsilon_{j+1}^\prime} T_j^{\varepsilon_j^\prime} T_{j+1}$.  In the first and third cases $\varepsilon_{j+1}^\prime=\varepsilon_{j+1}$ and $\varepsilon_j^\prime=\varepsilon_j$. In the second case $\varepsilon_{j+1}^\prime = -\varepsilon_{j+1}$ and $\varepsilon_j^\prime = -\varepsilon_j$. Since the total number of indices $k$ for which $\varepsilon_k=-1$ is unchanged, we have $\varepsilon_i^\tau=\varepsilon_i^{s_j \tau}$. In all cases, after the braid relation, we have
\begin{align*}
T_j Y_{s_j(i)}^\tau T_j^{-1} &= t^{\varepsilon_i^\tau} T_{i-1}^{\varepsilon_{i-1}} \cdots T_1^{\varepsilon_1} T_{j+1} \pi T_{n-1}^{\varepsilon_{n-1}} \cdots T_i^{\varepsilon_i} T_j^{-1} \\ 
&= t^{\varepsilon_i^\tau} T_{i-1}^{\varepsilon_{i-1}} \cdots T_1^{\varepsilon_1} \pi T_j T_{n-1}^{\varepsilon_{n-1}} \cdots T_i^{\varepsilon_i} T_j^{-1} \\ 
&= Y_i^{s_j \tau}.
\end{align*}

Next, assume $j>i$.  Then similarly
\begin{align*}
T_j Y_{s_j(i)}^\tau T_j^{-1} &= T_j \left(t^{\varepsilon_i^\tau} T_{i-1}^{\varepsilon_{i-1}} \cdots T_1^{\varepsilon_1} \pi T_{n-1}^{\varepsilon_{n-1}} \cdots T_i^{\varepsilon_i} \right)T_j^{-1} \\
&=t^{\varepsilon_i^\tau}T_{i-1}^{\varepsilon_{i-1}}  \cdots T_1^{\varepsilon_1} \pi T_{n-1}^{\varepsilon_{n-1}} \cdots T_{j-1} T_{j}^{\varepsilon_{j}}T_{j-1}^{\varepsilon_{j-1}}\cdots T_i^{\varepsilon_i} T_j^{-1} .
\end{align*}
Since $(j,j+1)$ is not an inversion in $\tau$ and $j >i$, we see that $T_{j-1} T_{j}^{\varepsilon_{j}}T_{j-1}^{\varepsilon_{j-1}}$ is equal to
\begin{itemize}
\item $T_{j-1} T_{j}^{-1} T_{j-1}^{-1} \Leftrightarrow \tau= \tau(1) \cdots  j \cdots j+1 \cdots i \cdots \tau(n)$, or
\item $T_{j-1} T_{j} T_{j-1}^{-1} \Leftrightarrow \tau= \tau(1) \cdots  j \cdots i \cdots j+1 \cdots \tau(n)$, or
\item $T_{j-1} T_j T_{j-1} \Leftrightarrow \tau= \tau(1) \cdots  i \cdots j \cdots j+1 \cdots \tau(n)$.   
\end{itemize}
Just as before, we can apply a braid relation yielding $T_j^{\varepsilon_j^\prime} T_{j-1}^{\varepsilon_{j-1}^\prime} T_j$.  In the first and third cases $\varepsilon_{j-1}^\prime=\varepsilon_{j-1}$ and $\varepsilon_j^\prime=\varepsilon_j$. In the second case $\varepsilon_{j-1}^\prime=-\varepsilon_{j-1}$ and $\varepsilon_j^\prime = -\varepsilon_j$.  Again, the total number of indices $k$ for which $\varepsilon_k=-1$ is unchanged, so $\varepsilon_i^\tau=\varepsilon_i^{s_j \tau}$. In all cases, after the braid relation, we have
\begin{align*}
T_j Y_{s_j(i)}^\tau T_j^{-1} &= t^{\varepsilon_i^\tau}T_{i-1}^{\varepsilon_{i-1}} \cdots T_1^{\varepsilon_1} T_{j+1} \pi T_{n-1}^{\varepsilon_{n-1}} \cdots T_i^{\varepsilon_i} T_j^{-1} \\ 
&= t^{\varepsilon_i^\tau} T_{i-1}^{\varepsilon_{i-1}} \cdots T_1^{\varepsilon_1} \pi T_j T_{n-1}^{\varepsilon_{n-1}} \cdots T_i^{\varepsilon_i} T_j^{-1} \\ 
&= Y_i^{s_j \tau}.
\end{align*}

Now assume $j=i$. Because $(i,i+1)$ is not an inversion in $\tau$, we may write
\begin{align*}
T_iY_{i+1}^\tau T_i^{-1} &= T_i \left( t^{\varepsilon_{i+1}^\tau} T_{i}^{-1} \cdots T_1^{\varepsilon_1} \pi T_{n-1}^{\varepsilon_{n-1}} \cdots T_{i+1}^{\varepsilon_{i+1}} \right) T_i^{-1} \\
&=  t^{\varepsilon_{i+1}^\tau} T_{i-1}^{\varepsilon_{i-1}} \cdots T_1^{\varepsilon_1} \pi T_{n-1}^{\varepsilon_{n-1}} \cdots T_{i+1}^{\varepsilon_{i+1}} T_i^{-1} \\
&=Y_i^{s_i \tau}
\end{align*}
because $\varepsilon_{i+1}^\tau =  \varepsilon_{i}^{s_i \tau}$.

Similarly, assume $j=i-1$.  Because $j=i-1$ and $i$ are not an inversion in $\tau$ we may write
\begin{align*}
T_{i-1} Y_{i-1}^\tau T_{i-1}^{-1} &= T_{i-1}\left( t^{\varepsilon_{i-1}^\tau} T_{i-2}^{\varepsilon_{i-2}} \cdots T_1^{\varepsilon_1} \pi T_{n-1}^{\varepsilon_{n-1}} \cdots T_{i-1} \right) T_{i-1}^{-1} \\
&=  t^{\varepsilon_{i-1}^\tau} T_{i-1} T_{i-2}^{\varepsilon_{i-2}} \cdots T_1^{\varepsilon_1} \pi T_{n-1}^{\varepsilon_{n-1}} \cdots T_i^{\varepsilon_i} \\
&=Y_i^{s_i \tau}
\end{align*}
because $\varepsilon_{i-1}^\tau=\varepsilon_{i}^{s_i \tau}$.

Thus we have shown $Y_i^\tau = T_\tau Y_{\tau^{-1}(i)}T_\tau^{-1}$ for $\tau \in S_n$.

\end{proof}

Using Proposition \ref{prop:theaction}, Remark \ref{rem:howtouse}, and Lemma \ref{lem:TsandYs} we can give the main result of this chapter.

\begin{proposition}{\label{prop:permbaseeigen}}
The functions $E_{\gamma, \tau}$ are simultaneous eigenfunctions of the operators $Y_i^\tau$.  
\end{proposition}
\begin{proof}
We compute
\begin{align*}
Y_i^\tau E_{\gamma, \tau} &= Y_i^\tau (t^{-c}T_\tau E_{\gamma, \epsilon}) \\
&=  t^{-c}T_\tau Y_{\tau^{-1}(i)}T_\tau^{-1} (T_\tau E_{\gamma, \epsilon}) \\
&=\tilde{\gamma}_{\tau^{-1}(i)} E_{\gamma, \tau}
\end{align*}
where $\tilde{\gamma}_{\tau^{-1}(i)}$ is the eigenvalue of $Y_{\tau^{-1}(i)}$ acting on the nonsymmetric Macdonald polynomial $E_\gamma$, and $c$ is the constant $c=\sum \chi(\gamma_i \leq \gamma_j)$, where the sum is over all inversions $(i,j)$ of $\tau$.

\end{proof}

   \backmatter
   
   \bibliographystyle{plain}
   
   \def\cprime{$'$}

\end{document}